\documentclass[11pt]{article}

\usepackage{amsfonts}
\usepackage{amssymb}
\usepackage{amsmath}
\usepackage{a4wide}
\usepackage{natbib}
\usepackage{graphicx}
\usepackage{setspace}
\usepackage{amsthm}
\usepackage{bm}
\usepackage{hyperref}
\usepackage[symbol]{footmisc}

\usepackage[labelfont=bf]{caption}
\theoremstyle{plain}
\newtheorem{thm}{Theorem}[section] 
\newtheorem{exmp}[thm]{Example} 
\newtheorem{lem}[thm]{Lemma} 
\newtheorem{pro}[thm]{Proposition} 
\newtheorem{rem}[thm]{Remark} 
\newtheorem{co}[thm]{Corollary}
\newtheorem{ass}[thm]{Assumption}

\makeatletter
\newcommand{\vast}{\bBigg@{3.2}}
\newcommand{\Vast}{\bBigg@{5}}

\makeatother

\begin{document}
	\begin{center}
		\section*{A universal robustification procedure}
		\subsection*{Riccardo Passeggeri\footnote[2]{\noindent Email: riccardo.passeggeri@imperial.ac.uk}, Nancy Reid\footnote[3]{\noindent Email: nancym.reid@utoronto.ca}}
		\subsection*{{\normalfont \large\textit{$^\dag$Imperial College London\\ $^\ddag$University of Toronto}}}
		\today
	\end{center}
	\begin{abstract}
	We develop a procedure that transforms any asymptotically normal estimator into an asymptotically normal estimator whose distribution is robust to arbitrary data contamination. More generally, our procedure transforms any estimator whose asymptotic distribution has positive and continuous density at the origin into an asymptotically normal estimator whose distribution is robust to arbitrary contamination. In developing such a procedure we prove new general properties of componentwise and geometric quantiles in both finite and infinite dimensions. \\\\
		\textbf{Keywords:} quantile, robust estimator, continuity, consistency, asymptotic normality, orthogonality\\\\
	\end{abstract}
	\tableofcontents

	\section{Introduction}


	In this paper we consider a Quantile-of-Estimators (QoE) procedure that provides universal robustness, in the sense that given an initial estimator the procedure makes it asymptotically normally distributed even in presence of unspecified contamination of the observations. The procedure applies to any initial estimator which has an asymptotic distribution with a positive and continuous density at the origin, \textit{e.g.}~any asymptotically normal estimator. The behaviour of the contamination can be arbitrary, and can increase with $n$, at a rate which depends on the rate of convergence of the initial estimator to its asymptotic distribution in the absence of contamination. For example, if the initial estimator is the sample mean and the observations have finite third moment then the QoE estimator has an explicit asymptotic normal distribution in the presence of $o(n^{1/4})$ arbitrarily contaminated observations. This gives further motivation for studying rate-of-convergence results for estimators. The asymptotic distribution of the QoE estimator is always explicit, Gaussian, and unaffected by the contamination up to a certain growth rate.  
	
	In finite dimensions the contamination is more general version of \cite{Huber} contamination, which considers $X_1,...,X_n$ to be $n$ independent and identically distributed (i.i.d.) random variables drawn from a distribution $F=(1-\varepsilon)G+\varepsilon H$ where $G$ and $H$ are symmetric distributions and $\varepsilon>0$. In contrast we make no assumptions on the behaviour of the contaminated data, and do not assume any independence within the contaminated data, or between them and the non-contaminated data, thus allowing for malicious, also known as adversarial, contamination (see \cite{Loh}).

	
	We call the robustification `universal', as it applies to any estimator whose asymptotic distribution has positive and continuous density at the origin, and it is developed in both finite- and infinite-dimensional settings. The possibility of our method to apply to the infinite-dimensional spaces is remarkable given that robustness results in these spaces are few and limited to specific problems,  \textit{e.g.}~\cite{CCZ13,Giulini}. Our infinite-dimensional spaces include Banach spaces, which are crucial  for functional data analysis, the space of continuous paths, $C[0,T]$,  and the space of c\`{a}dl\`{a}g paths, $D[0,T]$. These latter spaces are the natural spaces for stochastic processes, where interest focusses on asymptotic Gaussianity. For example, estimation of  stochastic volatility usually relies on a Brownian motion limit for the so-called realized power variations, and its asymptotic variance (or a function of it).  The method applies in any setting where an initial estimator is available, and thus to both nonparametric and parametric modelling. 
	
	In developing our procedure we prove new properties of quantiles in both finite and infinite dimensions, which constitutes the second main contribution of the paper.
	
	
	There is a large literature on quantiles. Building on the work of \cite{Valadier}, \cite{Kemperman} showed the existence, uniqueness and weak continuity of the median in certain infinite dimensional spaces. \cite{Chaudhuri} provided a notion of the geometric quantile for multivariate data. \cite{Koltchinskii} developed and studied  M-quantile functions, including the geometric quantile. More recent developments are treated in \cite{Cadre2001}, \cite{Gervini2008},  \cite{ChCh2014}, \cite{HPS}, and \cite{KoPa}, among many others.

	More recently, motivated by statistical learning theory, several papers investigate the role of the median in providing efficient and robust non-asymptotic probability bounds for the estimation of the mean of a random variable. A first work in this direction is \cite{LeOl} which builds on the Median-of-Means (MoM) estimator, originally introduced in \cite{Jerrum,Nemirovski}. \cite{Minsker} develops a general method, based on the geometric median in Banach spaces, to obtain estimators with tight concentration bounds around the true parameter. \cite{Lugosi} introduces the Median-of-Means (MoM) tournaments, and presents an estimator that outperforms classical least-squares estimators when data are heavy-tailed. An analysis of MoM robustness properties is provided in \cite{Laforgue}.

	In Section \ref{Sec-Quantiles} we recall the definition of the geometric quantile in Banach spaces and introduce three definitions of component-wise quantile in infinite dimensions, depending which notion of \textit{component} we consider. While we focus on these two common definitions, our procedure can be carried out for any definition of quantiles, including the $\rho$-quantiles defined in \cite{KoPa} and those defined via optimal transport (\cite{Chernozhukov,GhSe,Hallin}). 
	
	
	Using these definitions we build our robustification procedure, the QoE estimator, in Section \ref{Sec-QoE}. This estimator is constructed by dividing the $n$ observations into $k$ blocks, computing the initial estimate on each block, and then computing the sample quantile of these estimates. We establish consistency and asymptotic normality of the QoE estimator, for both fixed $k$ and $k$ increasing with $n$, and in finite- and infinite-dimensional settings. Several examples are presented. In Section \ref{SubSec-QoE-quantile-regression} we establish the robustness properties of ordinary quantile estimators, and distinguish our approach from quantile regression.
	
	It is important to make a brief comparison between our work and the MoM and quantile regression literature. 
	
	Given the procedure, it is natural to call our estimator the QoE estimator. However, the approach of our paper is different from the one of the MoM literature. First, our procedure applies to a large class of estimators, including the sample mean. Second, our goal is to obtain asymptotically normal robust version of an initial estimator, while the goal of the MoM is to obtain efficient probability bounds for the estimation of the mean of a random variable. Third, we focus on both finite- and infinite-dimensions while MoM literature focuses on finite dimensions. Fourth, by considering quantiles, we can obtain results for any estimator whose asymptotic distribution has positive and continuous density at the origin, as we explain in Section \ref{Sec-QoE}.
	
	In quantile regressions the goal is to provide robust estimation in linear regression models. Our goal is to provide a robustification procedure for general estimators, including those for linear regression models. The quantile regression estimator is not robust to contamination, in particular one contaminated observation is enough to destroy it (see \cite{koenker_2005}), while our estimator is robust to contamination of order $n^\gamma$ for some $\gamma\in(0,1)$ depending on the setting -- in linear regression $\gamma=1/4$. Moreover, as in Huber's contamination model the asymptotic normality of the estimator is obtained only under certain restrictive conditions of the distribution of the contaminated data and it depends on that distribution. In our setting we allow for any contamination and the asymptotic normality of the QoE estimator does not depend on it. Recently in \cite{Loh},  the authors allow for adversarial contamination and focus on non-asymptotic $\ell_2$-error bounds, but even in this linear regression setting no asymptotic normality of the estimator is retained. Our procedure can also be applied to regression quantile estimators making them robust to contamination.
	
	
	The QoE estimator and its properties are presented in Section \ref{Sec-QoE}. These results are established using new results on the existence, uniqueness, continuity, and representations of  geometric and component-wise quantiles, developed in Section \ref{Sec-Continuity}. In particular, we provide necessary and sufficient conditions for the uniqueness of geometric quantiles in Banach spaces, improving the results of \cite{Kemperman}. This is based on a novel convex analysis result on James orthogonality, Lemma \ref{lem-ineq}. Further, we show that the geometric quantile function is a continuous function of the data. In particular, we show that if a sequence of empirical distributions converge weakly to an empirical distribution then their quantiles converge \textit{strongly}, \textit{i.e.}~in norm; we note that only weak$^*$-convergence has been established to date \cite{Kemperman,Cadre2001,ChCh2014}. By showing continuity of the quantile function we can then use the continuous mapping theorem to obtain asymptotic properties of the QoE estimator when the number of blocks $k$ is fixed.

	In Hilbert spaces we are able to give an explicit formulation for the geometric quantile and show that it is a weighted sum of the data and of the quantile parameter, where the weights are solutions of a minimization problem over a finite and compact set (see Theorem \ref{thm-extension-Gervini}). This extends the main intuition of \cite{Gervini2008} to quantiles and opens the possibility of extending his results for quantiles, including his algorithm for computing the sample spatial median in Hilbert spaces; something that was thought not to be possible (see \cite{ChCh2014} page 1208). In finite-dimensional space we show that if a sequence of distributions converge weakly to a distribution then their quantiles converge uniformly, Lemma \ref{lem-uniform-continuity}.

	\section{Quantiles}\label{Sec-Quantiles}
	Let $k\in\mathbb{N}$ and consider $\mathbf{x}\in\mathbb{R}^k$. Order its components in increasing values and denote by $\tilde{\mathbf{x}}$ the corresponding vector. For $\alpha\in(0,1)$, the univariate $\alpha$-quantile of $\mathbf{x}=(x_1,...,x_k)$ is the function $q_\alpha:\mathbb{R}^k\to\mathbb{R}$ defined as
	\begin{equation*}
		q_\alpha(\mathbf{x})=\frac{1}{2}(\tilde{x}_{\lceil k\alpha\rceil}+\tilde{x}_{\lfloor k\alpha+1 \rfloor});
	\end{equation*}
	where $\lfloor\cdot\rfloor$ and $\lceil\cdot\rceil$ are the floor and ceiling functions respectively. $q_{1/2}$ is the median of $(x_1,...,x_k)$. An alternative definition is the so-called geometric quantile, defined as
	\begin{equation}\label{univariate-geometric-quantile}
		{\displaystyle {\underset {y\in \mathbb {R}}{\operatorname {arg\,min} }}\sum _{i=1}^{k}|x_{i}-y|+u(x_i-y)},\quad\textnormal{where $u=2\alpha-1$.}
	\end{equation}
	The geometric quantile is not unique when $\alpha\in\{\frac{1}{k},\frac{2}{k},...,\frac{k-1}{k}\}$; in this case the set of solutions is the interval $[\tilde{x}_{ \alpha k},\tilde{x}_{ \alpha k+1 }]$.
	
	We can naturally extend the definition of the geometric quantile to a Banach space $\mathbb{X}$ with dual $\mathbb{X}^*$ (\cite{ChCh2014,Chaudhuri}). The pairing of a functional $\mathbf{f}\in\mathbb{X}^*$ and an element $\mathbf{x}\in\mathbb{X}$ is denoted by $\langle\mathbf{f},\mathbf{x}\rangle$. Given $\mathbf{u}\in \mathbb{X}^*$ with $\|\mathbf{u}\|_{\mathbb{X}^*}<1$, the geometric quantile of $\mathbf{x}_{1},...,\mathbf{x}_{k}\in\mathbb{X}$ is defined as
	\begin{equation}\label{geometric-quantile-Banach}
		{\displaystyle {\underset {\mathbf{y}\in \mathbb{X}}{\operatorname {arg\,min} }}\sum _{i=1}^{k}\left\|\mathbf{x}_{i}-\mathbf{y}\right\|_{\mathbb{X}}}+\langle\mathbf{u},\mathbf{x}_{i}-\mathbf{y}\rangle,
	\end{equation}
	equivalently
	\begin{equation}\label{geometric-quantile-Banach-without-x}
		{\displaystyle {\underset {\mathbf{y}\in \mathbb{X}}{\operatorname {arg\,min} }}\sum _{i=1}^{k}\left\|\mathbf{x}_{i}-\mathbf{y}\right\|_{\mathbb{X}}}-k\langle\mathbf{u} ,\mathbf{y}\rangle,
	\end{equation}
	
	As in the univariate case, the geometric quantile suffers from the lack of uniqueness, discussed further in Section \ref{Sec-Continuity}. To overcome this problem we proceed as follows. Let $\mathbb{X}^{k}:=\bigoplus_{i=1}^{k}\mathbb{X}$ be the $k$ times direct sum, \textit{i.e.}~the Cartesian product, of $\mathbb{X}$ with norm $\|(\mathbf{x}_1,...,\mathbf{x}_k)\|_{\mathbb{X}^{k}}:=\sum_{i=1}^{k}\|\mathbf{x}_i\|_{\mathbb{X}}$, which is itself a Banach space. Let $\mathbf{x}_1,...,\mathbf{x}_k\in\mathbb{X}$ lie on a straight line, which we call $W$. If the data are translated by $\mathbf{z}\in\mathbb{X}$ then (\ref{geometric-quantile-Banach}) is translated by $\mathbf{z}$, so wlog we assume that $W$ passes through the origin. If $\mathbb{X}$ is a Hilbert space, identify $\mathbb{X}^*$ with $\mathbb{X}$, as usual. Let $\mathbf{e}\in W$ be such that $\|\mathbf{e}\|_{\mathbb{X}}=1$ and denote by $W^{\perp}$ the orthogonal complement of $W$. Then \begin{equation*}
		\mathbf{u}=u\mathbf{e}+\mathbf{v},\quad\textnormal{where $u\in\mathbb{R}$ and $\mathbf{v}\in W^{\perp}$. }
	\end{equation*}
	Since $\|\mathbf{u}\|_{\mathbb{X}}<1$, we have that $u\in(-1,1)$. If $\mathbb{X}$ is a Banach space let $W^{\perp}$ be the annihilator of $W$, \textit{i.e.}
	\begin{equation*}
		W^{\perp}=\{\mathbf{v}\in\mathbb{X}^*:\langle\mathbf{v},\mathbf{s}\rangle=0,\,\,\forall\mathbf{s}\in W\}=\{\mathbf{v}\in\mathbb{X}^*:W\subseteq\textnormal{Ker}(\mathbf{v})\}.
	\end{equation*}
	By Theorem 5.110 in \cite{Hitchhiker2006} we have that $W^{\perp}$ has co-dimension 1, so 
	\begin{equation}\label{direct-sum-decomposition}
		\mathbb{X}^*=Z\oplus W^{\perp},
	\end{equation}
	where $Z$ is a one-dimensional subspace. Thus, we have that $\mathbf{u}=u\mathbf{e}+\mathbf{v}$ where $\mathbf{e}\in Z$ with $\|\mathbf{e}\|_{\mathbb{X}^{*}}=1$, $u\in\mathbb{R}$, and $\mathbf{v}\in W^{\perp}$. Since $\|\mathbf{u}\|_{\mathbb{X}^{*}}=\sup\limits_{\mathbf{z}\in\mathbb{X}}\frac{|\langle\mathbf{u},\mathbf{z}\rangle|}{\|\mathbf{z}\|_{\mathbb{X}}}$ and since $\|\mathbf{u}\|_{\mathbb{X}^{*}}<1$, we have that for any $\mathbf{s}\in W$
	\begin{equation*}
		\|\mathbf{s}\|_{\mathbb{X}}>\|\mathbf{s}\|_{\mathbb{X}}\|\mathbf{u}\|_{\mathbb{X}^{*}}\geq|\langle\mathbf{u},\mathbf{s}\rangle|=|\langle u\mathbf{e},\mathbf{s}\rangle|=|u||\langle \mathbf{e},\mathbf{s}\rangle|
	\end{equation*}
	and since $\|\mathbf{s}\|_{\mathbb{X}}=\|\mathbf{s}\|_{\mathbb{X}}\|\mathbf{e}\|_{\mathbb{X}^{*}}\geq|\langle \mathbf{e},\mathbf{s}\rangle|$ we conclude that $u\in(-1,1)$ and is a well-defined parameter of any geometric quantile in a Banach space. We define $\alpha:=(u+1)/2$.

	Consider the function $\mathbf{q}_{\mathbf{u}}:\mathbb{X}^{k}\to \mathbb{X}$ defined as follows. If $\|\mathbf{u}\|_{\mathbb{X}^*}\notin\{1-\frac{2j}{k},j=1,...,\lfloor\frac{k}{2}\rfloor\}$ let $\mathbf{q}_{\mathbf{u}}(\mathbf{x}_1,...,\mathbf{x}_k)$ be the geometric quantile (\ref{geometric-quantile-Banach}), which by Corollary \ref{co-unique} is uniquely defined. If $\|\mathbf{u}\|_{\mathbb{X}^*}\in\{1-\frac{2j}{k},j=1,...,\lfloor\frac{k}{2}\rfloor\}$, let
	\begin{equation}\label{QoE-geo}
		\mathbf{q}_{\mathbf{u}}(\mathbf{x}_1,...,\mathbf{x}_k):=\begin{cases}
			{\displaystyle {\underset {\mathbf{y}\in \mathbb{X}}{\operatorname {arg\,min} }}\sum _{i=1}^{k}\left\|\mathbf{x}_{i}-\mathbf{y}\right\|_{\mathbb{X}}+\langle \mathbf{u},\mathbf{x}_i-\mathbf{y}\rangle},&\textnormal{if $\mathbf{x}_1,...,\mathbf{x}_k$ do not lie on a straight line},\\q_{\alpha}(x_1,...,x_k)& \textnormal{otherwise,}
		\end{cases}
	\end{equation}
	where $x_i$ is such that $\mathbf{x}_i=x_i\mathbf{h}$ and $\mathbf{h}\in W$ with $\|\mathbf{h}\|_{\mathbb{X}}=1$, $i=1,...,k$. For the median we have $\|\mathbf{u}\|_{\mathbb{X}^*}=0$, so $\|\mathbf{u}\|_{\mathbb{X}^*}\notin\{1-\frac{2j}{k},j=1,...,\lfloor\frac{k}{2}\rfloor\}$ when $k$ is odd. 
	
	We define now component-wise quantiles in spaces with possibly infinite dimensions. Informally, each component of the component-wise quantile is the univariate quantile of the corresponding component of $\mathbf{x}_{1},...,\mathbf{x}_{k}$. This definition is not formal because the words \textit{component} might have different meanings, especially in infinite dimensional spaces. The existence of the component-wise quantiles we present in this section is proved in Section \ref{Sec-Continuity}.
	
	The most general definition of the component-wise quantile uses the Hamel basis. Let $\mathbb{X}$ be a vector space. Any $\mathbf{x}_{1},...,\mathbf{x}_{k}\in\mathbb{X}$ can be expressed as $\sum_{l\in I}x^{(l)}_1\mathbf{b}_{l},...,\sum_{l\in I}x^{(l)}_k\mathbf{b}_{l}$, respectively, where the $x^{(l)}_i$'s are real valued constants, only finitely many of which are non-zero, and $(\mathbf{b}_l)_{l\in I}$, for some index set $I$, is the Hamel basis of $\mathbb{X}$. The component-wise quantile with respect to the Hamel basis and with parameter $\bm{\alpha}=(\alpha_l)_{l\in I}$, where $\alpha_l\in(0,1)$ for every $l\in I$, is defined as 
	\begin{equation}\label{QoE-Hamel}
		\mathbf{q}_{Hamel,\bm{\alpha}}(\mathbf{x}_{1},...,\mathbf{x}_{k}):= \sum_{l\in I}\check{x}_l\mathbf{b}_{l},
	\end{equation}
	where $\check{x}_l=q_{\alpha_l}(x^{(l)}_1,...,x^{(l)}_k)$. The map that sends any $\mathbf{x}\in\mathbb{X}$ (rewritten using the Hamel basis as $\sum_{l\in I}x^{(l)}\mathbf{b}_{j}$) to $\sum_{l\in I}|x^{(l)}|$ is a norm on $\mathbb{X}$; we denote this by $\|\cdot\|_{\mathbb{X},Hamel}$. The map $(\mathbf{x}_1,...,\mathbf{x}_k)\mapsto\sum_{i=1}^{k}\|\mathbf{x}_i\|_{\mathbb{X},Hamel}$ defines a norm on the product space $\mathbb{X}^k$, which we denote $\|(\mathbf{x}_1,...,\mathbf{x}_k)\|_{\mathbb{X},Hamel,k}$.
	
	In practice, since the Hamel basis might not be explicitly available, it is easier to work with the Schauder basis. We define the component-wise quantile with respect to Schauder basis similarly. For any $\mathbf{x}_{1},...,\mathbf{x}_{k}$ belonging to some Banach space $\mathbb{X}$ possessing a Schauder basis $(\mathbf{d}_{l})_{l\in \mathbb{N}}$ we have $\mathbf{x}_{i}=\sum_{l\in \mathbb{N}}x^{(l)}_i\mathbf{d}_{l}$, where the $x^{(l)}_i$'s are real valued constants and $i=1,...,k$ (these $x^{(l)}_i$'s are generally not related to those in the Hamel basis decomposition). The component-wise quantile with respect to the Schauder basis, with parameter $\bm{\alpha}=(\alpha_l)_{l\in\mathbb{N}}$, where $\alpha_l\in(0,1)$ for every $l\in\mathbb{N}$, is defined as
	\begin{equation}\label{QoE-Schauder}
		\mathbf{q}_{S,\bm{\alpha}}(\mathbf{x}_{1},...,\mathbf{x}_{k}):= \sum_{l\in \mathbb{N}}\check{x}_l\mathbf{d}_{l},
	\end{equation}
	where $\check{x}_l=q_{\alpha_l}(x^{(l)}_1,...,x^{(l)}_k)$. A Schauder basis $(\mathbf{d}_{l})_{l\in \mathbb{N}}$ is said to be unconditional if whenever the series $\sum_{l\in \mathbb{N}}x_l\mathbf{d}_l$ converges, it converges unconditionally, \textit{i.e.}~all reorderings of the series converge to the same value. In this paper we concentrate on spaces that possess an unconditional Schauder basis. The standard bases of the sequence spaces $c_0$ and $\ell_p$ for $1 \leq p < \infty$, as well as every orthonormal basis in a Hilbert space, are unconditional. The Haar system and the Franklin system are unconditional bases in $L_p$ for any $1 < p < \infty$. Further, the Franklin system is an unconditional basis in all reflexive Orlicz spaces.
	
	There are some important spaces which have no unconditional Schauder basis, like $C[0,1]$, $L_1[0,1]$, and $L_\infty[0,1]$, among others. For these spaces and for function spaces in general we can employ the following definition of component-wise quantile. For the sake of clarity, we focus on function spaces $\mathbb{X}$ being either $C[0,1]$, or $D[0,1]$ or $L_p[0,1]$, with $1 \leq p \leq \infty$. Consider any $\mathbf{f}_1,...,\mathbf{f}_k\in \mathbb{X}$. Then, we can define the point-wise component-wise quantile with parameter $\bm{\alpha}=(\alpha_l)_{l\in[0,1]}$, where $\alpha_l\in(0,1)$ for every $l\in[0,1]$, as follows: for every $x\in[0,1]$
	\begin{equation}\label{QoE-point}
		\mathbf{q}_{P,\bm{\alpha}}(\mathbf{f}_{1},...,\mathbf{f}_{k})(x):=q_{\alpha_x}(\mathbf{f}_{1}(x),...,\mathbf{f}_{k}(x)).
	\end{equation}
	In this paper we consider $C[0,1]$ and $D[0,1]$ to be endowed with the uniform topology, \textit{i.e.}~the topology generated by the supremum norm.
	
	\begin{rem}\label{rem-definitions}
		The three definitions of component-wise quantile are equivalent when $\mathbb{X}$ is finite-dimensional.
	\end{rem} 
	On the other hand, in infinite-dimensional Banach spaces, not all function spaces have an unconditional Schauder basis, and by the Baire category theorem any Hamel basis is necessarily uncountable (see Corollary 5.23 in \cite{Hitchhiker2006}). 
	\section{Robustness of the QoE estimator}\label{Sec-QoE}
	In this section we present the main results of the paper. We present a procedure to make the asymptotic properties of a wide range of estimators robust to data contamination.
	
	Consider a sample of $n$ i.i.d.~observations. Divide these into $k$ blocks of equal size $\lfloor n/k\rfloor$, for some $k\in\mathbb{N}$. Let $Z_n$ be any estimator of a parameter of interest $\bm \theta_0\in\mathbb{X}$, for some space $\mathbb{X}$, and let $Z^{(i)}_{\lfloor n/k\rfloor}$ be the estimator from the observations in block $i$, $i=1,...,k$. We define the Quantile-of-Estimates (QoE) estimators $T_{\mathbf{u},n,k}:=\mathbf{q}_{\mathbf{u}}(Z^{(1)}_{\lfloor n/k\rfloor},...,Z^{(k)}_{\lfloor n/k\rfloor})$,  $T_{S,\bm{\alpha},n,k}:=\mathbf{q}_{S,\bm{\alpha}}(Z^{(1)}_{\lfloor n/k\rfloor},...,Z^{(k)}_{\lfloor n/k\rfloor})$, and similarly $T_{Hamel,\bm{\alpha},n,k}$ and $T_{P,\bm{\alpha},n,k}$ using (\ref{QoE-geo}), (\ref{QoE-Schauder}), (\ref{QoE-Hamel}), and (\ref{QoE-point}) respectively.
	
	By \textit{contaminated data} we mean that a portion of the sample is substituted by any arbitrary random variables. This framework, which is in line with Assumption 2 in \cite{Laforgue}, is much more general than Huber's contamination model. To distinguish this case from the uncontaminated case, we use the notation $\tilde{Z}^{(i)}_{\lfloor n/k\rfloor}$ for the sample estimator related to the block $i$, $i=1,...,k$, when $l_n$ observations are contaminated. By a slight abuse of notation we let $T_{\mathbf{u},n,k}:=\mathbf{q}_{\mathbf{u}}(\tilde{Z}^{(1)}_{\lfloor n/k\rfloor},...,\tilde{Z}^{(k)}_{\lfloor n/k\rfloor})$ and similarly $T_{S,\bm{\alpha},n,k}$, $T_{Hamel,\bm{\alpha},n,k}$ and $T_{P,\bm{\alpha},n,k}$.

	\begin{rem}\label{rem-uniquely-defined}
		The initial estimator $\tilde{Z}^{(i)}_{\lfloor n/k\rfloor}$ might not be uniquely defined due to contamination, for some $i=1,..,k$. This is for example the case of the least squares estimator when contamination creates multicollinearity in some blocks. In this case, we allow $\tilde{Z}^{(i)}_{\lfloor n/k\rfloor}$ to take any of the possible values.
	\end{rem}
	\begin{rem}
		All the proofs of the robustness results we consider the worst case scenario, namely that each contaminated data contaminate a different block and so for example if we have $100$ observations and $10$ blocks and $5$ contaminated data then $5$ blocks are contaminated. Therefore, all the robustness results hold under any permutation of the data.
	\end{rem}
	We consider two assumptions. In the first one the initial estimator $Z_n$ is consistent, while in the second one it has an asymptotic distribution.
	\begin{ass}[Consistency]\label{A1}
		We have that $Z_n\stackrel{p}{\to}\bm\theta_0$ as $n\to\infty$
	\end{ass}
	\begin{ass}[Asymptotic distribution]\label{A2}
		We have that $a_{n}(Z_n-\bm\theta_0)\stackrel{d}{\to}Y$ as $n\to\infty$, for some sequence $(a_{n})_{n\in\mathbb{N}}$ of non-negative numbers going to infinity and some random variable $Y$.
	\end{ass}
	We stress that these assumptions are on the behavior of  the initial estimator in the non-contaminated sample $Z_n$, even when we are considering statements about the asymptotic behavior of $T_{\mathbf{u},n,k}=\mathbf{q}_{\mathbf{u}}(\tilde{Z}^{(1)}_{\lfloor n/k\rfloor},...,\tilde{Z}^{(k)}_{\lfloor n/k\rfloor})$. For example, if the initial $n$ i.i.d.~observations have finite second moment and $Z_n$ is the sample mean, then Assumptions \ref{A1} and \ref{A2} are satisfied regardless of the contamination of the sample.

	\subsection{Component-wise quantile}\label{SubSec-QoE-component}
	First, we investigate the asymptotic behavior of the QoE's $T_{S,\bm{\alpha},n,k}$, $T_{Hamel,\bm{\alpha},n,k}$, and $T_{P,\bm{\alpha},n,k}$ as $n\to\infty$ for fixed $k\in\mathbb{N}$. The proofs of all our results are in the online Supplementary Material.
	\begin{thm}\label{t-1-q-component_1}
		Let $\mathbb{X}$ be a vector space endowed with the topology induced by $\|\cdot\|_{Hamel}$. Under Assumption \ref{A1}, $T_{Hamel,\bm{\alpha},n,k}\stackrel{p}{\to}\bm\theta_{0}$ as $n\to\infty$. Under Assumption \ref{A2}, $a_{\lfloor n/k\rfloor}(T_{Hamel,\bm{\alpha},n,k}-\bm\theta_{0})\stackrel{d}{\to} \mathbf{q}_{Hamel,\bm{\alpha}}(Y_1,...,Y_k)$ as $n\to\infty$, where $Y_i\stackrel{iid}{\sim}Y$, $i=1,...,k$.
	\end{thm}
	\begin{thm}\label{t-1-q-component_2}
		Let $\mathbb{X}$ be a Banach space possessing an unconditional Schauder basis. Under Assumption \ref{A1}, $T_{S,\bm{\alpha},n,k}\stackrel{p}{\to}\bm\theta_{0}$ as $n\to\infty$. Under Assumption \ref{A2}, $a_{\lfloor n/k\rfloor}(T_{S,\bm{\alpha},n,k}-\bm\theta_{0})\stackrel{d}{\to} \mathbf{q}_{S,\bm{\alpha}}(Y_1,...,Y_k)$ as $n\to\infty$, where $Y_i\stackrel{iid}{\sim}Y$, $i=1,...,k$.
	\end{thm}
	\begin{thm}\label{t-1-q-component_3}
		Let $\mathbb{X}$ be either $L_p[0,1]$ with $1 \leq p \leq \infty$ or $\mathbb{X}=C[0,1]$ or $\mathbb{X}=D[0,1]$. Under Assumption \ref{A1}, $T_{P,\bm{\alpha},n,k}\stackrel{p}{\to}\bm\theta_{0}$ as $n\to\infty$. Under Assumption \ref{A2}, $a_{\lfloor n/k\rfloor}(T_{P,\bm{\alpha},n,k}-\bm\theta_{0})\stackrel{d}{\to} \mathbf{q}_{P,\bm{\alpha}}(Y_1,...,Y_k)$ as $n\to\infty$, where $Y_i\stackrel{iid}{\sim}Y$, $i=1,...,k$.
	\end{thm}
	We now consider $k$ increasing with $n$. We focus first on the case $\mathbb{X}$ being finite dimensional. In this case the three definitions of component-wise quantiles coincide (see Remark \ref{rem-definitions}) and without loss of generality we use the notation $T_{P,\bm{\alpha},n,k_n}$. In the next result we investigate the asymptotic behavior in probability of $T_{P,\bm{\alpha},n,k_n}$ and its robustness to contaminated data, of size $l_n$.
	\begin{pro}\label{t-k-to-infinity-component}
		Let $\mathbb{X}=\mathbb{R}^d$, $d\in\mathbb{N}$. Under Assumption \ref{A1}, $T_{P,\bm{\alpha},n,k_n}\stackrel{p}{\to}\bm\theta_{0}$ as $n\to\infty$ for any $l_n$ and $k_n=o(n)$ such that $\lim\limits_{n\to\infty}l_n/k_n=0$.	
	\end{pro}
	
	Since we are in $\mathbb{R}^d$ we have $\bm\theta_{0}=(\theta_{0,1},...,\theta_{0,d})$ and similarly $Z^{(1)}_{\lfloor n/k\rfloor}=(Z^{(1)}_{\lfloor n/k\rfloor,1},...,Z^{(1)}_{\lfloor n/k\rfloor,d})$ and $Y=(Y_1,...,Y_d)$. Let $F_i$ and $f_i$ be the distribution and the density (if it exists) of $Y_i$, for $i=1,..,d$. Further, let $\Sigma_{\bm \alpha}$ be the $d\times d$ matrix with $(i,j)$ element
	\begin{equation*}
		\frac{\mathbb{P}\{Y_{i}> F_{i}^{-1}(\alpha_i) ,Y_{j}> F_{j}^{-1}(\alpha_j) \}-(1-\alpha_i)(1-\alpha_j)}{f_i(F_i^{-1}(\alpha_i))f_j(F_j^{-1}(\alpha_j))}.
	\end{equation*}
	For $i=1,..,d$ let $F_{n,i}$ be the cdf of $a_{\lfloor n/cn^\beta\rfloor}(Z^{(1)}_{\lfloor n/cn^\beta\rfloor,i}-\theta_{0,i})$ and let 
	\begin{equation}\label{beta^*}
		\beta_i^{*}=\sup\{\beta\in(0,1):\lim\limits_{n\to\infty} \sqrt{ n^\beta}(F_i(x)-F_{n,i}(x))=0,\textnormal{for $x=F_i^{-1}(\alpha_i)+\frac{y}{\sqrt{\lfloor n^\beta\rfloor}}$ and $y\in\mathbb{R}$}\}.
	\end{equation} 
	\begin{thm}\label{t-k-to-infinity-multivariate-robust}
		Let $\mathbb{X}=\mathbb{R}^d$, $d\in\mathbb{N}$. Under Assumption \ref{A2} if $f_i$ is continuous and strictly positive at $F_i^{-1}(\alpha_i)$, then
		\begin{equation*}
			a_{\lfloor n/cn^\beta\rfloor}(T_{P,\bm{\alpha},n,cn^\beta}-\bm\theta_{0})\stackrel{p}{\to} (F_1^{-1}(\alpha_1),...,F_d^{-1}(\alpha_d)),
		\end{equation*}
		as $n\to\infty$ for any $l_n=O(n^{\gamma})$, $c>0$, and $\beta\in(0,1)$ such that $\beta>\gamma$. Moreover, 
		\begin{equation*}
			\sqrt{cn^{\beta}}(a_{\lfloor n/cn^\beta\rfloor}(T_{P,\bm{\alpha},n,cn^\beta}-\bm\theta_{0})-(F_1^{-1}(\alpha_1),...,F_d^{-1}(\alpha_d)))\stackrel{d}{\to} N(\mathbf{0},\Sigma_{\bm \alpha})
		\end{equation*}
		as $n\to\infty$, for any $l_n=O(n^{\gamma})$, $c>0$, and $\beta\in(2\gamma,\max_{i=1,..,d}\beta_i^{*})$.
	\end{thm}
	
	\begin{rem}
		$\beta^*$ is a measure of how fast the asymptotic distribution of the initial estimator converges. The level of contamination $\gamma$ must be less than half of $\beta^*$. If this happens, then we can always find a growth rate for the number of blocks ($\beta$) which allows for this level of contamination. It is possible to see that we can always allow for an increasing number of contaminated data, even for the second statement of Theorem \ref{t-k-to-infinity-multivariate-robust}. This is because $\beta^*$ will always be strictly greater than zero simply because our initial estimator has an asymptotic distribution to which it is converging. The faster the initial estimator is converging, the more contaminated data we can allow for in our dataset.
	\end{rem}
	
	From the above result we have the following robustness property of the asymptotic distribution of the QoE estimator. Consider $Y$ to be such that $(F_1^{-1}(\alpha_1),...,F_d^{-1}(\alpha_d))=\mathbf{0}$, for some $\alpha_1,...,\alpha_d\in(0,1)$, and $a_{n}$ to be equal to $\sqrt{n}$. This happens in the typical case of asymptotically normal estimators, where the distribution of $Y$ is symmetric around zero, $(F_1^{-1}(1/2),...,F_d^{-1}(1/2))=\mathbf{0}$, and $a_{n}=\sqrt{n}$.
	\begin{co}\label{co-robust-component}
		Let $\mathbb{X}=\mathbb{R}^d$, $d\in\mathbb{N}$. Under Assumption \ref{A2} hold, if $(F_1^{-1}(\alpha_1),...,$ $F_d^{-1}(\alpha_d))=\mathbf{0}$, for some $\alpha_1,...,\alpha_d\in(0,1)$ and $f_i$ is continuous and strictly positive at $F_i^{-1}(\alpha_i)$, then
		\begin{equation*}
			\sqrt{cn^{\beta}}a_{\lfloor n/cn^\beta\rfloor}(T_{P,\bm{\alpha},n,cn^\beta}-\bm \theta_{0})\stackrel{d}{\to} N(\mathbf{0},\Sigma_{\bm \alpha})
		\end{equation*}
		as $n\to\infty$, and if $a_{n}=\sqrt{n}$ then
		\begin{equation*}
			\sqrt{n}(T_{P,\bm{\alpha},n,cn^\beta}-\bm \theta_{0})\stackrel{d}{\to} N(\mathbf{0},\Sigma_{\bm \alpha}),
		\end{equation*}
		for any $l_n=O(n^{\gamma})$, $c>0$, and $\beta\in(2\gamma,\max_{i=1,..,d}\beta_i^{*})$.
	\end{co}
	
	\begin{exmp}[Sample mean]
		Consider the case of estimating the mean of a population and the initial estimator is the sample mean, \textit{i.e.}~$Z_n$ is the sample mean. If observations have finite variance, then by the central limit theorem the distribution of $Y$ is normal, $a_{n}=\sqrt{n}$, and, if focusing on the median, $(F_1^{-1}(\alpha_1),...,$ $F_d^{-1}(\alpha_d))=\mathbf{0}$. Thus, Corollary \ref{co-robust-component} applies. In addition, if the observations have finite third moment then by the Berry-Esseen theorem $\max_{i=1,..,d}\beta_i^{*}=1/2$ and so the allowed number of contaminated data is $l_n=o(n^{1/4})$.
	\end{exmp}
	\begin{exmp}[$U$-statistics]
		Let $X_1,...,X_n$ be i.i.d.~random variables and consider the case of estimating $\mathbb{E}[h(X_1,...,X_r)]$ for some $r\in\mathbb{N}$ and some function $h$ using a $U$-statistic. Assume that its asymptotic distribution has a continuous and strictly positive density at $F_i^{-1}(\alpha_i)$ and that $(F_1^{-1}(\alpha_1),...,$ $F_d^{-1}(\alpha_d))=\mathbf{0}$. For example, this is the case when $h$ is symmetric. In this case we have asymptotic normality of the distribution of the $U$-statistic with $a_n=\sqrt{n}$. Then, by \cite{U-Statistics} the allowed number of contaminated data for the asymptotic distribution is $l_n=o(n^{1/4})$.
	\end{exmp}
	\begin{exmp}[Settings where Berry-Esseen type results apply]
		If Berry-Esseen type results are known then our results apply. Since the literature on Berry-Esseen type results is huge, we just mention three very recent examples: random forests and other supervised learning ensembles (\cite{Random-forest}), nonlinear statistics related to $M$-estimators and stochastic gradient descent algorithms (\cite{Shao}), and Efron’s bootstrap (\cite{Mayya-Zhilova}).
	\end{exmp}
	\begin{exmp}
		[MLE]  Our procedure can also be used in the case the initial estimator is the maximum likelihood estimator (MLE). This is an interesting example because contamination can easily destroy the asymptotic properties of the MLE. The growth rate of the contamination, $\gamma$, will then depend on the speed of convergence of the asymptotic distribution of the MLE in the non-contaminated setting.
	\end{exmp}
	\begin{exmp}
		[Linear regression models]  Consider the linear regression model $y_i=\mathbf{x}_i\bm{\beta}+u_i$, $i=1,...,n$, where $\mathbf{x}$ has dimension $p\in\mathbb{N}$, and assume as usual $\mathbb{E}[\mathbf{x}^{\top}u]=0$ and $\textnormal{rank}\,\, \mathbb{E}[\mathbf{x}^{\top}\mathbf{x}]=p$. Here we use small letters for $y$, $\mathbf{x}$, and $u$ even if they are random variables for consistency with the notation of the literature. For simplicity we focus on the homeoskedastic case, \textit{i.e.}~$\mathbb{E}[u^2\mathbf{x}^{\top}\mathbf{x}]=\sigma^{2}\mathbb{E}[\mathbf{x}^{\top}\mathbf{x}]$, where $\sigma^2:=\mathbb{E}[u^2]$, but the results hold without this assumption. Then, using the OLS estimator as initial estimator and focusing on the median (\textit{i.e.}~$\bm\alpha=\mathbf{0.5}=(0.5,...,0.5)$) we have that, as $n\to\infty$,
		\begin{equation*}
			\sqrt{n}(T_{P,\mathbf{0.5},n,cn^\beta}-\bm \beta)\stackrel{d}{\to} N(\mathbf{0},\Gamma),
		\end{equation*}
		where $\Gamma$ is a $p\times p$ matrix with $(i,j)$ element
		\begin{equation*} \label{eq-Gamma}
			2\pi\sqrt{\mathbb{E}[x_i^2]\mathbb{E}[x_j^2]}(\mathbb{P}\{Y_{i}> 0 ,Y_{j}> 0 \}-1/4)
		\end{equation*}
		and the $(i,i)$ element reduces to 
		\begin{equation*}
			\pi\sigma^2\mathbb{E}[x_i^2],
		\end{equation*}
		where $x_i$ is the $i$-th element of $\mathbf{x}$ and in \eqref{eq-Gamma}, the vector $(Y_1,...,Y_p)\sim N(0,\sigma^2\mathbb{E}[\mathbf{x}^{\top}\mathbf{x}])$. This convergence holds for any $\beta$ and $\gamma$ such that $\beta\in(2\gamma, \max_{i=1,..,d}\beta_i^{*})$.  If the data have finite third moment then by the Berry-Esseen theorem $\max_{i=1,..,d}\beta_i^{*}=1/2$, and the contaminated data $l_n$ccan grow at order $o(n^{1/4})$. This is a remarkable result given how easy it is to implement our procedure, and given that the quantile regression estimator is not robust to contamination. 
	\end{exmp}
	\begin{exmp}[Quantile regression] As we will discuss in Section \ref{SubSec-QoE-quantile-regression} the quantile regression estimator is not robust to contamination. However, we can robustify it, 
		as the Bahadur representation of the quantile regression estimator shows that it is a sample mean of i.i.d.~observations, and so by the Berry-Esseen theorem Corollary \ref{co-robust-component} applies with $l_n=o(n^{1/4})$. 
	\end{exmp}
	
	\begin{rem}
		In Huber's contamination model the contamination level is of order $l_n=\varepsilon n$ for some small $\varepsilon>0$. However, the distribution of the contaminated observations has to be symmetric and asymptotic normality of the $M$-estimator depends on that distribution. In contrast, in our setting we allow for any kind of contamination and almost any estimator, not just a subclass of $M$-estimators, and the asymptotic normality of the QoE estimator does not depend on the contamination.
	\end{rem}
	
	We conclude this section with a discussion on the asymptotic behavior of $T_{P,\bm{\alpha},n,cn^{\beta}}$ in the infinite-dimensional setting. From Theorem \ref{t-k-to-infinity-multivariate-robust} we have convergence in finite-dimensional distributions (fdd), see the first part of the next result. To obtain convergence in distribution in these spaces we need convergence in fdd and tightness. However, there are cases where just convergence in fdd implies convergence in distributions. We show this in the second part of the next result.
	
	Recall that $T_{P,\bm{\alpha},n,cn^{\beta}}$ is the point-wise component-wise quantile in functions spaces. Let $Y=(Y_t)_{t\in[0,1]}$ ,and let $F_t$ and $f_t$ be the distribution and the density (if exists) of $Y_t$. Further, let $\mathbf{F}^{-1}(\bm{\alpha}):=(F^{-1}_{t}(\alpha_{t}))_{t\in[0,1]}$ and let $B_{\bm\alpha}$ be a mean zero Gaussian process such that $(B_{t_1,\alpha_{t_1}},...,B_{t_r,\alpha_{t_r}})$ has covariance matrix
	\begin{equation*}
		\Bigg(\frac{\mathbb{P}(Y_{t_i}> F_{t_i}^{-1}(\alpha_{t_i}) ,Y_{t_j}> F_{t_j}^{-1}(\alpha_{t_j}) )-(1-\alpha_{t_i})(1-\alpha_{t_j})}{f_{t_i}(F_{t_i}^{-1}(\alpha_{t_i}))f_{t_j}(F_{t_j}^{-1}(\alpha_{t_j}))}\Bigg)_{i,j=1,...,r}.
	\end{equation*}
	and for any $t\in[0,T]$ we let $\beta_{t}^{*}$ be defined as $\beta^{*}_i$ in (\ref{beta^*}) with $F_t$ and $\alpha_{t}$ replacing $F_i$ and $\alpha_i$.
	\begin{pro}\label{pro-fdd}
		Let $\mathbb{X}$ be either $L_p[0,T]$ with $1 \leq p \leq \infty$ or $\mathbb{X}=C[0,T]$ or $\mathbb{X}=D[0,T]$. Let Assumption \ref{A2} hold. Assume that $f_{t}$ is continuous and strictly positive at $F_{t}^{-1}(\alpha_{t})$, for every $t\in[0,1]$. Then, 
		\begin{equation*}
			\sqrt{cn^{\beta}}(a_{\lfloor n/cn^\beta\rfloor}(T_{P,\bm{\alpha},n,cn^\beta}-\bm\theta_{0})-\mathbf{F}^{-1}(\bm{\alpha}))\stackrel{fdd}{\to} B_{\bm\alpha},
		\end{equation*}
		as $n\to\infty$, and if $a_{n}=\sqrt{n}$ and $F_{t}^{-1}(\alpha_{t})=0$ for every $t\in[0,1]$, then
		\begin{equation}\label{limit-fdd}
			\sqrt{n}(T_{P,\bm{\alpha},n,cn^\beta}-\bm\theta_{0})\stackrel{fdd}{\to} B_{\bm\alpha},
		\end{equation}
		for any $l_n=O(n^{\gamma})$, $c>0$, and $\beta\in(2\gamma,\sup_{t}\beta^{*}_{t})$. Moreover, if 
		\begin{equation*}
			\sup_{n\in\mathbb{N}}\mathbb{E}\left[\left\|\sqrt{cn^{\beta}}(a_{\lfloor n/cn^\beta\rfloor}(T_{P,\bm{\alpha},n,cn^\beta}-\bm\theta_{0})-\mathbf{F}^{-1}(\bm{\alpha}))\right\|^{p}_{L^{p}[0,T]}\right]<\infty
		\end{equation*}
		for some $p\in(1,\infty)$, then
		\begin{equation*}
			\sqrt{cn^{\beta}}(a_{\lfloor n/cn^\beta\rfloor}(T_{P,\bm{\alpha},n,cn^\beta}-\bm\theta_{0})-\mathbf{F}^{-1}(\bm{\alpha}))\stackrel{d}{\to} B_{\bm\alpha},
		\end{equation*}
		as $n\to\infty$ in $L^r[0,T]$, for any $r\in[1,p)$, $l_n=O(n^{\gamma})$, $c>0$, and $\beta\in(2\gamma,\sup_{t}\beta^{*}_{t})$.
	\end{pro}
	\begin{exmp}[Convergence to Brownian motion]
		In many situations in statistical analysis of stochastic processes, the estimator is converging toward a Brownian motion with rate $\sqrt{n}$. This is for example the case for the asymptotic convergence of the realized power variations in the estimation of stochastic volatility. So, the process $Y_t$ is a Brownian motion and when focusing on the median it is possible to see that $F_{t}^{-1}(\alpha_{t})=0$ for every $t\in[0,1]$. Moreover, the $i,j$ element of the covariance matrix of  $(B_{t_1,0},...,B_{t_r,0})$ is then given by
		\begin{equation*}
			\frac{\mathbb{P}(Y_{t_i}> 0 ,Y_{t_j}> 0 )-1/4}{f_{t_i}(0)f_{t_j}(0)}=2\pi \sqrt{t_i t_j}(\mathbb{P}(Y_{t_i}-Y_{t_j}+Y_{t_j}> 0 ,Y_{t_j}> 0 )-1/4)
		\end{equation*}
		and by the independence of the increments of the Brownian motion, the tower property of conditional expectations and assuming wlog that $t_i>t_j$ we get 
		\begin{equation*}
			\mathbb{P}(Y_{t_i}-Y_{t_j}+Y_{t_j}> 0 ,Y_{t_j}> 0 )=\mathbb{E}[\mathbb{E}[\mathbf{1}_{\{Y_{t_i}-Y_{t_j}+Y_{t_j}> 0\}} \mathbf{1}_{\{Y_{t_j}> 0\}}| Y_{t_j}]]
		\end{equation*}
		\begin{equation*}
			=\mathbb{E}\bigg[\bigg(1-\Phi\bigg(-\frac{Y_{1}\sqrt{t_j}}{\sqrt{t_i-t_j}}\bigg)\bigg)\mathbf{1}_{\{Y_{1}> 0\}}\bigg]=\frac{1}{2}-\int_{0}^{\infty}\Phi\bigg(-\frac{x\sqrt{t_j}}{\sqrt{t_i-t_j}}\bigg)\phi(x)dx
		\end{equation*}
		\begin{equation*}
			=\frac{1}{4}+\frac{1}{2\pi}\arctan\bigg(\frac{\sqrt{t_j}}{\sqrt{t_i-t_j}}\bigg)
		\end{equation*}
		where $\Phi$ and $\phi$ are the cdf and pdf of a standard normal distribution (hence of $Y_1$). Thus, we have 
		\begin{equation*}
			\frac{\mathbb{P}(Y_{t_i}> 0 ,Y_{t_j}> 0 )-1/4}{f_{t_i}(0)f_{t_j}(0)}= \sqrt{t_i t_j}\arctan\bigg(\frac{\sqrt{t_j}}{\sqrt{t_i-t_j}}\bigg).
		\end{equation*}
		Observe that when $t_i=t_j$ we obtain $\frac{t_i\pi}{2}$ as expected. The arguments of this example can be easily generalized for the general case of $\int_{0}^{t}\sigma(s)dB_s$, for some deterministic kernel $\sigma$ which usually represents the volatility component of the underlying process (see for example \cite{passeggeri_veraart_2019}). Proposition \ref{pro-fdd} applies also for the case of $\sigma$ being stochastic, but then more attention needs to be paid on the assumptions on $\sigma$.
	\end{exmp}
	
	\subsection{Geometric quantile}\label{SubSec-QoE-geo}
	In this subsection we show that the robustification procedure can be carried out using the geometric quantile as well. The first two results are the extension of Lemma 2.1 and Theorem 3.1 in \cite{Minsker} to the quantile case.
	\begin{lem}\label{Extension-of-Lemma2.1}
		Let $\mathbb{X}$ be a Banach space. Let $\mathbf{x}_1, . . . , \mathbf{x}_k \in \mathbb{X}$ and let $\mathbf{u}\in\mathbb{X}^*$ with $\|\mathbf{u}\|_{\mathbb{X}^*}<1$, and let $\mathbf{x}^*$ be a geometric quantile (which we assume exists). Fix $\nu\in(0, \frac{1-\|\mathbf{u}\|_{\mathbb{X}^*}}{2})$ and assume that $\mathbf{z}\in\mathbb{X}$ is such that $\| \mathbf{x}^*-\mathbf{z} \|_{\mathbb{X}}>C_{\nu} r$, where
		\begin{equation*}
			C_\nu=\dfrac{2(1-\nu)}{1-2\nu-\|\mathbf{u}\|_{\mathbb{X}^*}}
		\end{equation*}
		and $r>0$. Then there exists a subset $J\subset \{1,...,k\}$ of cardinality $|J|>\nu k$ such that for all $j\in J$, $\| \mathbf{x}_j-\mathbf{z} \|_{\mathbb{X}}>r$.
	\end{lem}
	For $0<p<s< \frac{1}{2}$, define
	\begin{equation*}
		\psi(s;p)=(1-s)\log\frac{1-s}{1-p}+s\log\frac{s}{p}.
	\end{equation*}
	\begin{pro}\label{Extension-Theorem3.1}
		Let $\mathbb{X}$ be a Banach space. Assume that $\mu\in\mathbb{X}$ is a parameter of interest, and let $\hat{\mu}_1,...,\hat{\mu}_k\in\mathbb{X}$ be a collection of independent estimators of $\mu$. Let $\mathbf{u}\in\mathbb{X}^*$ with $\|\mathbf{u}\|_{\mathbb{X}^*}<1$, and let $\hat{\mu}$ be a geometric quantile of $\hat{\mu}_1,...,\hat{\mu}_k$ (which we assume exists). Fix $\nu\in(0, \frac{1-\|\mathbf{u}\|_{\mathbb{X}^*}}{2})$. Let $0<p<\nu$ and $\varepsilon>0$ be such that for all $j=1,...,k$
		\begin{equation}\label{bob}
			\mathbb{P}(\|\hat{\mu}_j-\mu \|_{\mathbb{X}}>\varepsilon)\leq p.
		\end{equation}
		Then, 
		\begin{equation*}
			\mathbb{P}(\|\hat{\mu}-\mu \|_{\mathbb{X}}>C_{\nu}\varepsilon)\leq \exp(-k\psi(\nu;p)),
		\end{equation*}
		where $C_\nu$ is defined in Lemma \ref{Extension-of-Lemma2.1}. Further, if (\ref{bob}) only holds for $j\in J$, where $J\subset\{1,...,k\}$ with $|J|=(1-\tau)k$ for $0\leq\tau\leq \frac{\nu-p}{1-p}$, then
		\begin{equation*}
			\mathbb{P}(\|\hat{\mu}-\mu \|_{\mathbb{X}}>C_{\nu}\varepsilon)\leq \exp\bigg(-k(1-\tau)\psi\Big(\frac{\nu-\tau}{1-\tau};p\Big)\bigg).
		\end{equation*}
	\end{pro}
	We now investigate the asymptotic behavior of $T_{\mathbf{u},n,k}$, as $n\to\infty$ for fixed $k$.
	\begin{thm}\label{t-1-q}
		Let $\mathbb{X}$ be a reflexive and strictly convex Banach space. Under Assumption \ref{A1}, $T_{\mathbf{u},n,k}\stackrel{p}{\to}\bm\theta_0$ as $n\to\infty$. Let Assumption \ref{A2} hold. If either the distribution of $Y$ is nonatomic, or $\mathbb{X}$ is smooth and $\|\mathbf{u}\|_{\mathbb{X}^*}\notin\{1-\frac{2j}{k},j=1,...,\lfloor\frac{k}{2}\rfloor\}$, or $\mathbb{X}=\mathbb{R}$, then $a_{\lfloor n/k\rfloor}(T_{\mathbf{u},n,k}-\bm\theta_0)\stackrel{d}{\to} \mathbf{q}_\mathbf{u}(Y_1,...,Y_k)$ as $n\to\infty$, where $Y_i\stackrel{iid}{\sim}Y$, $i=1,...,k$.
	\end{thm}
	We explore now how the results of Theorem \ref{t-1-q} change when $k_n\to\infty$ as $n\to\infty$. First, we investigate the behavior of the convergence in probability in Theorem \ref{t-1-q} and its robustness in the presence of contaminated data.
	
	\begin{pro}\label{t-k-to-infinity}
		Let $\mathbb{X}$ be a reflexive Banach space and let Assumption \ref{A1} hold. Then, $T_{\mathbf{u},n,k_n}\stackrel{p}{\to}\bm\theta_0$ as $n\to\infty$ for any $l_n$ and $k_n$ such that $k_n=o(n)$ and $\lim\limits_{n\to\infty}\frac{l_n}{k_n}=0$ .
	\end{pro}
	In the following result we describe a remarkable new property of the geometric quantile: given a set of initial points and quantile parameter $\mathbf{u}$, if some of the points are modified then, it is always possible to adjust the quantile parameter $\mathbf{u}$ so that the geometric quantile does not change.
	\begin{lem}\label{Lemma-v}
		
		Let $\mathbb{X}$ be Banach space. Let $k\in\mathbb{N}$, $\mathbf{x}_1,...,\mathbf{x}_k\in\mathbb{X}$, and $\mathbf{u}\in \mathbb{X}^*$ with $\|\mathbf{u}\|_{\mathbb{X}^*}<1$. Assume that exists at least one geometric quantiles and denote it by $\mathbf{x}^*$. Let $\tilde{\mathbf{x}}_1,...,\tilde{\mathbf{x}}_k\in\mathbb{X}$ where $\tilde{\mathbf{x}}_i=\mathbf{x}_i$, for $i=p+1,...,k$ with $1\leq p<\frac{k}{2}(1-\|\mathbf{u}\|_{\mathbb{X}^*})$. Denote by $g_j^*$ a subdifferential of the norm $\|\cdot\|_{\mathbb{X}}$ evaluated at $\mathbf{x}_{j}-\mathbf{x}^*$. Then, 
		\begin{equation}\label{v-statement}
			\mathbf{v}:=\Bigg(\frac{2p}{2p+\sum _{i=1}^{k}\mathbf{1}_{\{\tilde{\mathbf{x}}_{i}=\mathbf{x}^*\}}}\Bigg)\Bigg(\frac{\mathbf{u}}{2p}\sum _{i=1}^{k}\mathbf{1}_{\{\tilde{\mathbf{x}}_{i}=\mathbf{x}^*\}}-\frac{1}{k}\sum _{j:\tilde{\mathbf{x}}_j\neq \mathbf{x}^*}^{k}g_j^*\Bigg)\mathbf{1}_{\{p\geq 1\}}+\mathbf{u}\mathbf{1}_{\{p=0\}}
		\end{equation}
		is such that $\mathbf{v}\in\mathbb{X}^*$, $\|\mathbf{v}\|_{\mathbb{X}^*}<1$, $\|\mathbf{v}-\mathbf{u}\|_{\mathbb{X}^*}\leq \frac{2p}{k}$, and $\mathbf{x}^*$ is a geometric quantile of $\tilde{\mathbf{x}}_1,...,\tilde{\mathbf{x}}_k\in\mathbb{X}$ with parameter $\mathbf{v}$.
	\end{lem}
	We explore now the central limit theorem corresponding to Proposition \ref{t-k-to-infinity}. First, we need to introduce some notation, which we partially borrow from \cite{ChCh2014}. Let $d_n\in\mathbb{N}$ and let $\mathbb{X}$ be a separable Hilbert space with orthonormal basis $\{\mathbf{e}_1,\mathbf{e}_2,...\}$. Let $\mathcal{Z}_n=span\{\mathbf{e}_1,...,\mathbf{e}_n\}$. For every $\mathbf{z}\in\mathbb{X}$, let $\mathbf{z}_{(n)}$ be the orthogonal projection of $\mathbf{z}$ onto $\mathcal{Z}_n$. Thus, for example if $Z^{(1)}_{\lfloor n/k_n\rfloor}$ has a nonatomic distribution, then $T_{\mathbf{u}_{(n)},n,k_n}$ denotes the minimizer of $	{\displaystyle {\underset {\mathbf{y}\in \mathcal{Z}_n}{\operatorname {arg\,min} }}\sum _{i=1}^{k}\left\| Z^{(i)}_{\lfloor n/k_n\rfloor,(n)}-\mathbf{y}\right\|_{\mathcal{Z}_n}}+\langle\mathbf{u}_{(n)} ,Z^{(i)}_{\lfloor n/k_n\rfloor,(n)}-\mathbf{y}\rangle$. Let $\mathbf{r}_{\mathbf{u}}$ denote the $\mathbf{u}$-quantile of the distribution of $Y$, \textit{i.e.}~$\mathbf{r}_\mathbf{u}$ is the $\mathbf{y}\in\mathbb{X}$ that solves $\mathbb{E}\Big[\frac{Y-\mathbf{y}}{\| Y-\mathbf{y}\|_{\mathbb{X}}}\Big]=\mathbf{u}$, and let $\mathbf{r}_{m,\mathbf{u}_{(n)}}$ denote the $\mathbf{u}_{(n)}$-quantile of the distribution of $a_{m}(Z_{m,(n)}-\bm\theta_{0,(n)})$.
	
	Consider a real-valued function $g(\mathbf{z})=\mathbb{E}[\|\mathbf{z}-Y\|-\|Y\|]-\langle \mathbf{u},\mathbf{z}\rangle$ and let $J_{\mathbf{z}}$ be its Hessian at $\mathbf{z}$. $J_{\mathbf{z}}$ is a symmetric bounded bilinear function from $\mathbb{X}\times\mathbb{X}$ into $\mathbb{R}$ satisfying
	\begin{equation*}
		\lim\limits_{t\to0}\frac{1}{t}\bigg|g(\mathbf{z}+t\mathbf{h})-g(\mathbf{z})-t\mathbb{E}\bigg[ \frac{\mathbf{z}-Y}{\| \mathbf{z}-Y\|_{\mathbb{X}}} -\mathbf{u}\bigg](\mathbf{h})-\frac{t^{2}}{2}J_{\mathbf{z}}(\mathbf{h},\mathbf{h})\bigg|=0,
	\end{equation*}
	for any $\mathbf{h}\in\mathbb{X}$. We let $\tilde{J}_{\mathbf{z}}:\mathbb{X}\to\mathbb{X}$ be the continuous linear operator defined by the equation $\langle \tilde{J}_{\mathbf{z}}(\mathbf{h}),\mathbf{v}\rangle=J_{\mathbf{z}}(\mathbf{h},\mathbf{v})$ for every $\mathbf{h},\mathbf{v}\in\mathbb{X}$. Similarly, we let $J_{\mathbf{z},n}$ be the Hessian of $\mathbb{E}[\|\mathbf{z}_{(n)}-Y_{(n)}\|-\|Y_{(n)}\|]-\langle \mathbf{u}_{(n)},\mathbf{z}_{(n)}\rangle$ and let $\tilde{J}_{\mathbf{z},n}$ be the associated linear operator. We define $W^{(j)}_{n,k_n,(n)}:=a_{\lfloor n/k_n\rfloor}(Z^{(j)}_{\lfloor n/k_n\rfloor,(n)}-\bm \theta_{0,(n)})$ and $\tilde{W}^{(j)}_{n,k_n,(n)}:=a_{\lfloor n/k_n\rfloor}(\tilde{Z}^{(j)}_{\lfloor n/k_n\rfloor,(n)}-\bm \theta_{0,(n)})$, for $j=1,...,k_n$.
	
	The following assumption naturally comes from Assumption (B) in \cite{ChCh2014}.
	\begin{ass}\label{A-B}
		Suppose that the probability measure of $Z_{n}$ is nonatomic and not entirely supported on a line in $\mathbb{X}$, for all $n$ sufficiently large. Further, assume that, for each $C>0$, $\sup\limits_{\mathbf{x}\in\mathcal{Z}_n,\|\mathbf{x}\|_{\mathcal{Z}_n}>C}\mathbb{E}[\|\mathbf{x}-W^{(1)}_{n,k_n,(n)}\|_{\mathcal{Z}_n}^{-2}]<\infty$ for all $n$ sufficiently large. 
	\end{ass}
	\begin{rem}
		In the finite dimensional case Assumption \ref{A-B} is satisfied when $W^{(1)}_{n,k_n,(n)}$ has a density which is bounded on any bounded set of $\mathbb{R}^d$ for all $n$ sufficiently large. Examples of general sufficient conditions for this to happen are in \cite{Sweeting}.
	\end{rem}
	\noindent Under Assumption \ref{A-B} since $Z_{n}$ is nonatomic for all $n$ large enough, $T_{\mathbf{u}_{(n)},n,k_n}$ is given by the geometric quantile with parameter $\mathbf{u}$ of $Z^{(1)}_{\lfloor n/k\rfloor},...,Z^{(k)}_{\lfloor n/k\rfloor}$ for all $n$ large enough. The next result generalizes Proposition 2.1 in \cite{CCZ13}.
	
	\begin{lem}\label{lem-CCZ}
		Let $\mathbb{X}$ be a separable Hilbert space. Under Assumptions \ref{A2} and \ref{A-B}, for each $A>0$ there exist $c_A,C_A\in(0,\infty)$, such that for all $n$ large enough we have $c_A\|\mathbf{h}\|^2\leq J_{n,\mathbf{z}}(\mathbf{h},\mathbf{h})\leq C_A\|\mathbf{h}\|^2$ for any $\mathbf{z},\mathbf{h}\in\mathcal{Z}_n$ with $\|\mathbf{z}\|\leq A$.
	\end{lem}
	
	We are now ready to present the Bahadur-type asymptotic linear representation for $a_{\lfloor n/k_n\rfloor}(T_{\mathbf{u}_{(n)},n,k_n} -\bm\theta_{0,(n)}) $.
	\begin{thm}\label{thm-Bahadur}
		Let $\mathbb{X}$ be a separable Hilbert space. Let Assumptions \ref{A2} and \ref{A-B} hold. Then, 
		\begin{equation*}
			a_{\lfloor n/k_n\rfloor}(T_{\mathbf{u}_{(n)},n,k_n} -\bm\theta_{0,(n)}) -\mathbf{r}_{\lfloor n/k_n\rfloor,\mathbb{E}[U_n]}
		\end{equation*} 
		\begin{equation}
			=-\frac{1}{k_n}\sum_{i=1}^{k_n} [\tilde{J}_{\mathbf{r}_{\lfloor n/k_n\rfloor,\mathbb{E}[U_n]},n}]^{-1}\bigg(\frac{\mathbf{r}_{\lfloor n/k_n\rfloor,\mathbb{E}[U_n]}-W^{(i)}_{n,k_n,(n)} }{\|\mathbf{r}_{\lfloor n/k_n\rfloor,\mathbb{E}[U_n]}-W^{(i)}_{n,k_n,(n)} \|_{\mathbb{X}}}-U_n   \bigg)+R_n,
		\end{equation}
		for any $l_n$, $k_n$ and $d_n$ such that $k_n=o(n)$, $\lim\limits_{n\to\infty}\frac{l_n}{k_n}=0$ and $\lim\limits_{n\to\infty}\frac{d_n}{k_n^{1-2\rho}}=c$ for some $\rho\in(0,1/2]$ and some $c>0$, where $R_n=O((\log k_n)/k_n^{2\rho})$ as $n\to\infty$ almost surely and
		\begin{equation*}
			U_n=\Bigg(\frac{2l_n}{2l_n+\sum _{i=1}^{k_n}\mathbf{1}_{\{\tilde{Z}^{(i)}_{\lfloor n/k_n\rfloor,(n)}=T_{\mathbf{u}_{(n)},n,k_n} \}}}\Bigg)\Bigg(\frac{\mathbf{u}}{2l_n}\sum _{i=1}^{k_n}\mathbf{1}_{\{\tilde{Z}^{(i)}_{\lfloor n/k_n\rfloor,(n)}=T_{\mathbf{u}_{(n)},n,k_n} \}}
		\end{equation*}
		\begin{equation*}
			-\frac{1}{k_n}\sum _{j:\tilde{Z}^{(j)}_{\lfloor n/k_n\rfloor,(n)}\neq T_{\mathbf{u}_{(n)},n,k_n} }^{k_n}\frac{\tilde{Z}^{(i)}_{\lfloor n/k_n\rfloor,(n)}-T_{\mathbf{u}_{(n)},n,k_n}}{\|\tilde{Z}^{(i)}_{\lfloor n/k_n\rfloor,(n)}-T_{\mathbf{u}_{(n)},n,k_n}\|}\Bigg)\mathbf{1}_{\{l_n\geq 1\}}+\mathbf{u}\mathbf{1}_{\{l_n=0\}}
		\end{equation*}
		is a $\mathcal{Z}_n$-valued random variable with $\|U_n-\mathbf{u}_{(n)}\|<\frac{2l_n}{k_n}$.
	\end{thm}
	Compare the formulation of $U_n$ with (\ref{v-statement}). Further, observe that if $k_n$ and $d_n$ are such that $\rho\in(1/4,1/2]$ we obtain asymptotic normality of $\sqrt{k_n}(a_{\lfloor n/k_n\rfloor}(T_{\mathbf{u}_{(n)},n,k_n} -\bm\theta_{0,(n)}) -\mathbf{r}_{\lfloor n/k_n\rfloor,\mathbb{E}[U_n]})$. By setting $d_n=d$ and $\rho=1/2$, for some $d\in\mathbb{N}$, we obtain the result for the finite-dimensional case of dimension $d$.
	\begin{thm}\label{thm-Bahadur-2}
		Let $\mathbb{X}$ be a separable Hilbert space. Let Assumptions \ref{A2} and \ref{A-B} hold. Let $l_n$, $k_n$ and $d_n$ be such that $k_n=o(n)$, $\lim\limits_{n\to\infty}\frac{l_n}{k_n}=0$ and $\lim\limits_{n\to\infty}\frac{d_n}{k_n^{1-2\rho}}=c$ for some $\rho\in(1/4,1/2]$ and some $c>0$. Assume that $\sqrt{k_n}\|\mathbf{r}_\mathbf{u}-\mathbf{r}_{\lfloor n/k_n\rfloor,\mathbb{E}[U_n]}\|\to0$ and that $\|\tilde{J}_{\mathbf{r}_{\lfloor n/k_n\rfloor,\mathbf{u}_{(n)} }}-\tilde{J}_{\mathbf{r}_{\mathbf{u}}}\|\to0$ as $n\to\infty$. Then, 
		\begin{equation}
			\sqrt{k_n}(a_{\lfloor n/k_n\rfloor}(T_{\mathbf{u}_{(n)},n,k_n} -\bm\theta_{0,(n)}) -\mathbf{r}_{\mathbf{u}})\stackrel{d}{\to}G_{\mathbf{u}},
		\end{equation}
		where $G_{\mathbf{u}}$ is a $\mathbb{X}$-valued mean zero Gaussian random variable with covariance given by $V_{\mathbf{u}}=[\tilde{J}_{\mathbf{r}_{\mathbf{u}}}]^{-1}\Lambda_\mathbf{u}[\tilde{J}_{\mathbf{r}_{\mathbf{u}}}]^{-1}$, where $ \Lambda_\mathbf{u}:\mathbb{X}\to\mathbb{X}$ satisfies $\langle \Lambda_\mathbf{u}(\mathbf{h}),\mathbf{s}\rangle=\mathbb{E}\big[\langle \frac{\mathbf{r}_{\mathbf{u}}-Y}{\| \mathbf{r}_{\mathbf{u}}-Y\|_{\mathbb{X}}} -\mathbf{u},\mathbf{h}\rangle \langle \frac{\mathbf{r}_{\mathbf{u}}-Y}{\| \mathbf{r}_{\mathbf{u}}-Y\|_{\mathbb{X}}} -\mathbf{u},\mathbf{s}\rangle\big]$, for every $\mathbf{h},\mathbf{s}\in\mathbb{X}$.
	\end{thm}
	Assuming that $\sqrt{k_n}\|\mathbf{r}_\mathbf{u}-\mathbf{r}_{\lfloor n/k_n\rfloor,\mathbb{E}[U_n]}\|\to0$ implies assuming that the quantile of the distribution converge to the quantile of the limiting distribution fast enough (see Theorem 3.4 in \cite{ChCh2014}). Since the distribution of $Z_{n}$ is not concentrated on a straight line we know that $\Lambda_{\mathbf{u},n}:\mathcal{Z}_n\to\mathcal{Z}_n$ given by $\langle \Lambda_{\mathbf{u},n}(\mathbf{h}),\mathbf{s}\rangle=\mathbb{E}\big[\langle \frac{\mathbf{r}_{\mathbf{u}_{(n)}}-W^{(1)}_{n,k_n,\mathbf{u}_{(n)}} }{\| \mathbf{r}_{\mathbf{u}_{(n)}}-W^{(1)}_{n,k_n,\mathbf{u}_{(n)}}\|_{\mathbb{X}}} -\mathbf{u}_{(n)},\mathbf{h}\rangle \langle \frac{\mathbf{r}_{\mathbf{u}_{(n)}}-W^{(1)}_{n,k_n,\mathbf{u}_{(n)}}}{\| \mathbf{r}_{\mathbf{u}_{(n)}}-W^{(1)}_{n,k_n,\mathbf{u}_{(n)}}\|_{\mathbb{X}}} -\mathbf{u}_{(n)},\mathbf{s}\rangle\big]$ is invertible for $n$ large enough. Thus, we can get rid of the assumption $\|\tilde{J}_{\mathbf{r}_{\lfloor n/k_n\rfloor,\mathbf{u}_{(n)} }}-\tilde{J}_{\mathbf{r}_{\mathbf{u}}}\|\to0$ as $n\to\infty$ at the expense of clarity. In this case, the result in Theorem \ref{thm-Bahadur-2} would be
	\begin{equation*}
		\tilde{J}_{\mathbf{r}_{\lfloor n/k_n\rfloor,\mathbf{u}_{(n)}}} \Lambda_{\mathbf{u},n}^{-1/2}\sqrt{k_n}(a_{\lfloor n/k_n\rfloor}(T_{\mathbf{u}_{(n)},n,k_n} -\bm\theta_{0,(n)}) -\mathbf{r}_{\mathbf{u}})\stackrel{d}{\to}S_{\mathbf{u}},
	\end{equation*}
	where $S_{\mathbf{u}}$ is a $\mathbb{X}$-valued mean zero Gaussian random variable with covariance given by the identity operator.
	
	From this we have the following robustness property of the QoE estimator. Consider $Y$ to be such that $\mathbf{r}_\mathbf{u}=\mathbf{0}$, for some $\mathbf{u}$, and $a_{n}=\sqrt{n}$, as in the typical case of asymptotically normal estimators in $\mathbb{R}^d$. In this case $Y$ is symmetrically Gaussian around zero, hence $\mathbf{r}_\mathbf{0}=\mathbf{0}$, and $a_{n}=\sqrt{n}$.
	\begin{co}\label{co-robust-geo-infinite}
		Under the assumptions of Theorem \ref{thm-Bahadur-2}, with $Y$ such that $\mathbf{r}_\mathbf{u}=\mathbf{0}$ for some $\mathbf{u}$, we have
		\begin{equation}
			\sqrt{k_n}a_{\lfloor n/k_n\rfloor}(T_{\mathbf{u}_{(n)},n,k_n} -\bm\theta_{0,(n)}) \stackrel{d}{\to}Z,
		\end{equation}
		and if $a_{n}=\sqrt{n}$, then
		\begin{equation}
			\sqrt{n}(T_{\mathbf{u}_{(n)},n,k_n} -\bm\theta_{0,(n)}) \stackrel{d}{\to}Z,
		\end{equation}
		where $Z$ is as in Theorem \ref{thm-Bahadur-2}.
	\end{co}
	Thus, we have asymptotic normality of the QoE estimator in the presence of contaminated data $l_n=o(n)$. By proving the results for quantiles instead of just for the median, we allow the results to be applied to any estimator whose asymptotic distribution is strictly positive on a neighborhood of $\mathbf{0}$.
	
	In the finite-dimensional setting some of the assumptions of Theorem \ref{thm-Bahadur-2} are always satisfied. By Lemma \ref{lem-uniform-continuity},  under Assumption \ref{A2}, $\sup\limits_{\mathbf{z}\in\mathbb{R}^d:\|\mathbf{z}-\mathbf{u}\|<l_n/k_n}|\mathbf{r}_{\lfloor n/k_n\rfloor,\mathbf{z}}-\mathbf{r}_{\mathbf{u}}|\to0$ for any sequences $l_n$ and $k_n=o(n)$ such that $\lim\limits_{n\to\infty}\frac{l_n}{k_n}=0$.  Let $\varepsilon>0$ and let $k_n=cn^{\beta}$ for $\beta\in(0,1)$. Define 
	\begin{equation*}
		\beta_{\varepsilon}^*:=\sup\{\beta\in(0,1):\sqrt{n^{\beta}}\sup\limits_{\|\mathbf{z}-\mathbf{u}\|<n^{-\varepsilon/2}}|\mathbf{r}_{\lfloor n/cn^{\beta}\rfloor,\mathbf{z}}-\mathbf{r}_{\mathbf{u}}|\to0\}.
	\end{equation*}
	Theorem \ref{thm-Bahadur-2} on $\mathbb{R}^d$ now simplifies as follows. 
	\begin{pro}\label{co-robust-geo-finite}
		Let $\mathbb{X}=\mathbb{R}^d$ and let Assumptions \ref{A2} and \ref{A-B} hold. Then, 
		\begin{equation*}
			\tilde{J}_{\mathbf{r}_{\lfloor n/k_n\rfloor,\mathbf{u}}} \Lambda_{\mathbf{u},n}^{-1/2}	\sqrt{cn^{\beta}}(a_{\lfloor n/cn^{\beta}\rfloor}(T_{\mathbf{u},n,cn^{\beta}} -\bm\theta_{0}) -\mathbf{r}_{\mathbf{u}})\stackrel{d}{\to}N(\mathbf{0},\mathbf{I})
		\end{equation*} for any $l_n=O(n^{\gamma})$, $\beta\in(\gamma+\varepsilon,\beta^*_\varepsilon)$, and $\varepsilon>0$, where $\mathbf{I}$ is the identity matrix. If $Y$ is such that $\mathbf{r}_\mathbf{u}=\mathbf{0}$ and $a_{n}=\sqrt{n}$, then 
		\begin{equation*}
			\tilde{J}_{\mathbf{r}_{\lfloor n/k_n\rfloor,\mathbf{u}}} \Lambda_{\mathbf{u},n}^{-1/2}	\sqrt{n}(T_{\mathbf{u},n,cn^{\beta}} -\bm\theta_{0}) \stackrel{d}{\to}N(\mathbf{0},\mathbf{I})
		\end{equation*} for any $l_n=O(n^{\gamma})$, $\beta\in(\gamma+\varepsilon,\beta^*_\varepsilon)$, and $\varepsilon>0$.
	\end{pro}
	We remark that the interval $(\gamma+\varepsilon,\beta^*_\varepsilon)$ might be improved. Indeed, if we knew that the set of $\beta\in(\gamma,\beta^{*}_\gamma)$ such that $\sqrt{n^{\beta}}\sup\limits_{\|\mathbf{z}-\mathbf{u}\|<n^{\gamma-\beta}}|\mathbf{r}_{\lfloor n/cn^{\beta}\rfloor,\mathbf{z}}-\mathbf{r}_{\mathbf{u}}|\to0$ is a connected interval, then the convergences in Proposition \ref{co-robust-geo-finite} would hold for any $l_n=O(n^{\gamma})$ and $\beta\in(\gamma,\beta^*_\gamma)$.

	\subsection{Robustness results for sample quantiles}\label{SubSec-QoE-quantile-regression}
	In this section we present a robustness result for the asymptotic distributions of quantiles. This result does not require the division of the observations into blocks. We focus on the one-dimensional case for the sake of brevity and because the extension of this result to other spaces is straightforward given the results presented so far. Informally, the next result can be summarized as follows: \textit{The asymptotic distribution of the sample quantile is robust to the contamination of $o(n^{1/2})$ many data.}
	
	\begin{pro}\label{pro-sample-quantile}
		Let $\mathbb{X}=\mathbb{R}$ and $\alpha\in(0,1)$. Consider $n$ i.i.d.~observations with distribution $F$ and density $f$ such that $f$ is strictly positive and continuous at $F^{-1}(\alpha)$. Let $l_n$ observations be contaminated and let $\tilde{q}_{\alpha}$ be the sample quantile with parameter $\alpha$ of the contaminated sample. If $l_n=o(n^{1/2})$, then $\sqrt{n}(q_\alpha-F^{-1}(\alpha)) \stackrel{d}{\to}N\Big(0,\frac{\alpha(1-\alpha)}{f^{2}(F^{-1}(\alpha))}\Big)$ as $n\to\infty$.
	\end{pro}
	
	We remark that it is not possible to obtain such a result for quantile regressions because the quantile regression estimator has a breakdown point of $1/n$ in the $(x,y)$-contamination and only marginally better in the fixed $x$-, $y$-contamination, see page 270 in \cite{koenker_2005}. Thus, for the quantile regression estimator consistency is not preserved under contamination, and so we cannot have robustness of its asymptotic distribution.

	\section{Existence, uniqueness, continuity, and representations}\label{Sec-Continuity}
	In this section we present various results concerning continuity, existence, and uniqueness properties of both geometric and component-wise quantiles. We present also a representation of geometric quantiles in Hilbert spaces that extend to quantiles the representation of the median in \cite{Gervini2008}. These results provide the building blocks for the proofs of the results presented in the previous section. We start with a short discussion of univariate quantiles.
	\begin{lem}\label{lem-uniq-defined}
		The univariate geometric quantile is uniquely defined if and only if $\alpha\in(0,1)\setminus\{\frac{1}{k},\frac{2}{k},...,\frac{k-1}{k}\}$. If the univariate geometric quantile is uniquely defined then it is equal to $q_\alpha$.
	\end{lem}
	An alternative definition of univariate quantile generalizes the definition of the univariate median adopted in \cite{Lugosi} (see also \cite{Robust-machine}). Define $q^{\circ}_\alpha(\mathbf{x}):=x_i$, where $x_i$ is such that
	\begin{equation*}
		|\{j\in\{1,...,k\}:x_j\leq x_i\}|\geq k\alpha\quad\textnormal{and}\quad |\{j\in\{1,...,k\}:x_j\geq x_i\}|\geq k(1-\alpha).
	\end{equation*}
	If several $x_i$'s satisfy the constraints we take the smallest one by convention.
	\begin{lem}\label{lem-con-quant}
		Let $\alpha\in(0,1)$. The functions $q_\alpha$ and $q^{\circ}_\alpha$ are Lipschitz continuous. In particular for any $\mathbf{x},\mathbf{y}\in\mathbb{R}^k$ we have
		\begin{equation}
			|q_\alpha(\mathbf{x})-q_\alpha(\mathbf{y})|\leq\max_{j=1,...,k}|x_j-y_j|,\quad\textnormal{and}\quad|q^{\circ}_\alpha(\mathbf{x})-q^{\circ}_\alpha(\mathbf{y})|\leq\max_{j=1,...,k}|x_j-y_j|.
		\end{equation}
	\end{lem}
	The above bounds can be improved as it is possible to see from the proof of Lemma \ref{lem-con-quant}. Further, we remark that all our results containing $q_{\alpha}$ hold also if we substitute $q_{\alpha}$ by $q^{\circ}_\alpha$.
	\subsection{Component-wise quantile}\label{SubSec-Continuity-component}
	We present existence and continuity results for the three definitions of component-wise quantiles.
	\begin{thm}\label{thm-continuity-component-Hamel}
		Let $\mathbb{X}$ be a vector space. The component-wise quantile $\mathbf{q}_{Hamel,\bm{\alpha}}$ exists and is unique. If $\mathbb{X}$ is endowed with the topology induced by $\|\cdot\|_{Hamel}$, $\mathbf{q}_{Hamel,\bm{\alpha}}$ is Lipschitz continuous: for every $\mathbf{x}_1,...,\mathbf{x}_k,\mathbf{z}_1,...,\mathbf{z}_k\in\mathbb{X}$ we have
		\begin{equation*}
			\|\mathbf{q}_{Hamel,\bm{\alpha}}(\mathbf{x}_{1},...,\mathbf{x}_{k})-\mathbf{q}_{Hamel,\bm{\alpha}}(\mathbf{z}_{1},...,\mathbf{z}_{k})\|_{\mathbb{X},Hamel}\leq\|(\mathbf{x}_1,...,\mathbf{x}_k)-(\mathbf{z}_1,...,\mathbf{z}_k)\|_{\mathbb{X},Hamel,k}.
		\end{equation*}
	\end{thm}

	Now, let $\mathbb{X}$ be a Banach space possessing an unconditional Schauder basis. Denote by $\mathcal{K}$ the unconditional basis constant of $\mathbb{X}$, \textit{i.e.}
	\begin{equation*}
		\mathcal{K}:=\sup\limits_{F\subset\mathbb{N},\epsilon=\pm1,\|\mathbf{x}\|_{\mathbb{X}}=1}\|\sum_{n\in F}\epsilon_{l}x_{l}\mathbf{d}_{l}\|_{\mathbb{X}}
	\end{equation*}
	Since the field of $\mathbb{X}$ is $\mathbb{R}$, we have that 
	\begin{equation*}
		\mathcal{K}=\sup\limits_{F\subset\mathbb{N},|\lambda_{l}|\leq 1,\|\mathbf{x}\|_{\mathbb{X}}=1}\|\sum_{n\in F}\lambda_{l}x_{l}\mathbf{d}_{l}\|_{\mathbb{X}}.
	\end{equation*}
	By Theorem 6.4 in \cite{Heil} $\mathcal{K}$ is always finite. Moreover, it is always possible to assign to $\mathbb{X}$ an equivalent norm such that $\mathcal{K}=1$. The unconditional basis constant is the focus of a vast literature, see \textit{e.g.}~\cite{Gordon-Lewis} and \cite{GowersJAMS}.
	\begin{thm}\label{thm-continuity-component-Scahuder}
		Let $\mathbb{X}$ be a Banach space possessing an unconditional Schauder basis and let $\mathcal{K}$ be its unconditional basis constant. The component-wise quantile $\mathbf{q}_{S,\bm{\alpha}}$ exists, is unique, and is Lipschitz continuous: for every $\mathbf{x}_1,...,\mathbf{x}_k,$ $\mathbf{z}_1,...,\mathbf{z}_k\in\mathbb{X}$ we have that
		\begin{equation*}
			\|\mathbf{q}_{S,\bm{\alpha}}(\mathbf{x}_{1},...,\mathbf{x}_{k})-\mathbf{q}_{S,\bm{\alpha}}(\mathbf{z}_{1},...,\mathbf{z}_{k})\|_{\mathbb{X}}\leq\mathcal{K}^{2}\|(\mathbf{x}_1,...,\mathbf{x}_k)-(\mathbf{z}_1,...,\mathbf{z}_k)\|_{\mathbb{X}^k},
		\end{equation*}
		and if $\mathbb{X}$ is a Hilbert space then
		\begin{equation*}
			\|\mathbf{q}_{S,\bm{\alpha}}(\mathbf{x}_{1},...,\mathbf{x}_{k})-\mathbf{q}_{S,\bm{\alpha}}(\mathbf{z}_{1},...,\mathbf{z}_{k})\|_{\mathbb{X}}\leq \sqrt{\sum_{i=1,...,k}\|\mathbf{x}_i-\mathbf{z}_i\|_{\mathbb{X}}^{2}}\leq\|(\mathbf{x}_1,...,\mathbf{x}_k)-(\mathbf{z}_1,...,\mathbf{z}_k)\|_{\mathbb{X}^k}.
		\end{equation*}
	\end{thm}
	
	Finally, we focus on the point-wise component-wise quantile, (\ref{QoE-point}).
	\begin{thm}\label{thm-continuity-component-point}
		Let $\mathbb{X}$ be either $L_p[0,1]$ with $1 \leq p \leq \infty$ or $\mathbb{X}=C[0,1]$ or $\mathbb{X}=D[0,1]$. The point-wise component-wise quantile $\mathbf{q}_{P,\bm{\alpha}}$ exists, is unique, and is Lipschitz continuous. In particular, if $\mathbb{X}=L_p[0,1]$ with $1 \leq p < \infty$, then
		\begin{equation*}
			\|\mathbf{q}_{P,\bm{\alpha}}(\mathbf{f}_{1},...,\mathbf{f}_{k})-\mathbf{q}_{P,\bm{\alpha}}(\mathbf{g}_{1},...,\mathbf{g}_{k})\|_{\mathbb{X}}\leq \bigg(\sum_{i=1}^{k}\int_{0}^{1}|\mathbf{f}_{i}(x)-\mathbf{g}_i(x)|^{p}dx\bigg)^{1/p}
		\end{equation*}
		\begin{equation*}
			\leq\|(\mathbf{f}_1,...,\mathbf{f}_k)-(\mathbf{g}_1,...,\mathbf{g}_k)\|_{\mathbb{X}^k}.
		\end{equation*}
		If either $\mathbb{X}=L_\infty[0,1]$ or if $\mathbb{X}=C[0,1]$ or $\mathbb{X}=D[0,1]$, then
		\begin{equation*}
			\|\mathbf{q}_{P,\bm{\alpha}}(\mathbf{f}_{1},...,\mathbf{f}_{k})-\mathbf{q}_{P,\bm{\alpha}}(\mathbf{g}_{1},...,\mathbf{g}_{k})\|_{\mathbb{X}}\leq \max_{i=1,...,k}\|\mathbf{f}_i-\mathbf{g}_i\|_{\mathbb{X}}\leq\|(\mathbf{f}_1,...,\mathbf{f}_k)-(\mathbf{g}_1,...,\mathbf{g}_k)\|_{\mathbb{X}^k}.
		\end{equation*}
	\end{thm}
	
	\begin{rem}
		For the sake of clarity we presented the results for the interval $[0,1]$, but it is easy to see that the result holds for any interval $[0,T]$, and for the space $D[0,\infty)$.
	\end{rem}
	It is possible to see that such a result might easily hold for other function spaces, such as Sobolev spaces.
	\subsection{Geometric quantile}\label{SubSec-Continuity-geo}
	
	The following convex analysis result generalizes Lemma 2.14 in \cite{Kemperman}.
	\begin{lem}\label{lem-ineq}
		Let $\mathbb{X}$ be a Banach space. Let $W$ be a one-dimensional linear subspace of $\mathbb{X}$. Then $\mathbb{X}=W\oplus Y$, where $Y$ is the kernel of a norm-1 projection operator with range $W$. Moreover, for every $\mathbf{x}\in W$ we have that
		\begin{equation}\label{ineq}
			\|\mathbf{x}\|_{\mathbb{X}}\leq \|\mathbf{x}+\mathbf{y}\|_{\mathbb{X}},\quad\forall\mathbf{y}\in Y,
		\end{equation}
		and if $\mathbb{X}$ is strictly convex then
		\begin{equation}\label{ineq-strict}
			\|\mathbf{x}\|_{\mathbb{X}}< \|\mathbf{x}+\mathbf{y}\|_{\mathbb{X}},\quad\forall\mathbf{y}\in Y.
		\end{equation}
		Finally, we have that for every $\mathbf{y}\in Y$
		\begin{equation}\label{ineq-y}
			\|\mathbf{y}\|_{\mathbb{X}}\leq \|\mathbf{x}+\mathbf{y}\|_{\mathbb{X}},\quad\forall\mathbf{x}\in W,
		\end{equation}
		and if $\mathbb{X}$ is strictly convex then
		\begin{equation}\label{ineq-strict-y}
			\|\mathbf{y}\|_{\mathbb{X}}< \|\mathbf{x}+\mathbf{y}\|_{\mathbb{X}},\quad\forall\mathbf{x}\in W.
		\end{equation}
	\end{lem}
	\begin{rem}
		Since $\mathbf{y}\in Y$ implies that $\lambda\mathbf{y}\in Y$ for every $\lambda\in \mathbb{R}$, we have that (\ref{ineq}) is equivalent to: for every $\mathbf{x}\in\mathbb{X}$ we have that $\|\mathbf{x}\|_{\mathbb{X}}\leq \|\mathbf{x}+\lambda\mathbf{y}\|_{\mathbb{X}}$, for every $\mathbf{y}\in Y$ and every $\lambda\in\mathbb{R}$. The same holds for (\ref{ineq-strict}) with strict inequality and the same arguments apply to (\ref{ineq-y}) and (\ref{ineq-strict-y}). Thus, (\ref{ineq}) ((\ref{ineq-strict})) is equivalent to James (strict) orthogonality between any element of $W$ and $Y$, which we write $\mathbf{x}\perp\mathbf{y}$ ($\mathbf{x}\perp_{s}\mathbf{y}$), and (\ref{ineq-y}) ((\ref{ineq-strict-y})) is equivalent to James (strict) orthogonality between any element of $Y$ and $W$, which we write $\mathbf{y}\perp\mathbf{x}$ ($\mathbf{y}\perp_{s}\mathbf{x}$).
	\end{rem}
	In the following result we present sufficient conditions for the uniqueness of the geometric quantile, thus generalizing Theorem 2.17 in \cite{Kemperman}.
	\begin{thm}\label{thm-extension-of-Kemperman-}
		Let $\mathbb{X}$ be a strictly convex Banach space. The geometric quantile (\ref{geometric-quantile-Banach}), if it exists, is unique if one of the following conditions hold:
		\\\textnormal{(i)} $\mathbf{x}_1,...,\mathbf{x}_k$ do not lie on a straight line,
		\\\textnormal{(ii)} $\mathbf{x}_1,...,\mathbf{x}_k$ lie on a straight line and $\alpha\in(0,1)\setminus\{\frac{1}{k},\frac{2}{k},...,\frac{k-1}{k}\}$.
	\end{thm}
	For the geometric median we obtain both necessary and sufficient conditions.
	\begin{thm}\label{thm-extension-of-Kemperman-median}
		Let $\mathbb{X}$ be a strictly convex Banach space. The geometric median (\ref{geometric-quantile-Banach}) with $\mathbf{u}=\mathbf{0}$, if it exists, is unique if and only if one of the following conditions hold:
		\\\textnormal{(i)} $\mathbf{x}_1,...,\mathbf{x}_k$ do not lie on a straight line,
		\\\textnormal{(ii)} $\mathbf{x}_1,...,\mathbf{x}_k$ lie on a straight line, and $k$ is odd.
	\end{thm}
	To obtain necessary and sufficient conditions for the geometric quantile we need:
	\begin{ass}\label{A-1}
		When $\mathbf{x}_1,...,\mathbf{x}_k$ lie on a straight line $W$ assume that the norm of $\mathbb{X}$ is G\^{a}teaux differentiable at $\mathbf{x}_i\neq\mathbf{0}$, for some $i=1,...,k$, in some direction $\mathbf{z}\in Y$, where $Y$ is complementary to $W$, \textit{i.e.}~$\mathbb{X}=Y+W$, and such that $\langle\mathbf{u} ,\mathbf{z}\rangle\neq 0$.
	\end{ass}
	\noindent This assumption is always satisfied when the norm of $\mathbb{X}$ is G\^{a}teaux differentiable at $\mathbf{x}_i\neq\mathbf{0}$, thus it is always satisfied when $\mathbb{X}$ is smooth at $\mathbf{x}_i\neq\mathbf{0}$, \textit{i.e.}~there exists a unique $f\in\mathbb{X}^*$ s.t.~$f(\mathbf{x})=\|\mathbf{x}\|_{\mathbb{X}}$, or more generally when $\mathbb{X}$ is smooth, \textit{i.e.}~$\mathbb{X}$ is smooth at every $\mathbf{x}\in\mathbb{X}$. Further, this assumption is satisfied when $\mathbb{X}^*$ is strictly convex, or when $\mathbb{X}$ is uniformly smooth, for example (see Chapter 8 in \cite{Fabian}). Note that $L_p$ spaces with $p\in(1,\infty)$ and Hilbert spaces are uniformly smooth and uniformly convex (and so strictly convex) Banach spaces. In particular, their norm is Fr\'{e}chet differentiable and so G\^{a}teaux differentiable. 
	
	Recall the decomposition $\mathbf{u}$ when $\mathbf{x}_1,...,\mathbf{x}_k$ lie on a straight line $W$, \textit{i.e.}~$\mathbf{u}=u\mathbf{e}+\mathbf{v}$, where $\mathbf{e}\in Y^{\perp}$ with $\|\mathbf{e}\|_{\mathbb{X}^{*}}=1$, $u\in(-1,1)$, and $\mathbf{v}\in W^{\perp}$.
	\begin{thm}\label{thm-extension-of-Kemperman}
		Let $\mathbb{X}$ be a strictly convex Banach space and let Assumption \ref{A-1} hold. The geometric quantile (\ref{geometric-quantile-Banach}), if it exists, is unique if and only if one of the following conditions hold:
		\\\textnormal{(i)} $\mathbf{x}_1,...,\mathbf{x}_k$ do not lie on a straight line,
		\\\textnormal{(ii)} $\mathbf{x}_1,...,\mathbf{x}_k$ lie on a straight line and $\mathbf{v}\neq\mathbf{0}$,
		\\\textnormal{(iii)} $\mathbf{x}_1,...,\mathbf{x}_k$ lie on a straight line, $\mathbf{v}=\mathbf{0}$, and $\alpha\in(0,1)\setminus\{\frac{1}{k},\frac{2}{k},...,\frac{k-1}{k}\}$.
	\end{thm}
	It is possible to see that $\mathbf{v}$, $\alpha$, and Assumption \ref{A-1} all depend on $W$. Thus, the conditions of Theorems \ref{thm-extension-of-Kemperman-} and \ref{thm-extension-of-Kemperman} depend on $W$. When the data are random we do not know what $W$ is, and so we cannot directly apply the previous theorems to know whether or not the quantile is unique. In the following result we solve this issue by showing that the quantile is unique if a certain condition on $\mathbf{u}$, which is given a priori, is satisfied.
	\begin{co}\label{co-unique}
		Let $\mathbb{X}$ be a smooth and strictly convex Banach space. The geometric quantile (\ref{geometric-quantile-Banach}) is unique if one of the following conditions hold: 
		\\\textnormal{(i)} $\mathbf{x}_1,...,\mathbf{x}_k$ do not lie on a straight line,
		\\\textnormal{(ii)} $\|\mathbf{u}\|_{\mathbb{X}^*}\notin\{1-\frac{2j}{k},j=1,...,\lfloor\frac{k}{2}\rfloor\}$.
	\end{co}
	\begin{rem}
		Point (i) of Corollary \ref{co-unique} implies the well known result that if the data are distributed according to an absolutely continuous distribution then the quantile is unique. The advantage of point (ii) of Corollary \ref{co-unique} consists in ensuring uniqueness of the geometric quantile \textnormal{independently} of the distribution of the data. Moreover, this condition is quite weak since among all the possible values that $\|\mathbf{u}\|_{\mathbb{X}^*}$ can take in $[0,1)$, only finitely many of them (no more than $\lfloor\frac{k}{2}\rfloor$) do not ensure uniqueness. For example, from Corollary \ref{co-unique} we see that if $k$ is odd the geometric median is always uniquely defined.
	\end{rem}
	
	In Theorems \ref{thm-extension-of-Kemperman-}, \ref{thm-extension-of-Kemperman-median}, and \ref{thm-extension-of-Kemperman}, we show the uniqueness of the geometric quantile and median under certain conditions on the position of the data $\mathbf{x}_1,...,\mathbf{x}_k$ (and on $\mathbf{v}$ and $\alpha$). In the following result we show that there is a precise relation between the position of the data $\mathbf{x}_1,...,\mathbf{x}_k$ and the position of the geometric quantiles, even when they are not unique.
	\begin{pro}\label{pro-extra}
		Let $\mathbb{X}$ be a strictly convex space and let $\mathbf{x}_1,...,\mathbf{x}_k$ lie on a straight line $W$. Then the set of geometric median is a subset of $W$, in particular
		\begin{equation*}
			\textnormal{(i)}\quad{\displaystyle {\underset {\mathbf{y}\in \mathbb{X}}{\operatorname {arg\,min} }}\sum _{i=1}^{k}\left\|\mathbf{x}_{i}-\mathbf{y}\right\|_{\mathbb{X}}}={\displaystyle {\underset {\mathbf{y}\in W}{\operatorname {arg\,min} }}\sum _{i=1}^{k}\left\|\mathbf{x}_{i}-\mathbf{y}\right\|_{\mathbb{X}}}.
		\end{equation*}
		Further, let $\mathbf{u}\in \mathbb{X}^*$ with $\|\mathbf{u}\|_{\mathbb{X}^*}<1$. If $\mathbf{v}=0$ then the set of geometric quantiles is a subset of $W$ and
		\begin{equation*}
			\textnormal{(ii)}\quad{\displaystyle {\underset {\mathbf{y}\in \mathbb{X}}{\operatorname {arg\,min} }}\sum _{i=1}^{k}\left\|\mathbf{x}_{i}-\mathbf{y}\right\|_{\mathbb{X}}}+\langle\mathbf{u} ,\mathbf{x}_{i}-\mathbf{y}\rangle={\displaystyle {\underset {\mathbf{y}\in W}{\operatorname {arg\,min} }}\sum _{i=1}^{k}\left\|\mathbf{x}_{i}-\mathbf{y}\right\|_{\mathbb{X}}}+\langle\mathbf{u} ,\mathbf{x}_{i}-\mathbf{y}\rangle.
		\end{equation*}
		Let $\mathbb{X}$ additionally satisfy Assumption \ref{A-1}. Then $\mathbf{v}=0$ if and only if the set of geometric quantiles is a subset of $W$. In particular, if $\mathbf{v}=0$ then
		\begin{equation*}
			\textnormal{(iii)}\quad{\displaystyle {\underset {\mathbf{y}\in \mathbb{X}}{\operatorname {arg\,min} }}\sum _{i=1}^{k}\left\|\mathbf{x}_{i}-\mathbf{y}\right\|_{\mathbb{X}}}+\langle\mathbf{u} ,\mathbf{x}_{i}-\mathbf{y}\rangle={\displaystyle {\underset {\mathbf{y}\in W}{\operatorname {arg\,min} }}\sum _{i=1}^{k}\left\|\mathbf{x}_{i}-\mathbf{y}\right\|_{\mathbb{X}}}+\langle\mathbf{u} ,\mathbf{x}_{i}-\mathbf{y}\rangle.
		\end{equation*}
	\end{pro}
	Since $\mathbf{u}$ is linear, the existence of the geometric quantile follows from the existence results of the geometric median. In particular, by \cite{Valadier} (see also Remark 3.5 in \cite{Kemperman}) existence is ensured when $\mathbb{X}$ is reflexive and by Theorem 3.6 in \cite{Kemperman} when $\mathbb{X}$ is the dual of a separable Banach space. This includes the case of smooth Banach spaces.
	
	We now focus on the continuity properties of the geometric quantile. The first result concerns the continuity of the geometric quantile for distributions (not necessarily discrete) and represents the quantile extension of a well-known result for the median (see \cite{Cadre2001,Kemperman}).
	\begin{lem}\label{lem-continuity-quantiles}
		Let $\mathbb{X}$ be a separable and strictly convex Banach space. Let $(\mu_n)_{n\in\mathbb{N}}$ be a sequence of probability measures on $\mathbb{X}$ such that $\mu_n\stackrel{w}{\to}\mu$, where $\mu$ is a probability measure on $\mathbb{X}$. Assume that $\mu$ posses a unique geometric quantile. Then, any geometric quantile of $\mu_n$ converges in the weak* topology to the geometric quantile of $\mu$. If $\mathbb{X}$ is finite dimensional then any geometric quantile of $\mu_n$ converges to the geometric quantile of $\mu$.
	\end{lem}
	\begin{rem}
		Any probability measure posseses a unique geometric quantile when it is not concentrated on a straight line (see \cite{Kemperman}). When $\mathbb{X}=\mathbb{R}$ a sufficient condition for the uniqueness of the geometric quantile is that density of $\mu$ is continuous and strictly positive around the quantile (see Theorem \ref{t-k-to-infinity-multivariate-robust} for $d=1$).
	\end{rem}
	
	We present a generalization of Theorem 2.24 in \cite{Kemperman} in which we show that on finite-dimensional spaces the convergence of the geometric quantile is uniform with respect to the quantile parameter $\mathbf{u}$.
	\begin{lem}\label{lem-uniform-continuity}
		Let $\mathbb{X}$ be a finite-dimensional Banach space. Let $(\mu_n)_{n\in\mathbb{N}}$ be a sequence of probability measures on $\mathbb{X}$ such that $\mu_n\stackrel{w}{\to}\mu$, where $\mu$ is a probability measure on $\mathbb{X}$. Assume that $\mu$ posses a unique geometric quantile $\mathbf{y}^*_\mathbf{u}$. Let $\mathbf{y}^*_{n,\mathbf{u}}$ be a geometric quantile of $\mu_n$. For any $\varepsilon>0$, $\sup_{\|\mathbf{u}\|_{\mathbb{X}^*}<1-\varepsilon}\|\mathbf{y}^*_{n,\mathbf{u}}-\mathbf{y}^*_\mathbf{u}\|_{\mathbb{X}}\to0$ as $n\to\infty$. Further, if $\mu$ is atomless, then for any $c_n\to 0$ and $\mathbf{v}\in\mathbb{R}^d$ with $\| \mathbf{v}\|_{\mathbb{X}^*}<1$, we have $\sup_{\|\mathbf{u}-\mathbf{v}\|_{\mathbb{X}^*}\leq c_n}\|\mathbf{y}^*_{n,\mathbf{u}}-\mathbf{y}^*_\mathbf{v}\|_{\mathbb{X}}\to0$ as $n\to\infty$.
	\end{lem}
	
	In the next results we strengthen Lemma \ref{lem-continuity-quantiles} for empirical distributions by proving strong convergence even in the infinite dimensional setting. 
	
	Let $\mathcal{X}_{k}:=\{(\mathbf{x}_1,...,\mathbf{x}_k)\in \mathbb{X}^{k}:\textnormal{$\mathbf{x}_1,...,\mathbf{x}_k$ do not lie on a straight line}\}$, where $\mathbb{X}$ is a Banach space.
	\begin{lem}\label{lem-open-set}
		Let $\mathbb{X}$ be a Banach space. The set $\mathcal{X}_{k}$ is open.
	\end{lem}
	\begin{thm}\label{thm-continuity}
		Let $\mathbb{X}$ be a reflexive and strictly convex Banach space. Then, the quantile function $\mathbf{q}_{\mathbf{u}}$ is continuous on $\mathcal{X}_{k}$. If $\mathbb{X}$ is also smooth and $\|\mathbf{u}\|_{\mathbb{X}^*}\notin\{1-\frac{2j}{k},j=1,...,\lfloor\frac{k}{2}\rfloor\}$, then the quantile function $\mathbf{q}_{\mathbf{u}}$ is continuous on $\mathbb{X}^{k}$.
	\end{thm}
	In the following result we show that when $\mathbb{X}$ is a Hilbert space the geometric quantile is (almost) a linear combination of $\mathbf{x}_1,...,\mathbf{x}_k,\mathbf{u}$ and that the geometric quantile is the solution of a $(k+1)$-dimensional minimization problem. This generalizes one of the main intuitions of \cite{Gervini2008} to quantiles.
	\begin{thm}\label{thm-extension-Gervini}
		Let $\mathbb{X}$ be a Hilbert space. Let $T_{k+1}:=\{\mathbf{w}\in[0,1]^{k+1}:w_1+\cdots+w_{k+1}=1\}$. Then, any geometric quantile is given by
		\begin{equation*}
			\frac{1}{1-w^*_{k+1}}\Big(\sum_{i=1}^{k}w^*_i\mathbf{x}_i+w^*_{k+1}\mathbf{u}\Big),
		\end{equation*}
		where $(w^*_1,...,w^*_{k+1})$ is given by
		\begin{equation*}
			{\displaystyle {\underset {\mathbf{w}\in T_{k+1}}{\operatorname {arg\,min} }}\sum _{i=1}^{k}\left\|\mathbf{x}_{i}-\frac{1}{1-w_{k+1}}\Big(\sum_{l=1}^{k}w_{l}\mathbf{x}_{l}+w_{k+1}\mathbf{u}\Big)\right\|_{\mathbb{X}}}+\langle \mathbf{u},\mathbf{x}_i-\frac{1}{1-w_{k+1}}\Big(\sum_{l=1}^{k}w_{l}\mathbf{x}_{l}+w_{k+1}\mathbf{u}\Big)\rangle.
		\end{equation*}
		In particular, if $\mathbf{y}^*$ is a geometric quantile and $\mathbf{y}^*\neq\mathbf{x}_i$, $i=1,...,k$, then \begin{equation*}
			w^*_i=\frac{\|\mathbf{x}_i-\mathbf{y}^*\|_{\mathbb{X}}^{-1}}{\sum_{i=1}^{k}\|\mathbf{x}_i-\mathbf{y}^*\|_{\mathbb{X}}^{-1}+k},\quad\textnormal{for $i=1,...,k$, and}\quad 	w^*_{k+1}=\frac{k}{\sum_{i=1}^{k}\|\mathbf{x}_i-\mathbf{y}^*\|_{\mathbb{X}}^{-1}+k}.
		\end{equation*}
	\end{thm}
	
	We stress that the above theorem holds even in the case of non-uniqueness of the geometric quantile. When we have uniqueness, \textit{i.e.}~when the data do not lie on a straight line or when $\|\mathbf{u}\|_{\mathbb{X}^*}\notin\{1-\frac{2j}{k},j=1,...,\lfloor\frac{k}{2}\rfloor\}$, then $\mathbf{q}_\mathbf{u}((\mathbf{x}_1,...,\mathbf{x}_k))=\frac{1}{1-w^*_{k+1}}\Big(\sum_{i=1}^{k}w^*_i\mathbf{x}_i+w^*_{k+1}\mathbf{u}\Big)$.
	
	In the remaining part of this subsection we investigate the properties of the geometric quantile when $\mathbb{X}=\ell_1$, which is not a strictly convex space. Observe that $\mathbf{u}=(u^{(l)})_{l\in\mathbb{N}}$ belongs to $\ell_{\infty}$. By abuse of notation we use $\mathbf{q}_{\mathbf{u}}$ to indicate the geometric quantile function in this space as well.
	\begin{pro}\label{pro-l1}
		Let $\mathbb{X}=\ell_1$. The geometric quantile exists. Further, the geometric quantile is unique if and only if $|u^{(l)}|\notin\{1-\frac{2j}{k},j=1,...,\lfloor\frac{k}{2}\rfloor\}$, $\forall l\in\mathbb{N}$. Further, if $|u^{(l)}|\notin\{1-\frac{2j}{k},j=1,...,\lfloor\frac{k}{2}\rfloor\}$, $\forall l\in\mathbb{N}$, the geometric quantile function $\mathbf{q}_{\mathbf{u}}$ is Lipschitz continuous on $\mathbb{X}^k$: for every $\mathbf{x}_1,...,\mathbf{x}_k,\mathbf{z}_1,...,\mathbf{z}_k\in \ell_1$, we have that
		\begin{equation*}
			\|\mathbf{q}_{\mathbf{u}}(\mathbf{x}_1,...,\mathbf{x}_k)-\mathbf{q}_{\mathbf{u}}(\mathbf{z}_1,...,\mathbf{z}_k)\|_{\ell_1}\leq\|(\mathbf{x}_1,...,\mathbf{x}_k)-(\mathbf{z}_1,...,\mathbf{z}_k)\|_{\ell_1^k}.
		\end{equation*}
	\end{pro}
	From Proposition \ref{pro-l1} we deduce that the geometric median is unique if and only if $k$ is odd, and when $k$ is odd the geometric median is Lipschitz continuous.

	\section{Appendix}
	\section*{Proofs of Section \ref{SubSec-QoE-component}}
\begin{proof}[Proofs of Theorems \ref{t-1-q-component_1}, \ref{t-1-q-component_2}, and \ref{t-1-q-component_3}]
	The results follow from the equivariance property of the univariate quantile, by the continuity results (Theorems \ref{thm-continuity-component-Hamel},  \ref{thm-continuity-component-Scahuder}, and \ref{thm-continuity-component-point}) and by the continuous mapping theorem.
\end{proof}
\begin{proof}[Proof of Proposition \ref{t-k-to-infinity-component}]
	
	Since convergence in probability of a random vector follows from the convergence in probability of each of its components, we obtain the result from the same arguments used for the proof of the convergence in probability in Theorem \ref{t-k-to-infinity-multivariate-robust}.
	
\end{proof}
\begin{proof}[Proofs of Theorem \ref{t-k-to-infinity-multivariate-robust}]
	For the first statement since the convergence in probability of a vector is determined by the convergence in probability of its components, it is enough to focus on the one dimensional case. Consider any $c>0$ and $\beta\in(0,1)$ such that $\lim\limits_{n\to\infty}\frac{l_n}{cn^{\beta}}=0$.	By assumption $F^{-1}(\alpha)$ is uniquely defined. Indeed, as the set of geometric quantiles is connected and compact, if $f$ is continuous and strictly positive at any $F^{-1}(\alpha)$, then $f$ is continuous and strictly positive on the set of geometric quantiles, which implies that the set consists of only one point. Now, for any $\varepsilon>0$
	\begin{equation*}
		\mathbb{P}(|a_{\lfloor n/cn^\beta\rfloor}(T_{\mathbf{u},n,cn^\beta}-\theta_{0})-F^{-1}(\alpha)|>\varepsilon)
	\end{equation*}
	\begin{equation*}
		=\mathbb{P}(a_{\lfloor n/cn^\beta\rfloor}(T_{\mathbf{u},n,cn^\beta}-\theta_{0})-F^{-1}(\alpha)>\varepsilon)+\mathbb{P}(a_{\lfloor n/cn^\beta\rfloor}(T_{\mathbf{u},n,cn^\beta}-\theta_{0})-F^{-1}(\alpha)<-\varepsilon).
	\end{equation*}
	To lighten the notation, let 
	\begin{equation*}
		W_n^{(j)}:=a_{\lfloor n/cn^\beta\rfloor}(Z^{(j)}_{\lfloor n/cn^\beta\rfloor}-z)-F^{-1}(\alpha)
	\end{equation*}
	and 
	\begin{equation*}
		\tilde{W}_n^{(j)}:=a_{\lfloor n/cn^\beta\rfloor}(\tilde{Z}^{(j)}_{\lfloor n/cn^\beta\rfloor}-z)-F^{-1}(\alpha),
	\end{equation*}
	for $j=1,...,\lfloor c n^{\beta}\rfloor$. By the equivariance property of the quantile we have
	\begin{equation*}
		a_{\lfloor n/cn^\beta\rfloor}(\theta_{\mathbf{u},n,cn^\beta}-\theta_{0})-F^{-1}(\alpha)=q_{\alpha}(\tilde{W}_n^{(1)},...,\tilde{W}_n^{(\lfloor c n^{\beta}\rfloor)}),
	\end{equation*}
	so
	\begin{equation*}
		\mathbb{P}(a_{\lfloor n/cn^\beta\rfloor}(T_{\mathbf{u},n,cn^\beta}-\theta_{0})-F^{-1}(\alpha)>\varepsilon)\leq\mathbb{P}\Bigg(\sum_{j=1}^{\lfloor cn^\beta\rfloor}\mathbf{1}_{\{\tilde{W}_n^{(j)}>\varepsilon \}}\geq\lfloor cn^\beta\rfloor(1-\alpha)\Bigg).
	\end{equation*}
	Now, assume without loss of generality that the last $l_n$ of the $\tilde{W}_n^{(j)}$'s are contaminated, then
	\begin{equation*}
		\mathbb{P}\Bigg(\sum_{j=1}^{\lfloor cn^\beta\rfloor}\mathbf{1}_{\{\tilde{W}_n^{(j)}>\varepsilon \}}\geq\lfloor cn^\beta\rfloor(1-\alpha)\Bigg)
	\end{equation*}
	\begin{equation*}
		\leq \mathbb{P}\Bigg(\sum_{j=1}^{\lfloor cn^\beta\rfloor-l_n}\mathbf{1}_{\{W_n^{(j)}>\varepsilon \}}+l_n\geq\lfloor cn^\beta\rfloor(1-\alpha)\Bigg)\leq \mathbb{P}\Bigg(\sum_{j=1}^{\lfloor cn^\beta\rfloor}\mathbf{1}_{\{W_n^{(j)}>\varepsilon \}}+l_n\geq\lfloor cn^\beta\rfloor(1-\alpha)\Bigg).
	\end{equation*}
	As $\mathbb{P}(a_{\lfloor n/cn^\beta\rfloor}(Z^{(j)}_{\lfloor n/cn^\beta\rfloor}-\theta_{0})>x)\to\mathbb{P}(Y>x)$ for any continuity point $x$ by Assumption \ref{A2} and $f$ is continuous at $F^{-1}(\alpha)$ by assumption of the statement, we obtain
	\begin{equation*}
		\mathbb{P}(W_n^{(1)}>\varepsilon)\to\mathbb{P}(Y-F^{-1}(\alpha)>\varepsilon)<1-\alpha,
	\end{equation*}
	as $n\to\infty$, where we consider that $F^{-1}(\alpha)+\varepsilon$ is a continuity point of $F$ when $\varepsilon$ is small enough. So using $\lim\limits_{n\to\infty}\frac{l_n}{cn^{\beta}}=0$, we have $1-\alpha-\mathbb{P}(W_n^{(1)}>\varepsilon)-\frac{l_n}{\lfloor cn^{\beta}\rfloor}>0$ for $n$ large enough. Then, by Hoeffding's inequality we obtain
	\begin{equation*}
		\mathbb{P}\Bigg(\frac{1}{\lfloor cn^\beta\rfloor}\sum_{j=1}^{\lfloor cn^\beta\rfloor}\mathbf{1}_{\{W_n^{(j)}>\varepsilon \}}-\mathbb{P}(W_n^{(1)}>\varepsilon)\geq 1-\alpha-\mathbb{P}(W_n^{(1)}>\varepsilon)-\frac{l_n}{\lfloor cn^{\beta}\rfloor}\Bigg)
	\end{equation*}
	\begin{equation*}
		\leq e^{-\frac{2}{\lfloor cn^\beta\rfloor}(1-\alpha-\mathbb{P}(W_n^{(1)}>\varepsilon)-\frac{l_n}{\lfloor cn^{\beta}\rfloor})^{2}}\to0,\quad\textnormal{as $n\to\infty$}.
	\end{equation*}
	From this we obtain $\mathbb{P}(a_{\lfloor n/cn^\beta\rfloor}(T_{\mathbf{u},n,cn^\beta}-\theta_{0})-F^{-1}(\alpha)>\varepsilon)\to0$ as $n\to\infty$. 
	
	By the same arguments $\mathbb{P}(a_{\lfloor n/cn^\beta\rfloor}(T_{\mathbf{u},n,cn^\beta}-\theta_{0})-F^{-1}(\alpha)<-\varepsilon)\to0$, as $n\to\infty$. Thus, we have that $a_{\lfloor n/cn^\beta\rfloor}(T_{\mathbf{u},n,cn^\beta}-\theta_{0})\stackrel{p}{\to} F^{-1}(\alpha)$, as $n\to\infty$. Hence, we obtain the first statement.

	Let $\beta\in(2\gamma,\max_{i=1,..,d}\beta_i^{*})$ and by slight abuse of notation let 
	\begin{equation*}
		W_{n,i}^{(j)}:=a_{\lfloor n/cn^\beta\rfloor}(Z^{(j)}_{\lfloor n/cn^\beta\rfloor,i}-\theta_{0,i})-F_{i}^{-1}(\alpha_i),
	\end{equation*}
	\begin{equation*}
		\tilde{W}_{n,i}^{(j)}:=a_{\lfloor n/cn^\beta\rfloor}(\tilde{Z}^{(j)}_{\lfloor n/cn^\beta\rfloor,i}-\theta_{0,i})-F_{i}^{-1}(\alpha_i),
	\end{equation*}
	\begin{equation*}
		W_n^{(j)}=(W_{n,1}^{(j)},...,W_{n,d}^{(j)}),
	\end{equation*}
	and 
	\begin{equation*}
		\tilde{W}_n^{(j)}=(\tilde{W}_{n,1}^{(j)},...,\tilde{W}_{n,d}^{(j)}),
	\end{equation*}
	for $j=1,...,\lfloor c n^{\beta}\rfloor$ and $i=1,...,d$, where $\bm\theta_{0}=(\theta_{0,1},...,\theta_{0,d})$. For every $x_1,...,x_d\in\mathbb{R}$,
	\begin{equation*}
		\Bigg\{\sum_{j=1}^{\lfloor cn^\beta\rfloor}\mathbf{1}_{\{W_{n,1}^{(j)}> x_1 \}}\leq\lfloor cn^\beta\rfloor(1-\alpha_1)-1-l_n,...,\sum_{j=1}^{\lfloor cn^\beta\rfloor}\mathbf{1}_{\{W_{n,d}^{(j)}> x_d \}}\leq\lfloor cn^\beta\rfloor(1-\alpha_d)-1-l_n \Bigg\}
	\end{equation*}	
	\begin{equation*}
		\stackrel{a.s.}{=}\Bigg\{\sum_{j=1}^{\lfloor cn^\beta\rfloor}\mathbf{1}_{\{W_{n,1}^{(j)}\leq x_1 \}}\geq\lfloor cn^\beta\rfloor\alpha_1+1+l_n,...,\sum_{j=1}^{\lfloor cn^\beta\rfloor}\mathbf{1}_{\{W_{n,d}^{(j)}\leq x_d \}}\geq\lfloor cn^\beta\rfloor\alpha_d+1+l_n \Bigg\}
	\end{equation*}
	\begin{equation*}
		\subset \Big\{(\mathbf{q}_{P,\bm{\alpha}}(\tilde{W}_n^{(1)},...,\tilde{W}_n^{(\lfloor c n^{\beta}\rfloor)}))_{1}\leq x_1,...,(\mathbf{q}_{P,\bm{\alpha}}(\tilde{W}_n^{(1)},...,\tilde{W}_n^{(\lfloor c n^{\beta}\rfloor)}))_{d}\leq x_d\Big\}
	\end{equation*}
	\begin{equation*}
		\subset \Bigg\{\sum_{j=1}^{\lfloor cn^\beta\rfloor}\mathbf{1}_{\{W_{n,1}^{(j)}\leq x_1 \}}\geq\lfloor cn^\beta\rfloor\alpha_1 -l_n,...,\sum_{j=1}^{\lfloor cn^\beta\rfloor}\mathbf{1}_{\{W_{n,d}^{(j)}\leq x_d \}}\geq\lfloor cn^\beta\rfloor\alpha_d -l_n\Bigg\}
	\end{equation*}
	\begin{equation}\label{inclusion}
		\stackrel{a.s.}{=} \Bigg\{\sum_{j=1}^{\lfloor cn^\beta\rfloor}\mathbf{1}_{\{W_{n,1}^{(j)}> x_1 \}}\leq\lfloor cn^\beta\rfloor(1-\alpha_1) +l_n,...,\sum_{j=1}^{\lfloor cn^\beta\rfloor}\mathbf{1}_{\{W_{n,d}^{(j)}> x_d \}}\leq\lfloor cn^\beta\rfloor(1-\alpha_d) +l_n\Bigg\}
	\end{equation}
	Let $\mathbf{V}_n$ be $d\times d$ matrix given by 
	\begin{equation*}
		\mathbf{V}_n=\Big(\mathbb{P}(W_{n,i}^{(1)}> x_i ,W_{n,j}^{(1)}> x_j )-\mathbb{P}(W_{n,i}^{(1)}> x_i )\mathbb{P}(W_{n,j}^{(1)}> x_j )\Big)_{i,j=1,...,d}.
	\end{equation*}
	By the central limit theorem
	\begin{equation*}
		\sqrt{\lfloor cn^\beta\rfloor}\mathbf{V}^{-1/2}_n\Bigg(\frac{1}{\lfloor cn^\beta\rfloor}\sum_{j=1}^{\lfloor cn^\beta\rfloor}\mathbf{1}_{\{W_{n,i}^{(j)}> x_i \}}-1+F_{n,i}(x_i+F_{i}^{-1}(\alpha_i))\Bigg)_{i=1,...,d}\stackrel{d}{\to} N(\mathbf{0},\mathbf{I}).
	\end{equation*}
	Let $x_i=\frac{y_i}{\sqrt{cn^\beta}}$, where $y_i\in\mathbb{R}$. By assumption we have that $\mathbf{V}_n\to\mathbf{V}$ as $n\to\infty$, where	
	\begin{equation*}
		\mathbf{V}=\Big(\mathbb{P}(Y_{i}> F_{i}^{-1}(\alpha_i) ,Y_{j}> F_{j}^{-1}(\alpha_j) )-(1-\alpha_i)(1-\alpha_j)\Big)_{i,j=1,...,d}.
	\end{equation*}	
	Further,
	\begin{equation*}
		\lim\limits_{n\to\infty} \sqrt{\lfloor cn^\beta\rfloor}\Bigg(\frac{\lfloor cn^\beta\rfloor(1-\alpha_i)-1-l_n}{\lfloor cn^\beta\rfloor}-1+F_{n,i}(x_i+F_{i}^{-1}(\alpha_i))\Bigg)
	\end{equation*}
	\begin{equation*}
		=\lim\limits_{n\to\infty} \sqrt{\lfloor cn^\beta\rfloor}\Big(F_{n,i}(x_i+F_{i}^{-1}(\alpha_i))-\alpha_i\Big)
	\end{equation*}
	\begin{equation*}
		=\lim\limits_{n\to\infty} \sqrt{\lfloor cn^\beta\rfloor}\Big(F_i(x_i+F_{i}^{-1}(\alpha_i))-\alpha_i\Big)=f_i(F_i^{-1}(\alpha_i))y_i,
	\end{equation*}
	where we used that $\lim\limits_{n\to\infty}\frac{l_n}{\sqrt{cn^{\beta}}}=0$ and 
	\begin{equation*}
		\lim\limits_{n\to\infty} \sqrt{\lfloor cn^\beta\rfloor}\Big(F_i(x_i+F_{i}^{-1}(\alpha_i))-F_{n,i}(x_i+F_{i}^{-1}(\alpha_i))\Big)=0
	\end{equation*}
	for $\beta\in(2\gamma,\max_{i=1,..,d}\beta_i^{*})$, for every $i=1,...,d$. Therefore, 
	\begin{equation*}
		\mathbb{P}\Bigg(\sum_{j=1}^{\lfloor cn^\beta\rfloor}\mathbf{1}_{\{W_{n,1}^{(j)}> x_1 \}}\leq\lfloor cn^\beta\rfloor(1-\alpha_1)-1-l_n,...,\sum_{j=1}^{\lfloor cn^\beta\rfloor}\mathbf{1}_{\{W_{n,d}^{(j)}> x_d \}}\leq\lfloor cn^\beta\rfloor(1-\alpha_d)-1-l_n \Bigg)
	\end{equation*}
	\begin{equation*}
		=\mathbb{P}\Bigg(	\sqrt{\lfloor cn^\beta\rfloor}\Bigg(\frac{1}{\lfloor cn^\beta\rfloor}\sum_{j=1}^{\lfloor cn^\beta\rfloor}\mathbf{1}_{\{W_{n,1}^{(j)}> x_1 \}}-1+F_{n,1}(x_1+F_{1}^{-1}(\alpha_1))\Bigg)
	\end{equation*}
	\begin{equation*}
		\leq  \sqrt{\lfloor cn^\beta\rfloor}\Bigg(\frac{\lfloor cn^\beta\rfloor(1-\alpha_1)-1-l_n}{\lfloor cn^\beta\rfloor}-1+F_{n,1}(x_1+F_{1}^{-1}(\alpha_1))\Bigg),
	\end{equation*}
	\begin{equation*}
		..., \sqrt{\lfloor cn^\beta\rfloor}\Bigg(\frac{1}{\lfloor cn^\beta\rfloor}\sum_{j=1}^{\lfloor cn^\beta\rfloor}\mathbf{1}_{\{W_{n,d}^{(j)}> x_d \}}-1+F_{n,d}(x_d+F_{d}^{-1}(\alpha_d))\Bigg)
	\end{equation*}
	\begin{equation*}
		\leq  \sqrt{\lfloor cn^\beta\rfloor}\Bigg(\frac{\lfloor cn^\beta\rfloor(1-\alpha_d)-1-l_n}{\lfloor cn^\beta\rfloor}-1+F_{n,d}(x_d+F_{d}^{-1}(\alpha_d))\Bigg)\Bigg)
	\end{equation*}
	\begin{equation*}
		\to\mathbb{P}\bigg(\Upsilon_1\leq f_1(F_1^{-1}(\alpha_1))y_1,...,\Upsilon_d\leq f_d(F_d^{-1}(\alpha_d))y_d\bigg),
	\end{equation*}
	as $n\to\infty$, where $(\Upsilon_1,...,\Upsilon_d)\sim N(\mathbf{0},\mathbf{V})$. The same convergence holds for 
	\begin{equation*}
		\Big(\sum_{j=1}^{\lfloor cn^\beta\rfloor}\mathbf{1}_{\{W_{n,i}^{(j)}> x_i \}}\leq\lfloor cn^\beta\rfloor(1-\alpha_i) +l_n\Big)_{i=1,...,d}.
	\end{equation*}

	Therefore, using (\ref{inclusion}) we conclude that
	\begin{equation*}
		\mathbb{P}\Bigg( (\mathbf{q}_{P,\bm{\alpha}}(\tilde{W}_{n,i}^{(1)},...,\tilde{W}_{n,i}^{(\lfloor c n^{\beta}\rfloor)}))_{1}\leq x_1,...,(\mathbf{q}_{P,\bm{\alpha}}(\tilde{W}_{n,i}^{(1)},...,\tilde{W}_{n,i}^{(\lfloor c n^{\beta}\rfloor)}))_{d}\leq x_d\Bigg)
	\end{equation*}
	\begin{equation*}
		=\mathbb{P}\Bigg( \sqrt{cn^{\beta}}(\mathbf{q}_{P,\bm{\alpha}}(\tilde{W}_{n,i}^{(1)},...,\tilde{W}_{n,i}^{(\lfloor c n^{\beta}\rfloor)}))_{1}\leq y_1,...,(\mathbf{q}_{P,\bm{\alpha}}(\tilde{W}_{n,i}^{(1)},...,\sqrt{cn^{\beta}}\tilde{W}_{n,i}^{(\lfloor c n^{\beta}\rfloor)}))_{d}\leq y_d\Bigg)
	\end{equation*}
	\begin{equation*}
		\to\mathbb{P}\bigg(\Upsilon_1\leq f_1(F_1^{-1}(\alpha_1))y_1,...,\Upsilon_d\leq f_d(F_d^{-1}(\alpha_d))y_d\bigg),
	\end{equation*}
	as $n\to\infty$, for every $y_1,...,y_d\in\mathbb{R}$.	Finally, using the change of variable $w_i=y_if_i(F_i^{-1}(\alpha_i))$, we obtain that
	\begin{equation*}
		\sqrt{cn^{\beta}}\Bigg(f_i(F_i^{-1}(\alpha_i)) \mathbf{q}_{P,\bm{\alpha}}(\tilde{W}_{n,i}^{(1)},...,\tilde{W}_{n,i}^{(\lfloor c n^{\beta}\rfloor)})\Bigg)_{i=1,...,d}\stackrel{d}{\to}N(\mathbf{0},\mathbf{V}),
	\end{equation*}
	as $n\to\infty$.
\end{proof}

\begin{proof}[Proof of Corollary \ref{co-robust-component}]
	The result follows from Theorem \ref{t-k-to-infinity-multivariate-robust}.
\end{proof}

\begin{proof}[Proof of Proposition \ref{pro-fdd}]
	The result follows from Theorem \ref{t-k-to-infinity-multivariate-robust} and Corollary 2.4 in 
	\cite{Bogachev}.
\end{proof}

\section*{Proofs of Section \ref{SubSec-QoE-geo}}
\begin{proof}[Proof of Lemma \ref{Extension-of-Lemma2.1}]
	
	We slightly extend the arguments of the proof of Lemma 2.1 in \cite{Minsker}. Assume that the claim does not hold and without loss of generality $\| \mathbf{x}_j-\mathbf{z} \|_{\mathbb{X}}\leq r$ for $j=1,...,\lfloor (1-\nu)k\rfloor+1$. Denote by $DF(\mathbf{x}^*,\frac{\mathbf{z}-\mathbf{x}^*}{\| \mathbf{z}-\mathbf{x}^* \|_{\mathbb{X}}})$ the directional derivative of $F(\mathbf{y})=\sum _{i=1}^{k}\left\|\mathbf{x}_{i}-\mathbf{y}\right\|_{\mathbb{X}}+\langle\mathbf{u} ,\mathbf{x}_{i}-\mathbf{y}\rangle$ in direction $\frac{\mathbf{z}-\mathbf{x}^*}{\| \mathbf{z}-\mathbf{x}^* \|_{\mathbb{X}}}$ at point $\mathbf{x}^*$. Then, we have
	\begin{equation*}
		DF\Big(\mathbf{x}^*,\frac{\mathbf{z}-\mathbf{x}^*}{\| \mathbf{z}-\mathbf{x}^* \|_{\mathbb{X}}}\Big)=-\sum _{j:\mathbf{x}_j\neq \mathbf{x}^*}^{k}\frac{\langle g_j^*, \mathbf{z}-\mathbf{x}^*\rangle}{\| \mathbf{z}-\mathbf{x}^* \|_{\mathbb{X}}}+\sum _{i=1}^{k}\mathbf{1}_{\{\mathbf{x}_{i}=\mathbf{x}^*\}}-k\Big\langle\mathbf{u} ,\frac{\mathbf{z}-\mathbf{x}^*}{\| \mathbf{z}-\mathbf{x}^* \|_{\mathbb{X}}}\Big\rangle
	\end{equation*}
	where $g_j^*$ is a subdifferential of the norm $\|\cdot\|_{\mathbb{X}}$ evaluated at point $\mathbf{x}_{j}-\mathbf{x}^*$. By definition of $\mathbf{x}^*$, $DF(\mathbf{x}^*,\mathbf{v})\geq 0$ for any $\mathbf{v}\in\mathbb{X}$. Since
	\begin{equation*}
		-k\Big\langle\mathbf{u} ,\frac{\mathbf{z}-\mathbf{x}^*}{\| \mathbf{z}-\mathbf{x}^* \|_{\mathbb{X}}}\Big\rangle\leq k\|\mathbf{u}\|_{\mathbb{X}^*}
	\end{equation*}
	we obtain 
	\begin{equation*}
		DF\Big(\mathbf{x}^*,\frac{\mathbf{z}-\mathbf{x}^*}{\| \mathbf{z}-\mathbf{x}^* \|_{\mathbb{X}}}\Big)< -(1-\nu)k\Big(1-\frac{2}{C_\nu}\Big)+\nu k+\|\mathbf{u}\|_{\mathbb{X}^*} k
	\end{equation*}
	which is $\leq$ 0 whenever $C_\nu\geq\dfrac{2(1-\nu)}{1-2\nu-\|\mathbf{u}\|_{\mathbb{X}^*}}$, which is a contradiction.
	
\end{proof}
\begin{proof}[Proof of Proposition \ref{Extension-Theorem3.1}]
	This follows from the same arguments as the ones used in Theorem 3.1 in \cite{Minsker}, using Lemma \ref{Extension-of-Lemma2.1} in place of Lemma 2.1 in \cite{Minsker}.
\end{proof}

\begin{proof}[Proof of Theorem \ref{t-1-q}]
	The convergence in probability follows from Theorem \ref{Extension-Theorem3.1}. Further, it is possible to see that the quantile function possesses an equivariance property under a certain class of affine transformations, in particular we have that 
	\begin{equation*}
		\mathbf{q}_\mathbf{u}(c(x_1-s,...,x_k-s))=c(\mathbf{q}_\mathbf{u}(x_1,...,x_k)-s)
	\end{equation*}
	for every $c\geq0$ and $x_1,...,x_k,s\in \mathbb{X}$. Thus,
	\begin{equation*}
		\mathbf{q}_\mathbf{u}\Big(a_{\lfloor n/k\rfloor}\big((Z^{(1)}_{\lfloor n/k\rfloor},...,Z^{(k)}_{\lfloor n/k\rfloor})-z\big)\Big)
	\end{equation*}
	\begin{equation*}
		=a_{\lfloor n/k\rfloor}\Big(\mathbf{q}_\mathbf{u}\big((Z^{(1)}_{\lfloor n/k\rfloor},...,Z^{(k)}_{\lfloor n/k\rfloor})\big)-z\Big)=a_{\lfloor n/k\rfloor}(\theta_{\mathbf{u},n,k}-z).
	\end{equation*}
	If $Y$ is continuous then $\mathbb{P}(Y_1,...,Y_k\in\mathbb{X}^k\setminus\mathcal{X}_k)=0$. Further, if $\mathbb{X}$ is smooth and $\|\mathbf{u}\|_{\mathbb{X}^*}\notin\{1-\frac{2j}{k},j=1,...,\lfloor\frac{k}{2}\rfloor\}$, then by Theorem \ref{thm-continuity} we obtain that $\mathbf{q}_\mathbf{u}$ is continuous on the whole $\mathbb{X}^k$. Observe that when $\mathbb{X}=\mathbb{R}$ we have that $\mathbf{q}_\mathbf{u}$ reduces to $q_\alpha$ and that $q_\alpha$ is continuous on $\mathbb{R}^k$ for any value of $\alpha$ by Lemma \ref{lem-con-quant}. Therefore, by the continuous mapping theorem we obtain the stated convergence in distribution.
\end{proof}
\begin{proof}[Proof of Proposition \ref{t-k-to-infinity}]
	
	This follows from Theorem \ref{Extension-Theorem3.1}.
	
\end{proof}

\begin{proof}[Proof of Lemma \ref{Lemma-v}]
	We prove the case of $p\geq 1$, because the case $p=0$ is trivial. A necessary and sufficient condition for $\mathbf{x}^*$ to be a geometric quantile is that the directional derivative of 
	\begin{equation*}
		F(\mathbf{y})=\sum _{i=1}^{k}\left\|\mathbf{x}_{i}-\mathbf{y}\right\|_{\mathbb{X}}+\langle\mathbf{u} ,\mathbf{x}_{i}-\mathbf{y}\rangle
	\end{equation*}
	evaluated at $\mathbf{x}^*$ in direction $\mathbf{z}\in\mathbb{X}$ is non-negative for any $\mathbf{z}\in\mathbb{X}$. Denote by $DF(\mathbf{x}^*,\mathbf{z})$ the directional derivative of $F(\mathbf{y})$ in direction $\mathbf{z}$ at point $\mathbf{x}^*$. Without loss of generality we consider any $\mathbf{z}\in\mathbb{X}$ such that $\|\mathbf{z}\|_{\mathbb{X}}=1$. Then, we have
	\begin{equation*}
		DF(\mathbf{x}^*,\mathbf{z})=-\sum _{j:\mathbf{x}_j\neq \mathbf{x}^*}^{k}\langle g_j^*, \mathbf{z}\rangle+\sum _{i=1}^{k}\mathbf{1}_{\{\mathbf{x}_{i}=\mathbf{x}^*\}}-k\langle\mathbf{u} ,\mathbf{z}\rangle\geq 0.
	\end{equation*}
	Consider $\tilde{\mathbf{x}}_1,...,\tilde{\mathbf{x}}_k\in\mathbb{X}$ as in the statement. We need to show that, for every $\mathbf{z}\in\mathbb{X}$ such that $\|\mathbf{z}\|_{\mathbb{X}}=1$,
	\begin{equation}\label{v-point}
		DF(\mathbf{x}^*,\mathbf{z})=-\sum _{j:\tilde{\mathbf{x}}_j\neq \mathbf{x}^*}^{k}\langle g_j^*, \mathbf{z}\rangle+\sum _{i=1}^{k}\mathbf{1}_{\{\tilde{\mathbf{x}}_{i}=\mathbf{x}^*\}}-k\langle\mathbf{v} ,\mathbf{z}\rangle\geq 0
	\end{equation}
	for $\mathbf{v}\in \mathbb{X}^*$ as in the statement. So consider $\mathbf{v}$ given by (\ref{v-statement}), which can be rewritten
	\begin{equation*}
		\mathbf{v}=\frac{\mathbf{u}-\mathbf{v}}{2p}\sum _{i=1}^{k}\mathbf{1}_{\{\tilde{\mathbf{x}}_{i}=\mathbf{x}^*\}}-\frac{1}{k}\sum _{j:\tilde{\mathbf{x}}_j\neq \mathbf{x}^*}^{k}g_j^*
	\end{equation*}
	\begin{equation*}
		=\frac{1}{k}\Bigg(\frac{k(\mathbf{u}-\mathbf{v})}{2p}\sum _{i=1}^{k}\mathbf{1}_{\{\tilde{\mathbf{x}}_{i}=\mathbf{x}^*\}}-\sum _{j:\tilde{\mathbf{x}}_j\neq \mathbf{x}^*}^{p}g_j^*-\sum _{j>p:\mathbf{x}_j\neq \mathbf{x}^*}^{k}g_j^*\Bigg)
	\end{equation*}
	Using $DF(\mathbf{x}^*,\mathbf{z})\geq0$ and $DF(\mathbf{x}^*,-\mathbf{z})\geq0$, we get
	\begin{equation*}
		-\sum _{i=1}^{k}\mathbf{1}_{\{\mathbf{x}_{i}=\mathbf{x}^*\}}-\sum _{j:\mathbf{x}_j\neq \mathbf{x}^*}^{k}\langle g_j^*, \mathbf{z}\rangle \leq k\langle\mathbf{u} ,\mathbf{z}\rangle\leq \sum _{i=1}^{k}\mathbf{1}_{\{\mathbf{x}_{i}=\mathbf{x}^*\}}-\sum _{j:\mathbf{x}_j\neq \mathbf{x}^*}^{k}\langle g_j^*, \mathbf{z}\rangle
	\end{equation*}
	and
	\begin{equation*}
		-\sum _{i=1}^{k}\mathbf{1}_{\{\mathbf{x}_{i}=\mathbf{x}^*\}}-\sum _{j:\mathbf{x}_j\neq \mathbf{x}^*}^{p}\langle g_j^*, \mathbf{z}\rangle+\sum _{j:\tilde{\mathbf{x}}_j\neq \mathbf{x}^*}^{p}\langle g_j^*, \mathbf{z}\rangle -\frac{k\langle\mathbf{u}-\mathbf{v} ,\mathbf{z}\rangle}{2p}\sum _{i=1}^{k}\mathbf{1}_{\{\tilde{\mathbf{x}}_{i}=\mathbf{x}^*\}}
	\end{equation*}
	\begin{equation*}
		\leq k\langle\mathbf{u}-\mathbf{v} ,\mathbf{z}\rangle
	\end{equation*}
	\begin{equation*}
		\leq \sum _{i=1}^{k}\mathbf{1}_{\{\mathbf{x}_{i}=\mathbf{x}^*\}}-\sum _{j:\mathbf{x}_j\neq \mathbf{x}^*}^{p}\langle g_j^*, \mathbf{z}\rangle+\sum _{j:\tilde{\mathbf{x}}_j\neq \mathbf{x}^*}^{p}\langle g_j^*, \mathbf{z}\rangle -\frac{k\langle\mathbf{u}-\mathbf{v} ,\mathbf{z}\rangle}{2p}\sum _{i=1}^{k}\mathbf{1}_{\{\tilde{\mathbf{x}}_{i}=\mathbf{x}^*\}},
	\end{equation*}
	which leads to
	\begin{equation*}
		\Big(k+\frac{k}{2p}\sum _{i=1}^{k}\mathbf{1}_{\{\tilde{\mathbf{x}}_{i}=\mathbf{x}^*\}}\Big)\|\mathbf{u}-\mathbf{v}\|_{\mathbb{X}^*}	\leq \sum _{i=1}^{k}\mathbf{1}_{\{\mathbf{x}_{i}=\mathbf{x}^*\}}+\sum _{i=1}^{p}\mathbf{1}_{\{\mathbf{x}_{i}\neq\mathbf{x}^*\}}+\sum _{i=1}^{p}\mathbf{1}_{\{\tilde{\mathbf{x}}_{i}\neq\mathbf{x}^*\}}
	\end{equation*}
	\begin{equation*}
		\leq \sum _{i=1}^{k}\mathbf{1}_{\{\tilde{\mathbf{x}}_{i}=\mathbf{x}^*\}}+2p,
	\end{equation*}
	where we used 
	\begin{equation*}
		\sum _{i=1}^{k}\mathbf{1}_{\{\mathbf{x}_{i}=\mathbf{x}^*\}}+\sum _{i=1}^{p}\mathbf{1}_{\{\mathbf{x}_{i}\neq\mathbf{x}^*\}}=\sum _{i=1}^{k}\mathbf{1}_{\{\tilde{\mathbf{x}}_{i}=\mathbf{x}^*\}}+\sum _{i=1}^{p}\mathbf{1}_{\{\tilde{\mathbf{x}}_{i}\neq\mathbf{x}^*\}}
	\end{equation*}
	and $\sum _{i=1}^{p}\mathbf{1}_{\{\tilde{\mathbf{x}}_{i}\neq\mathbf{x}^*\}}\leq p$. Therefore, we obtain
	\begin{equation*}
		\|\mathbf{u}-\mathbf{v}\|_{\mathbb{X}^*}	\leq \frac{2p}{k}.
	\end{equation*}
	It remains to show that (\ref{v-point}) is satisfied. However, this simply follows from 
	\begin{equation*}
		\frac{k(\|\mathbf{u}-\mathbf{v}\|_{\mathbb{X}^*}	)}{2p}\leq 1.
	\end{equation*}
\end{proof}
\section*{Proof of Theorem \ref{thm-Bahadur}}
In order to prove Theorem \ref{thm-Bahadur} we need to introduce and prove various auxiliary results. In this section in order to lighten the notation we define $X_{n,(n)}:=W_{n,k_n,(n)}$ and $X_n:=W_{n,k_n}$, and we do not explicitly specify the space in the norm notation, that is if $\mathbf{x}\in\mathcal{Z}_n$ then we write $\|\mathbf{x}\|$ for $\|\mathbf{x}\|_{\mathcal{Z}_n}$ and similarly if $\mathbf{x}\in\mathbb{X}$.

\begin{proof}[Proof of Lemma \ref{lem-CCZ}]
	We first show that $X_{n,(n)}\stackrel{d}{\to}Y$. We have that $X_{n,(n)}=X_{n,(n)}-X_n+X_n$. Since $X_n\stackrel{d}{\to}Y$, it remains to show that $X_{n,(n)}-X_n\stackrel{p}{\to}0$. Fix $\varepsilon>0$. Since $X_n\stackrel{d}{\to}Y$ the sequence $X_n$ is tight and so for every $\delta>0$ there exists $M_\delta>0$ such that 
	\begin{equation*}
		\sup_{n\in\mathbb{N}}\mathbb{P}(\|X_n\|> M_\delta )<\delta/4.
	\end{equation*}
	Further, there exists an $m'\in\mathbb{N}$ large enough such that
	\begin{equation*}
		\mathbb{P}(\|Y-Y_{(m')}\| >\varepsilon, \|Y\|\leq M_\delta )<\delta/4.
	\end{equation*}
	By the continuous mapping theorem $(\| X_{n}-X_{n,(m)}\|,\|X_n\|)\stackrel{d}{\to}(\| Y-Y_{(m)}\|,\|Y\|)$ as $n\to\infty$, for every $m\in\mathbb{N}$. Thus, there exists an $n'\in\mathbb{N}$ such that 
	\begin{equation*}
		\mathbb{P}(\|X_{n}-X_{n,(m')}\| >\varepsilon, \|X_n\|\leq M_\delta )<\delta/2
	\end{equation*}
	for any $n\geq n'$. Using that for any $m_1\geq m_2$ we have $\|X_n-X_{n,(m_1)}\| \leq \|X_n-X_{n,(m_2)}\| $, we obtain that 
	\begin{equation*}
		\mathbb{P}(\|X_{n}-X_{n,(n)}\| >\varepsilon, \|X_n\|\leq M_\delta )<\delta/2
	\end{equation*}
	for any $n\geq \max(n',m')$. Hence, 
	\begin{equation*}
		\mathbb{P}(\|X_{n}-X_{n,(n)}\| >\varepsilon)
	\end{equation*}
	\begin{equation*}
		=\mathbb{P}(\|X_{n}-X_{n,(n)}\| >\varepsilon, \|X_n\|\leq M_\delta )+\mathbb{P}(\|X_{n}-X_{n,(n)}\| >\varepsilon, \|X_n\|> M_\delta )
	\end{equation*}
	\begin{equation*}
		<\delta/2+\delta/4<\delta
	\end{equation*}
	for any $n\geq \max(n',m')$. Thus, $X_{n,(n)}-X_n\stackrel{p}{\to}0$.

	From the proof of Proposition 2.1 in \cite{CCZ13} we have that there exist two orthogonal elements $\mathbf{v}_1,\mathbf{v}_2\in\mathbb{X}$ and a $K>0$ large enough such that $Var(\langle \mathbf{v},Y\mathbf{1}_{\| Y\|\leq K}\rangle)\geq c$ for some $c>0$, where $\mathbf{v}=\mathbf{v}_1,\mathbf{v}_2$. Wlog $\|\mathbf{v}_1\|=\|\mathbf{v}_2\|=1$. Hence, any $\mathbf{v}\in span(\mathbf{v}_1,\mathbf{v}_2)$ with $\|\mathbf{v}\|=1$ can be written as $\mathbf{v}=\cos(t)\mathbf{v}_1+\sin(t)\mathbf{v}_2$ for some $t\in[0,2\pi]$. Since $X$ is not concentrated on a straight line we have that 
	\begin{equation*}
		\min_{\mathbf{v}\in span(\mathbf{v}_1,\mathbf{v}_2),\|\mathbf{v}\|}Var(\langle \mathbf{v},X\mathbf{1}_{\| Y\|\leq K}\rangle)
	\end{equation*}
	\begin{equation*}
		=\min_{t\in[0,2\pi]}Var(\langle \cos(t)\mathbf{v}_1+\sin(t)\mathbf{v}_2,X\mathbf{1}_{\| Y\|\leq K}\rangle)\geq c'
	\end{equation*}
	for some $c'>0$. We remark that we are using the minimum instead of the infimum to stress that, since the function is continuous on compact spaces, the function attains its minimum. 
	
	Since $X_{n,(n)}\stackrel{d}{\to}Y$ and $Y$ has an atomless distribution, by the Portmanteau theorem (see Theorem 13.16 in \cite{Klenke})
	\begin{equation*}
		Var(\langle \mathbf{v},X_{n,(n)}\mathbf{1}_{\| X_{n,(n)}\|\leq K}\rangle)
	\end{equation*} 
	\begin{equation*}
		\geq Var(\langle \mathbf{v},X_{n,(n)}\mathbf{1}_{\| X_{n,(n)}\|\leq K}\rangle)-Var(\langle \mathbf{v},Y\mathbf{1}_{\| Y\|\leq K}\rangle)+c'
	\end{equation*}
	\begin{equation*}
		=\cos(t)^{2}(\mathbb{E}[\langle \mathbf{v}_1,X_{n,(n)}\rangle^{2}\mathbf{1}_{\| X_{n,(n)}\|\leq K}]-\mathbb{E}[\langle \mathbf{v}_1,Y\rangle^{2}\mathbf{1}_{\| Y\|\leq K}])
	\end{equation*}
	\begin{equation*}
		+\sin(t)^{2}(\mathbb{E}[\langle \mathbf{v}_2,X_{n,(n)}\rangle^{2}\mathbf{1}_{\| X_{n,(n)}\|\leq K}]-\mathbb{E}[\langle \mathbf{v}_2,Y\rangle^{2}\mathbf{1}_{\| Y\|\leq K}])
	\end{equation*}
	\begin{equation*}
		+2\cos(t)\sin(t)(\mathbb{E}[\langle \mathbf{v}_1,X_{n,(n)}\rangle  \langle \mathbf{v}_2,X_{n,(n)}\rangle  \mathbf{1}_{\| X_{n,(n)}\|\leq K}]-\mathbb{E}[\langle \mathbf{v}_1,Y\rangle  \langle \mathbf{v}_2,Y\rangle  \mathbf{1}_{\| Y\|\leq K}])
	\end{equation*}
	\begin{equation*}
		-\Big(\cos(t)(\mathbb{E}[\langle \mathbf{v}_1,X_{n,(n)}\rangle\mathbf{1}_{\| X_{n,(n)}\|\leq K}]-\mathbb{E}[\langle \mathbf{v}_1,Y\rangle\mathbf{1}_{\| Y\|\leq K}])
	\end{equation*}
	\begin{equation*}
		+\sin(t)(\mathbb{E}[\langle \mathbf{v}_2,X_{n,(n)}\rangle\mathbf{1}_{\| X_{n,(n)}\|\leq K}]-\mathbb{E}[\langle \mathbf{v}_2,Y\rangle\mathbf{1}_{\| Y\|\leq K}])\Big)+c'.
	\end{equation*}
	Thus, there is an $N$ large enough such that for every $\mathbf{v}\in span(\mathbf{v}_1,\mathbf{v}_2)$ with $\|\mathbf{v}\|=1$, namely for every $t\in[0,2\pi]$, 
	\begin{equation*}
		Var(\langle \mathbf{v},X_{n,(n)}\mathbf{1}_{\| X_{n,(n)}\|\leq K}\rangle)\geq c'/2
	\end{equation*}
	for every $n>N$. From this inequality we easily obtain the result by the arguments of Proposition 2.1 in \cite{CCZ13}.
\end{proof}
We notice that $U_n$ in the statement of Theorem  \ref{thm-Bahadur} is a well-defined $\mathcal{Z}_n$-valued random variable and so $\mathbf{q}_{U_n}(Z^{(1)}_{\lfloor n/k_n\rfloor,(n)},...,Z^{(k_n)}_{\lfloor n/k_n\rfloor,(n)})$ is also a well-defined random variable. Further, by Lemma \ref{Lemma-v}, we have
\begin{equation*}
	T_{\mathbf{u}_{(n)},n,k_n}=\mathbf{q}_{U_n}(Z^{(1)}_{\lfloor n/k_n\rfloor,(n)},...,Z^{(k_n)}_{\lfloor n/k_n\rfloor,(n)}).
\end{equation*}
Thus, by the equivariance property of the geometric quantile we have
\begin{equation*}
	a_{\lfloor n/k_n\rfloor}(T_{\mathbf{u}_{(n)},n,k_n} -\bm\theta_{0,(n)}) 
\end{equation*}	
\begin{equation*}
	=a_{\lfloor n/k_n\rfloor}(\mathbf{q}_{U_n}(Z^{(1)}_{\lfloor n/k_n\rfloor,(n)},...,Z^{(k_n)}_{\lfloor n/k_n\rfloor,(n)}) -\bm\theta_{0,(n)}) 
\end{equation*}
\begin{equation*}
	=\mathbf{q}_{U_n}(a_{\lfloor n/k_n\rfloor}Z^{(1)}_{\lfloor n/k_n\rfloor,(n)}-\bm\theta_{0,(n)},...,a_{\lfloor n/k_n\rfloor}Z^{(k_n)}_{\lfloor n/k_n\rfloor,(n)}-\bm\theta_{0,(n)}).
\end{equation*}
\begin{lem}\label{lem-A-5}
	Let $\mathbb{X}$ be a separable Hilbert space. Let Assumptions \ref{A2} and \ref{A-B} hold and let $C>0.$ Then there exist $b,B\in(0,\infty)$ such that for all sufficiently large $n$ and any $\mathbf{z},\mathbf{h}\in\mathcal{Z}_n$ with $\|\mathbf{z}-\mathbf{r}_{\lfloor n/k_n\rfloor,\mathbb{E}[U_n]}\|\leq C$, we have
	\begin{equation}
		\bigg\|\mathbb{E}\bigg[\frac{\mathbf{z}-X_{n,(n)}}{\|\mathbf{z}-X_{n,(n)}\|}-U_n \bigg]\bigg\|\geq b\|\mathbf{z}-X_{n,(n)}\|
	\end{equation}
	\begin{equation}
		\sup\limits_{\|\mathbf{h}\|=\mathbf{v}=1}|J_{n,\mathbf{z}}(\mathbf{h},\mathbf{v})-J_{n,\mathbf{r}_{\lfloor n/k_n\rfloor,\mathbb{E}[U_n]}}(\mathbf{h},\mathbf{v})|\leq B \|\mathbf{z}-X_{n,(n)}\|
	\end{equation}
	\begin{equation}
		\bigg\|\mathbb{E}\bigg[\frac{\mathbf{z}-X_{n,(n)}}{\|\mathbf{z}-X_{n,(n)}\|}-U_n \bigg]-\tilde{J}_{n,\mathbf{r}_{\lfloor n/k_n\rfloor,\mathbb{E}[U_n]}}(\mathbf{z}-\mathbf{r}_{\lfloor n/k_n\rfloor,\mathbb{E}[U_n]})\bigg\| \leq B \|\mathbf{z}-X_{n,(n)}\|^{2}.
	\end{equation}
\end{lem}
\begin{proof}
	Using Lemma \ref{lem-CCZ}, it follows from the same arguments of Lemma A.5 in \cite{ChCh2014}.
\end{proof}
In the following result and proof we generalize Theorem 3.1.1 in \cite{Chaudhuri} and the arguments of its proof.
\begin{lem}\label{lem-bounded-quantile-a.s.}
	Let $\mathbb{X}$ be a separable Hilbert space and let Assumptions \ref{A2} and \ref{A-B} hold. Then, there exists a constant $K_1$ s.t.~$\| 	a_{\lfloor n/k_n\rfloor}(T_{\mathbf{u}_{(n)},n,k_n} -\bm\theta_{0,(n)}) \|\leq K_1$ for all sufficiently large $n$ almost surely.
\end{lem}
\begin{proof}
	Since $X_n\stackrel{d}{\to}X$ by continuity of the norm we have $\|X_n\|\stackrel{d}{\to}\|X\|$, and since the probability measure of $X$ is non-atomic we have continuity of the distribution of $\|X\|$. Combining these we have, for any constant $c>0$, $|\mathbb{P}(\|X_n\|>c)-\mathbb{P}(\|X\|>c)|\to0$ as $n\to\infty$. Choose $\delta\in(0,(1-\|\mathbf{u}\|)/12)$ and $K_1$ such that
	\begin{equation*}
		\mathbb{P}(\|X\|>C_2K_1)\leq \delta/2,
	\end{equation*}
	where $C_2=\frac{4}{(3+\|\mathbf{u}\|)(3-2\|\mathbf{u}\|-\|\mathbf{u}\|^{2})}$. Then, there is an $m$ large enough such that 
	\begin{equation*}
		\mathbb{P}(\|X_{n}^{(1)}\|>C_2K_1)\leq \delta,
	\end{equation*}
	and hence
	\begin{equation*}
		\mathbb{P}(\|X_{n,(n)}^{(1)}\|>C_2K_1)\leq \delta,
	\end{equation*}
	for every $n>m$. We used the fact that for any random variable $Z$ in a separable Banach space $\mathbb{Z}$ is tight and so for any $\varepsilon>0$ there exists a constant $K$ such that $\mathbb{P}(\|Z\|_{\mathbb{Z}}>K)<\varepsilon$. By Hoeffding's inequality
	\begin{equation*}
		\mathbb{P}\bigg(\bigg|\frac{1}{k_n}\sum_{i=1}^{k_n}\mathbf{1}_{\{\|X_{n,(n)}^{(i)}\|>C_2K_1\}}-	\mathbb{P}(\|X_{n,(n)}^{(1)}\|>C_2K_1)\bigg|\geq \delta\bigg)\leq 2\exp(-k_n\delta^2),
	\end{equation*}
	and so by the Borel-Cantelli lemma, almost surely, we have
	\begin{equation}\label{Borel-Cantelli}
		\frac{1}{k_n}\sum_{i=1}^{k_n}\mathbf{1}_{\{\|X_{n,(n)}^{(i)}\|>C_2K_1\}}\leq 2\delta
	\end{equation}
	for all $n$ sufficiently large. Since $\|U_{n}-\mathbf{u}_{(n)}\|\to0$ almost surely, then almost surely we have 
	\begin{equation}\label{Borel-Cantelli2}
		\|U_{n}-\mathbf{u}_{(n)}\|<\frac{1-\|\mathbf{u}\|}{4}
	\end{equation}
	for every $n$ sufficiently large.
	
	If (\ref{Borel-Cantelli}) holds, for any $\mathbf{z}\in\mathcal{Z}_n$
	\begin{equation*}
		\bigg|\frac{1}{k_n}\sum_{i=1}^{k_n}(\|X_{n,(n)}^{(i)}-\mathbf{z}\|-	\|X_{n,(n)}^{(i)}\|)\mathbf{1}_{\{\|X_{n,(n)}^{(i)}\|> C_2K_1\}}\bigg|\leq2\delta\|\mathbf{z}\|.
	\end{equation*}
	Now, note that for any $\mathbf{z}\in\mathcal{Z}_n$
	\begin{equation*}
		\|X_{n,(n)}^{(i)}-\mathbf{z}\|-	\|X_{n,(n)}^{(i)}\|> \|\mathbf{z}\|\bigg(\frac{1- \|\mathbf{u}\|}{2}\bigg)
		\Leftrightarrow		\|X_{n,(n)}^{(i)}\|< C_2  \|\mathbf{z}\|.
	\end{equation*}
	If (\ref{Borel-Cantelli}) holds, for any $\mathbf{z}\in\mathcal{Z}_n$ such that $\|\mathbf{z}\|>K_1$ we have
	\begin{equation*}
		\frac{1}{k_n}\sum_{i=1}^{k_n}(\|X_{n,(n)}^{(i)}-\mathbf{z}\|-	\|X_{n,(n)}^{(i)}\|)\mathbf{1}_{\{\|X_{n,(n)}^{(i)}\|\leq C_2K_1\}}>\|\mathbf{z}\|\bigg(\frac{1- \|\mathbf{u}\|}{2}\bigg)(1-2\delta).
	\end{equation*}
	Since when (\ref{Borel-Cantelli2}) holds we have
	\begin{equation*}
		\|\mathbf{z}\|\bigg(\frac{1- \|\mathbf{u}\|}{2}\bigg)(1-2\delta)-2\delta\|\mathbf{z}\|-\|U_{n}\|\|\mathbf{z}\|>0,
	\end{equation*}
	we conclude that when (\ref{Borel-Cantelli}) and (\ref{Borel-Cantelli2}) hold for any $\mathbf{z}\in\mathcal{Z}_n$ such that $\|\mathbf{z}\|>K_1$,
	\begin{equation*}
		\frac{1}{k_n}\sum_{i=1}^{k_n}\|X_{n,(n)}^{(i)}-\mathbf{z}\|-\langle U_{n},X_{n,(n)}^{(i)}-\mathbf{z} \rangle>\frac{1}{k_n}\sum_{i=1}^{k_n}	\|X_{n,(n)}^{(i)}\|-\langle U_{n},X_{n,(n)}^{(i)} \rangle
	\end{equation*}
	This shows that the minimizer of $\frac{1}{k_n}\sum_{i=1}^{k_n}\|X_{n,(n)}^{(i)}-\mathbf{z}\|-\langle U_{n},X_{n,(n)}^{(i)}-\mathbf{z} \rangle$ lies almost surely inside the ball of radius $K_1$ for all $n$ sufficiently large.		
\end{proof}
\begin{pro}\label{pro-Big-O-delta_n}
	Let $\mathbb{X}$ be a separable Hilbert space and let Assumptions \ref{A2} and \ref{A-B} hold. Let $\rho\in(0,1/2]$ be s.t.~$d_n/k_n^{1-2\rho}\to c$ as $n\to\infty$ for some $c>0$. Then, $\| 		a_{\lfloor n/k_n\rfloor}(T_{\mathbf{u}_{(n)},n,k_n} -\bm\theta_{0,(n)})-\mathbf{r}_{\lfloor n/k_n\rfloor,\mathbb{E}[U_n]} \|=O(\delta_n)$ as $n\to\infty$ almost surely, where $\delta_n\sim\sqrt{\ln k_n}/k_n^{\rho}$.
\end{pro}
\begin{proof}
	For the sake of clarity in this proof we let $\hat{Y}_{n,U_n}:=a_{\lfloor n/k_n\rfloor}(T_{\mathbf{u}_{(n)},n,k_n} -\bm\theta_{0,(n)})$. By Lemma \ref{lem-bounded-quantile-a.s.} we deduce the existence of $C_3>0$ satisfying $\| 	\hat{Y}_{n,U_n}-\mathbf{r}_{\lfloor n/k_n\rfloor,\mathbb{E}[U_n]}\|\leq C_3$ for all sufficiently large $n$ almost surely. Building on the arguments of the proof of Proposition A.6 in \cite{ChCh2014}, we define \begin{equation*}
		G_n=\Bigg\{\mathbf{r}_{\lfloor n/k_n\rfloor,\mathbb{E}[U_n]}+\sum_{j\leq d_n}\beta_j\phi_j\,:\,k_n^4\beta_j\in\mathbb{Z}\,\,\textnormal{and}\,\,\| \sum_{j\leq d_n}\beta_j\phi_j\|\leq C_3\Bigg\}.
	\end{equation*}
	Define
	\begin{equation*}
		E_n=\Bigg\{\max_{\mathbf{z}\in G_n} \bigg\| \frac{1}{k_n}\sum_{i=1}^{k_n} \bigg(\frac{\mathbf{z}-X^{(i)}_{n,(n)} }{\| \mathbf{z}-X^{(i)}_{n,(n)} \|}-U_n\bigg)- \mathbb{E}\bigg[\frac{\mathbf{z}-X^{(i)}_{n,(n)} }{\| \mathbf{z}-X^{(i)}_{n,(n)} \|}-U_n  \bigg]\bigg\|  \leq C_4\delta_n\Bigg \}.
	\end{equation*}
	where $C_4>0$. By Bernstein's inequality and using that $\| \frac{\mathbf{z}-X^{(i)}_{n,(n)} }{\| \mathbf{z}-X^{(i)}_{n,(n)} \|}-U_n\|\leq 2$ for all $\mathbf{z}\in G_n$ and $n\in\mathbb{N}$, we have that $P(E_n^c)\leq 2(3C_3k_n^4)^{d_n}\exp(-k_n C^2_{5}\delta_n^2/2)$ for some $C_5>0$. Thus, using the definition of $\delta_n$, we can choose $C_5$ such that $\sum_{n=1}^{\infty}\mathbb{P}(E_n^c)<\infty$. Hence,
	\begin{equation}\label{A.10}
		\mathbb{P}(\textnormal{$E_n$ occurs for all sufficiently large $n$})=1.
	\end{equation}
	By Assumption \ref{A-B} and by the Markov inequality  $\mathbb{P}(\|\mathbf{z}-X_{n,(n)}\|\leq k_n^{-2})\leq M_n k_n^{-2}\leq C_6\delta_n^2$ for any $\mathbf{z}\in G_n$ and all $n$ large enough. By Bernstein's inequality there exists a $C_7>0$ s.t.~$(F_n^c)\leq (3C_3k_n^4)^{d_n}\exp(-k_n C^2_{7}\delta_n^2)$ for all sufficiently large $n$, where 
	\begin{equation*}
		F_n=\Bigg\{\max_{\mathbf{z}\in G_n} \sum_{i=1}^{k_n}\mathbf{1}_{ \{ \|\mathbf{z}-X^{(i)}_{n,(n)}\|\leq k_n^{-2} \} }  \leq C_6 k_n\delta_n^2\Bigg\}.
	\end{equation*}
	Thus, for $C_7$ large enough
	\begin{equation}\label{A.11}
		\mathbb{P}(\textnormal{$F_n$ occurs for all sufficiently large $n$})=1.
	\end{equation}
	Let $\bar{\mathbf{r}}_{\lfloor n/k_n\rfloor,\mathbb{E}[U_n]}$ be a point in $G_n$ nearest to $\hat{Y}_{n,U_n}$. Then, $\| \hat{Y}_{n,U_n}-\bar{\mathbf{r}}_{\lfloor n/k_n\rfloor,\mathbb{E}[U_n]}\|\leq C_8d_n/k_n^4$ for $C_8>0$. Using (\ref{A.11}) and
	\begin{equation}\label{A.12}
		\bigg\|\dfrac{\hat{Y}_{n,U_n}-X^{(i)}_{n,(n)}}{\| \hat{Y}_{n,U_n}-X^{(i)}_{n,(n)}\|} -\dfrac{\bar{\mathbf{r}}_{\lfloor n/k_n\rfloor,\mathbb{E}[U_n]}-X^{(i)}_{n,(n)}}{\| \bar{\mathbf{r}}_{\lfloor n/k_n\rfloor,\mathbb{E}[U_n]}-X^{(i)}_{n,(n)}\|} \bigg\|\leq \dfrac{2\| \hat{Y}_{n,U_n}-\bar{\mathbf{r}}_{\lfloor n/k_n\rfloor,\mathbb{E}[U_n]}\|}{\| \bar{\mathbf{r}}_{\lfloor n/k_n\rfloor,\mathbb{E}[U_n]}-X^{(i)}_{n,(n)}\|},
	\end{equation}
	we have
	\begin{equation*}
		\bigg\|\frac{1}{k_n}\sum_{i=1}^{k_n}	\dfrac{\bar{\mathbf{r}}_{\lfloor n/k_n\rfloor,\mathbb{E}[U_n]}-X^{(i)}_{n,(n)}}{\| \bar{\mathbf{r}}_{\lfloor n/k_n\rfloor,\mathbb{E}[U_n]}-X^{(i)}_{n,(n)}\|} -U_n\bigg\| \leq \bigg\|\frac{1}{k_n}\sum_{i=1}^{k_n}	\dfrac{\hat{Y}_{n,U_n}-X^{(i)}_{n,(n)}}{\|\hat{Y}_{n,U_n}-X^{(i)}_{n,(n)}\|} -U_n\bigg\|
	\end{equation*}
	\begin{equation*}
		+\bigg\|\frac{1}{k_n}\sum_{i=1}^{k_n}	\dfrac{\bar{\mathbf{r}}_{\lfloor n/k_n\rfloor,\mathbb{E}[U_n]}-X^{(i)}_{n,(n)}}{\| \bar{\mathbf{r}}_{\lfloor n/k_n\rfloor,\mathbb{E}[U_n]}-X^{(i)}_{n,(n)}\|} -	\dfrac{\hat{Y}_{n,U_n}-X^{(i)}_{n,(n)}}{\|\hat{Y}_{n,U_n}-X^{(i)}_{n,(n)}\|} \bigg\|
	\end{equation*}
	\begin{equation*}
		\leq \bigg\|\frac{1}{k_n}\sum_{i=1}^{k_n}	\dfrac{\hat{Y}_{n,U_n}-X^{(i)}_{n,(n)}}{\|\hat{Y}_{n,U_n}-X^{(i)}_{n,(n)}\|} -U_n\bigg\|+2C_8d_nk_n^{-2}
	\end{equation*}
	\begin{equation*}
		+\frac{2}{k_n}\sum_{i=1}^{k_n}\mathbf{1}_{\{\| \bar{\mathbf{r}}_{\lfloor n/k_n\rfloor,\mathbb{E}[U_n]}-X^{(i)}_{n,(n)}\|\leq k_n^{-2}\}}
	\end{equation*}
	\begin{equation}\label{A.13}
		\leq \bigg\|\frac{1}{k_n}\sum_{i=1}^{k_n}	\dfrac{\hat{Y}_{n,U_n}-X^{(i)}_{n,(n)}}{\|\hat{Y}_{n,U_n}-X^{(i)}_{n,(n)}\|} -U_n\bigg\|+C_9\delta_n^2,
	\end{equation}
	for some $C_9>0$. From similar arguments as in the proof of Theorem 4.11 in \cite{Kemperman} (see also the proof of Proposition A.6 in \cite{ChCh2014}), we have $\|\sum_{i=1}^{k_n}	\frac{\hat{Y}_{n,U_n}-X^{(i)}_{n,(n)}}{\|\hat{Y}_{n,U_n}-X^{(i)}_{n,(n)}\|} -U_n\|<1$. Thus, from (\ref{A.13}) we get
	\begin{equation}\label{A.14}
		\bigg\|\sum_{i=1}^{k_n}	\dfrac{\bar{\mathbf{r}}_{\lfloor n/k_n\rfloor,\mathbb{E}[U_n]}-X^{(i)}_{n,(n)}}{\| \bar{\mathbf{r}}_{\lfloor n/k_n\rfloor,\mathbb{E}[U_n]}-X^{(i)}_{n,(n)}\|} -k_n U_n\bigg\| \leq 3C_7k_n\delta_n,
	\end{equation}
	for all sufficiently large $n$ almost surely. Suppose that $\mathbf{z}\in G_n$ and $\|\mathbf{z}-\mathbf{r}_{\lfloor n/k_n\rfloor,\mathbb{E}[U_n]}\|>C_{10}\delta_n$ for some $C_{10}>0$. Then, by (\ref{A.10}) and the first inequality in Lemma \ref{lem-A-5} we have 
	\begin{equation*}
		\big\|\sum_{i=1}^{k_n}	\frac{\mathbf{z}-X^{(i)}_{n,(n)}}{\| \mathbf{z}-X^{(i)}_{n,(n)}\|} -k_n U_n\big\| \geq (C_{10}b'-C_4)k_n\delta_n
	\end{equation*}
	for all sufficiently large $n$ almost surely. By choosing $C_{10}$ s.t.~$C_{10}b'-C_4>4C_7$ in view of (\ref{A.14}), we have $\| \bar{\mathbf{r}}_{\lfloor n/k_n\rfloor,\mathbb{E}[U_n]}-\mathbf{r}_{\lfloor n/k_n\rfloor,\mathbb{E}[U_n]}\|\leq C_{10}\delta_n$, for all sufficiently large $n$ almost surely. This implies that $ \| \hat{Y}_{n,U_n}  -\mathbf{r}_{\lfloor n/k_n\rfloor,\mathbb{E}[U_n]} \|\leq C_{11}\delta_n$ for all sufficiently large $n$ almost surely.
\end{proof}
\begin{thm}\label{thm-v}
	Let $\mathbb{X}$ be a separable Hilbert space and let Assumptions \ref{A2} and \ref{A-B} hold. Then the following Bahadur-type asymptotic linear representation holds if for some $\rho\in(0,1/2]$, $d_n/k_n^{1-2\rho}\to c$ as $n\to\infty$ for some $c>0$:
	\begin{equation*}
		a_{\lfloor n/k_n\rfloor}(T_{\mathbf{u}_{(n)},n,k_n} -\bm\theta_{0,(n)}) -\mathbf{r}_{\lfloor n/k_n\rfloor,\mathbb{E}[U_n]}
	\end{equation*} 
	\begin{equation*}
		=-\frac{1}{k_n}\sum_{i=1}^{k_n} [\tilde{J}_{n,\mathbf{r}_{\lfloor n/k_n\rfloor,\mathbb{E}[U_n]}}]^{-1}\bigg(\frac{\mathbf{r}_{\lfloor n/k_n\rfloor,\mathbb{E}[U_n]}-X^{(i)}_{n,(n)} }{\| \mathbf{r}_{\lfloor n/k_n\rfloor,\mathbb{E}[U_n]}-X^{(i)}_{n,(n)} \|}-U_n   \bigg)+R_n
	\end{equation*}
	where $R_n=O((\log k_n)/k_n^{2\rho})$ as $n\to\infty$ almost surely.
\end{thm}
\begin{proof}
	Let $H_n=\{ \mathbf{z}\in G_n\,:\, \| \mathbf{z} -\mathbf{r}_{\lfloor n/k_n\rfloor,\mathbb{E}[U_n]} \|\leq C_{11}\delta_n\}$. Let
	\begin{equation*}
		\Gamma_n(\mathbf{z},X_n^{(i)})=\frac{\mathbf{r}_{\lfloor n/k_n\rfloor,\mathbb{E}[U_n]}-X^{(i)}_{n,(n)} }{\| \mathbf{r}_{\lfloor n/k_n\rfloor,\mathbb{E}[U_n]}-X^{(i)}_{n,(n)} \|}-\frac{\mathbf{z}-X^{(i)}_{n,(n)} }{\| \mathbf{z}-X^{(i)}_{n,(n)} \|}+\mathbb{E}\bigg[ \frac{\mathbf{z}-X^{(i)}_{n,(n)} }{\| \mathbf{z}-X^{(i)}_{n,(n)} \|}-U_n \bigg]
	\end{equation*}
	and
	\begin{equation*}
		\Delta_{n}(\mathbf{z})=\mathbb{E}\bigg[ \frac{\mathbf{z}-X_{n,(n)} }{\| \mathbf{z}-X_{n,(n)} \|}- \frac{\mathbf{r}_{\lfloor n/k_n\rfloor,\mathbb{E}[U_n]}-X_{n,(n)} }{\| \mathbf{r}_{\lfloor n/k_n\rfloor,\mathbb{E}[U_n]}-X_{n,(n)} \|}\bigg]-\tilde{J}_{n,\mathbf{r}_{\lfloor n/k_n\rfloor,\mathbb{E}[U_n]}}(\mathbf{z}-\mathbf{r}_{\lfloor n/k_n\rfloor,\mathbb{E}[U_n]})
	\end{equation*}
	where $\mathbf{z}\in\mathcal{Z}_n$. Using Assumption \ref{A-B},
	\begin{equation*}
		\mathbb{E}[\|\Gamma_n(\mathbf{z},X_n^{(i)})\|^2]\leq 2\mathbb{E}\bigg[\bigg\|  \frac{\mathbf{r}_{\lfloor n/k_n\rfloor,\mathbb{E}[U_n]}-X^{(i)}_{n,(n)} }{\| \mathbf{r}_{\lfloor n/k_n\rfloor,\mathbb{E}[U_n]}-X^{(i)}_{n,(n)} \|}-\frac{\mathbf{z}-X^{(i)}_{n,(n)} }{\| \mathbf{z}-X^{(i)}_{n,(n)} \|} \bigg\|^{2} \bigg]
	\end{equation*}
	\begin{equation*}
		+2\bigg\|  \mathbb{E}\bigg[ \frac{\mathbf{r}_{\lfloor n/k_n\rfloor,\mathbb{E}[U_n]}-X^{(i)}_{n,(n)} }{\| \mathbf{r}_{\lfloor n/k_n\rfloor,\mathbb{E}[U_n]}-X^{(i)}_{n,(n)} \|} \bigg]- \mathbb{E}\bigg[ \frac{\mathbf{z}-X^{(i)}_{n,(n)} }{\| \mathbf{z}-X^{(i)}_{n,(n)} \|} \bigg] \bigg\|^{2}\leq C_{12}\|\mathbf{z}-\mathbf{r}_{\lfloor n/k_n\rfloor,\mathbb{E}[U_n]}\|^{2},
	\end{equation*}
	for some $C_{12}>0$. Hence, by Bernstein's inequality there exists a constant $C_{13}>0$ such that
	\begin{equation}\label{A.15}
		\max_{\mathbf{z}\in H_n}\bigg\| \frac{1}{k_n}\sum_{i=1}^{k_n}\Gamma_n(\mathbf{z},X_n^{(i)}) \bigg\|  \leq C_{13}\delta_n^2
	\end{equation}
	for all $n$ sufficiently large almost surely. Using the third inequality in Lemma \ref{lem-A-5}, there exists a constant $C_{14}>0$ such that $\| \Delta_{n}(\mathbf{z}) \|\leq C_{14}\|\mathbf{z}-\mathbf{r}_{\lfloor n/k_n\rfloor,\mathbb{E}[U_n]}\|^{2}$ for all $n$ large enough. Using this result with (\ref{A.15}), the equality $\mathbb{E}[U_n]=\mathbb{E}\Big[\frac{\mathbf{r}_{\lfloor n/k_n\rfloor,\mathbb{E}[U_n]}-X^{(i)}_{n,(n)} }{\| \mathbf{r}_{\lfloor n/k_n\rfloor,\mathbb{E}[U_n]}-X^{(i)}_{n,(n)} \|}\Big]$ and the definitions of $\Gamma_n$ and $\Delta_{n}$, we have
	\begin{equation*}
		\tilde{J}_{n,\mathbf{r}_{\lfloor n/k_n\rfloor,\mathbb{E}[U_n]}}(\mathbf{z}-\mathbf{r}_{\lfloor n/k_n\rfloor,\mathbb{E}[U_n]})
	\end{equation*}
	\begin{equation*}
		= -\frac{1}{k_n}\sum_{i=1}^{k_n}\bigg(\frac{\mathbf{r}_{\lfloor n/k_n\rfloor,\mathbb{E}[U_n]}-X^{(i)}_{n,(n)} }{\| \mathbf{r}_{\lfloor n/k_n\rfloor,\mathbb{E}[U_n]}-X^{(i)}_{n,(n)} \|}-\frac{\mathbf{z}-X^{(i)}_{n,(n)} }{\| \mathbf{z}-X^{(i)}_{n,(n)} \|}  \bigg)+\tilde{R}_{n}(\mathbf{z}),
	\end{equation*}
	where $\max_{\mathbf{z}\in H_n}\| \tilde{R}_{n}(\mathbf{z})\|=O(\delta_n^2)$ as $n\to\infty$ almost surely. From Lemma \ref{lem-CCZ} it follows that the operator norm of $\tilde{J}_{n,\mathbf{r}_{\lfloor n/k_n\rfloor,\mathbb{E}[U_n]}}$ is bounded away from 0 and $[\tilde{J}_{n,\mathbf{r}_{\lfloor n/k_n\rfloor,\mathbb{E}[U_n]}}]^{-1}$ is defined on the whole of $\mathcal{Z}_n$ for $n$ large enough. Thus, there exists $C_{15}>0$ such that 
	\begin{equation*}
		\max_{\mathbf{z}\in H_n}\| [\tilde{J}_{n,\mathbf{r}_{\lfloor n/k_n\rfloor,\mathbb{E}[U_n]}}]^{-1}(\tilde{R}_{n}(\mathbf{z}))\| \leq C_{15}\delta_n^2
	\end{equation*}
	for all sufficiently large $n$ almost surely. Hence, setting $\bar{\mathbf{r}}_{\lfloor n/k_n\rfloor,\mathbb{E}[U_n]}$, defined in the proof of Proposition \ref{pro-Big-O-delta_n}, in place of $\mathbf{z}$ and using $(\ref{A.13})$ we obtain
	\begin{equation*}
		a_{\lfloor n/k_n\rfloor}(T_{\mathbf{u}_{(n)},n,k_n} -\bm\theta_{0,(n)}) -\mathbf{r}_{\lfloor n/k_n\rfloor,\mathbb{E}[U_n]}
	\end{equation*}
	\begin{equation*}
		=-\frac{1}{k_n}\sum_{i=1}^{k_n} [\tilde{J}_{n,\mathbf{r}_{\lfloor n/k_n\rfloor,\mathbb{E}[U_n]}}]^{-1}\bigg(\frac{\mathbf{r}_{\lfloor n/k_n\rfloor,\mathbb{E}[U_n]}-X^{(i)}_{n,(n)} }{\| \mathbf{r}_{\lfloor n/k_n\rfloor,\mathbb{E}[U_n]}-X^{(i)}_{n,(n)} \|}-U_n   \bigg)+R_n
	\end{equation*}
	where $\|R_n\|=O(\delta_n^2)$ as $n\to\infty$ almost surely.
\end{proof}
\begin{rem}
	In the proof of Proposition A.6 in \cite{ChCh2014} the definition of $G_n$ required also that $n^{4}\beta_j$ was in $[-C_3,C_3]$, but this is clearly a typo of the authors.
\end{rem}

We can finally prove Theorem \ref{thm-Bahadur}.
\begin{proof}[Proof of Theorem \ref{thm-Bahadur}]
	We notice that $U_n$ is a well-defined $\mathcal{Z}_n$-valued random variable and so $\mathbf{q}_{U_n}(Z^{(1)}_{\lfloor n/k_n\rfloor,(n)},...,Z^{(k_n)}_{\lfloor n/k_n\rfloor,(n)})$ is also a well-defined random variable. Further, thanks to Lemma \ref{Lemma-v}, we have that $T_{\mathbf{u}_{(n)},n,k_n}=\mathbf{q}_{U_n}(Z^{(1)}_{\lfloor n/k_n\rfloor,(n)},...,Z^{(k_n)}_{\lfloor n/k_n\rfloor,(n)})$. Thus,
	\begin{equation*}
		a_{\lfloor n/k_n\rfloor}(T_{\mathbf{u}_{(n)},n,k_n} -\bm\theta_{0,(n)}) -\mathbf{r}_{\lfloor n/k_n\rfloor,\mathbb{E}[U_n]}
	\end{equation*}	
	\begin{equation*}
		=a_{\lfloor n/k_n\rfloor}(\mathbf{q}_{U_n}(Z^{(1)}_{\lfloor n/k_n\rfloor,(n)},...,Z^{(k_n)}_{\lfloor n/k_n\rfloor,(n)}) -\bm\theta_{0,(n)}) -\mathbf{r}_{\lfloor n/k_n\rfloor,\mathbb{E}[U_n]}
	\end{equation*}
	\begin{equation*}
		=\mathbf{q}_{U_n}(a_{\lfloor n/k_n\rfloor}Z^{(1)}_{\lfloor n/k_n\rfloor,(n)}-\bm\theta_{0,(n)},...,a_{\lfloor n/k_n\rfloor}Z^{(k_n)}_{\lfloor n/k_n\rfloor,(n)}-\bm\theta_{0,(n)})-\mathbf{r}_{\lfloor n/k_n\rfloor,\mathbb{E}[U_n]}
	\end{equation*}
	and by applying Theorem \ref{thm-v} we obtain the result.
\end{proof}
\section*{The remaining proofs of Section \ref{SubSec-QoE-geo}}
\begin{proof}[Proof of Theorem \ref{thm-Bahadur-2}]
	It follows from Theorem \ref{thm-Bahadur} and the central limit theorem for triangular arrays in Hilbert spaces (see Corollary 7.8 in \cite{Gine}).
\end{proof}
\begin{proof}[Proof of Corollary \ref{co-robust-geo-infinite}]
	It follows from Theorem \ref{thm-Bahadur-2}.
\end{proof}
\begin{proof}[Proof of Proposition \ref{co-robust-geo-finite}]
	By Lemma \ref{lem-uniform-continuity},  under Assumption \ref{A2}, 
	\begin{equation*}
		\sup\limits_{\mathbf{z}\in\mathbb{R}^d:\|\mathbf{z}-\mathbf{u}\|<l_n/k_n}|\mathbf{r}_{\lfloor n/k_n\rfloor,\mathbf{z}}-\mathbf{r}_{\mathbf{u}}|\to0.
	\end{equation*}
	The result follows from Theorem \ref{thm-Bahadur} and the central limit theorem for triangular arrays.
\end{proof}

\section*{Proofs of Section \ref{SubSec-QoE-quantile-regression}}
\begin{proof}[Proof of Proposition \ref{pro-sample-quantile}]
	It follows from the arguments of the proof of Theorem  \ref{t-k-to-infinity-multivariate-robust} and the Berry-Esseen bound of order $n^{1/2}$ in \cite{Bentkus}.
\end{proof}

\section*{Proofs of Section \ref{Sec-Continuity}}
\begin{proof}[Proof of Lemma \ref{lem-uniq-defined}]
	Let $\mathbf{x}\in\mathbb{R}^k$. Take the derivative of 
	\begin{equation*}
		\sum _{i=1}^{k}|x_{i}-y|+u(x_i-y)
	\end{equation*}
	with respect to $y$, with $y\neq x_i$ for every $i=1,...,k$, and set it equal zero, that is
	\begin{equation}\label{b_y}
		(1-u)b_{y}-(1+u)(k-b_{y})=0,
	\end{equation}
	where $b_{y}=|\{i:x_i<y\}|$. The only solution of (\ref{b_y}) is $b_y=k\alpha$, and this can only happen if $\alpha\in\{\frac{1}{k},\frac{2}{k},...,\frac{k-1}{k}\}$. In this case the set of geometric quantiles is $[\tilde{x}_{k\alpha},\tilde{x}_{k\alpha+1}]$.
	
	When $\alpha\notin\{\frac{1}{k},\frac{2}{k},...,\frac{k-1}{k}\}$ then (\ref{b_y}) is never satisfied and thus no geometric quantile lies in $\mathbb{R}\setminus\{x_1,...,x_k\}$. In particular, $(1-u)b_{y}-(1+u)(k-b_{y})$ is an increasing function of $b_y$ and so there exists an $\tilde{x}_i$ for some $i=1,...,k$ such that $(1-u)i-(1+u)(k-i)<0$ and $(1-u)(i+1)-(1+u)(k-i-1)>0$. Such $\tilde{x}_i$ is the unique geometric quantile.
	
	Finally to show equality with $q_\alpha(\mathbf{x})$, it is sufficient to see that $i$ is the integer in the interval $(k\alpha,k\alpha+1)$, which can be written as $\lceil k\alpha\rceil$ or equivalently as $\lfloor k\alpha+1 \rfloor$.
\end{proof}

\begin{proof}[Proof of Lemma \ref{lem-con-quant}]
	We first focus on $q_{\alpha}$. Let $\alpha\in(0,1)\setminus\{\frac{1}{k},\frac{2}{k},...,\frac{k-1}{k}\}$. Consider any $\mathbf{x},\mathbf{y}\in\mathbb{R}^k$. By definition we have that $|q_\alpha(\mathbf{x})-q_\alpha(\mathbf{y})|=|x_i-y_j|$ for some $i,j\in\{1,...,k\}$. Without loss of generality we consider that $|x_i-y_j|=x_i-y_j$. We know that there are at least $k-\lfloor k\alpha\rfloor$ components of $\mathbf{x}$ that are greater than or equal to $x_i$. Let $G_\mathbf{x}\subset\{1,...,k\}$ be the set of the indices of such components, formally $G_\mathbf{x}:=\{l\in\{1,...,k\}:x_l\geq q_\alpha(\mathbf{x})\}$. Similarly we know that there are at least $\lfloor k\alpha\rfloor+1$ components of $\mathbf{y}$ that are less than or equal to $y_j$. Let $L_\mathbf{y}:=\{l\in\{1,...,k\}:y_l\leq q_\alpha(\mathbf{y})\}$. 
	
	Observe that $G_\mathbf{x}\cap L_\mathbf{y}\neq\emptyset$ because $G_\mathbf{x}$ and $L_\mathbf{x}$ contain $k-\lfloor k\alpha\rfloor$ and $\lfloor k\alpha\rfloor+1$ distinct elements of $\{1,...,k\}$, respectively. Then, we have
	\begin{equation*}
		x_i-y_j\leq x_l-y_l,\quad\forall l\in G_\mathbf{x}\cap L_\mathbf{y},
	\end{equation*}
	and the result follows.
	
	Similar arguments apply for the case of $\alpha\in\{\frac{1}{k},\frac{2}{k},...,\frac{k-1}{k}\}$. In particular, we have that $|q_\alpha(\mathbf{x})-q_\alpha(\mathbf{y})|=\frac{1}{2}|x_i+x_l-y_j-y_h|$ for some $i,l,j,h\in\{1,...,k\}$ with $i\neq l$ and $j\neq h$. Without loss of generality we consider that $|x_i+x_l-y_j-y_h|=x_i+x_l-y_j-y_h$ and that $x_i\geq x_l$ and $y_j\geq y_h$. Consider the sets $G_\mathbf{x}\cup\{l\}$ and $L_\mathbf{y}\cup\{j\}$; they contain at least $k-\lfloor k\alpha\rfloor+1$ and $\lfloor k\alpha\rfloor+1$ distinct elements of $\{1,...,k\}$, respectively. Thus, we have that $|(G_\mathbf{x}\cup\{l\})\cap( L_\mathbf{y}\cup\{j\})|\geq 2$. Since $x_i+x_l\leq x_p+x_q$ for any $p,q\in G_\mathbf{x}\cup\{l\}$ with $p\neq q$ and $y_j+y_h\geq y_r+y_s$ for any $r,s\in L_\mathbf{y}\cup\{j\}$ with $r\neq s$, we have that 
	\begin{equation*}
		\frac{1}{2}(x_i+x_l-y_j-y_h)\leq \frac{1}{2}(x_p+x_q-y_p-y_q)\leq \max(x_p-y_p,x_q-y_q),
	\end{equation*}
	for every $p,q\in (G_\mathbf{x}\cup\{l\})\cap(L_\mathbf{y}\cup\{j\})$ with $p\neq q$,	from which we obtain the result.

	The arguments for $q^{\circ}_{\alpha}$ are similar. By definition we have that $|q^\circ_\alpha(\mathbf{x})-q^\circ_\alpha(\mathbf{y})|=|x_i-y_j|$ for some $i,j\in\{1,...,k\}$. Without loss of generality we consider that $|x_i-y_j|=x_i-y_j$. By definition we know that there are at least $\lceil k(1-\alpha)\rceil$ components of $\mathbf{x}$ that are greater than or equal to $x_i$. Let $\tilde{G}_\mathbf{x}:=\{l\in\{1,...,k\}:x_l\geq q^\circ_\alpha(\mathbf{x})\}$. Similarly we know that there are at least $\lceil k\alpha\rceil$ components of $\mathbf{y}$ that are less than or equal to $y_j$. Let $\tilde{L}_\mathbf{y}:=\{l\in\{1,...,k\}:y_l\leq q^\circ_\alpha(\mathbf{y})\}$. 
	
	Observe that $\tilde{G}_\mathbf{x}\cap \tilde{L}_\mathbf{y}\neq\emptyset$ by the following argument. If $\alpha\neq a/k$ where $a=1,...,k$ then $|\tilde{G}_\mathbf{x}|+|\tilde{L}_\mathbf{y}|\geq\lceil k(1-\alpha)\rceil+\lceil k\alpha\rceil\geq k+1$. If $\alpha= a/k$ for some $a=1,...,k-1$ then there are at least two $x_i$'s that satisfy the condition
	\begin{equation*}
		|\{j\in\{1,...,k\}:x_j\leq x_i\}|\geq k\alpha\quad\textnormal{and}\quad |\{j\in\{1,...,k\}:x_j\geq x_i\}|\geq k(1-\alpha)
	\end{equation*}
	and since in the definition of $q$ we consider the smallest one, we have that $|\tilde{G}_\mathbf{x}|\geq k(1-\alpha)+1$. Since $|\tilde{L}_\mathbf{y}|\geq k\alpha$, we have that $|\tilde{G}_\mathbf{x}|+|\tilde{L}_\mathbf{y}|\geq k+1$. Therefore, we have
	\begin{equation*}
		x_i-y_j\leq x_l-y_l,\quad\forall l\in \tilde{G}_\mathbf{x}\cap \tilde{L}_\mathbf{y},
	\end{equation*}
	and the result follows.
\end{proof}
\section*{Proofs of Section \ref{SubSec-Continuity-component}}
\begin{proof}[Proof of Theorem \ref{thm-continuity-component-Hamel}]
	By the axiom of choice every vector space has an Hamel basis. Moreover, we have that $\mathbf{q}_{Hamel,\bm{\alpha}}(\mathbf{x}_{1},...,\mathbf{x}_{k})\in\mathbb{X}$ because it is a finite linear combination of elements of $\mathbb{X}$. Thus, we have existence. Uniqueness follows from the definition of $q_\alpha$. Let $I_{\mathbf{x}_1,...,\mathbf{x}_k}:=\{l\in I:x_{i}^{(l)}\neq0\textnormal{ for some $i=1,...,k$}\}$. Consider any $\mathbf{x}_1,...,\mathbf{x}_k,\mathbf{z}_1,...,\mathbf{z}_k$ in $\mathbb{X}$ and let $I'=I_{\mathbf{x}_1,...,\mathbf{x}_k}\cup I_{\mathbf{z}_1,...,\mathbf{z}_k}$. Then, by Lemma \ref{lem-con-quant} we have
	\begin{equation*}
		\bigg\|\sum_{l\in I}q_{\alpha_l}(x^{(l)}_1,...,x^{(l)}_k)\mathbf{b}_{l}-\sum_{l\in I}q_{\alpha_l}(z^{(l)}_1,...,z^{(l)}_k)\mathbf{b}_{l}\bigg\|_{\mathbb{X},Hamel}
	\end{equation*}
	\begin{equation*}
		\leq\sum_{l\in I'}\Big|q_{\alpha_l}(x^{(l)}_1,...,x^{(l)}_k)-q_{\alpha_l}(z^{(l)}_1,...,z^{(l)}_k)\Big|
	\end{equation*}
	\begin{equation*}
		\leq\sum_{l\in I'}\sum_{i=1}^{k}|x^{(l)}_i-z^{(l)}_i|=\sum_{i=1}^{k}\|\mathbf{x}_i-\mathbf{z}_i\|_{\mathbb{X},Hamel}.
	\end{equation*}
\end{proof}

\begin{proof}[Proof of Theorem \ref{thm-continuity-component-Scahuder}]
	To prove existence we need to show that 
	\begin{equation*}
		\mathbf{q}_{S,\bm{\alpha}}(\mathbf{x}_{1},...,\mathbf{x}_{k})\in\mathbb{X}
	\end{equation*}
	By completeness of $\mathbb{X}$, it is enough to show that the sequence $(\sum_{l=1}^{n}\check{x}_l\mathbf{d}_{l})_{n\in\mathbb{N}}$ is Cauchy. By Lemma \ref{lem-con-quant} we have that $|\check{x}_{l}|\leq\sum_{i=1}^{k}|x^{(l)}_{i}|$. So, by Theorem 6.7 (c) in \cite{Heil},
	\begin{equation*}
		\bigg\|\sum_{l=1}^{n}\check{x}_l\mathbf{d}_{l}\bigg\|_{\mathbb{X}}\leq \mathcal{K}\bigg\|\sum_{l=1}^{n}\sum_{i=1}^{k}|x^{(l)}_{i}|\mathbf{d}_{l}\bigg\|_{\mathbb{X}}\leq \mathcal{K}\sum_{i=1}^{k}\bigg\|\sum_{l=1}^{n}|x^{(l)}_{i}|\mathbf{d}_{l}\bigg\|_{\mathbb{X}}\leq \mathcal{K}^2\sum_{i=1}^{k}\bigg\|\sum_{l=1}^{n}x^{(l)}_{i}\mathbf{d}_{l}\bigg\|_{\mathbb{X}},
	\end{equation*}
	for every $n\in\mathbb{N}$, where the last inequality follows from Theorem 6.7 (d) in \cite{Heil}. We remark that $\mathcal{K}$ is a constant independent of $n$ and of $\check{x}_l,x^{(l)}_{1},....,x^{(l)}_{k}$ for every $l\in\mathbb{N}$. By the same arguments and by taking $\check{x}_{1},...,\check{x}_m=0$ for $m<n$, we have
	\begin{equation*}
		\bigg\|\sum_{l=1}^{n}\check{x}_l\mathbf{d}_{l}-\sum_{l=1}^{m}\check{x}_l\mathbf{d}_{l}\bigg\|_{\mathbb{X}}=\bigg\|\sum_{l=m+1}^{n}\check{x}_l\mathbf{d}_{l}\bigg\|_{\mathbb{X}}\leq \mathcal{K}^{2}\sum_{i=1}^{k}\bigg\|\sum_{l=m+1}^{n}x^{(l)}_{i}\mathbf{d}_{l}\bigg\|_{\mathbb{X}}.
	\end{equation*}
	Since by the definition of Schauder basis we have that $(\sum_{l=1}^{n}x^{(l)}_i\mathbf{d}_{l})_{n\in\mathbb{N}}$ is a Cauchy sequence, for every $i=1,...,k$, we conclude that $(\sum_{l=1}^{n}\check{x}_l\mathbf{d}_{l})_{n\in\mathbb{N}}$ is Cauchy. Thus, we have existence. Uniqueness follows from the definition of $q_\alpha$.
	
	To prove continuity we consider any $\mathbf{x}_1,...,\mathbf{x}_k,\mathbf{z}_1,...,\mathbf{z}_k$ in $\mathbb{X}$. By Theorem 6.4 (f) in \cite{Heil}
	\begin{equation*}
		\bigg\|\sum_{l=1}^{\infty}q_{\alpha_l}(x^{(l)}_1,...,x^{(l)}_k)\mathbf{d}_{l}-\sum_{l=1}^{\infty}q_{\alpha_l}(z^{(l)}_1,...,z^{(l)}_k)\mathbf{d}_{l}\bigg\|_{\mathbb{X}}
	\end{equation*}
	\begin{equation}\label{Heil-1}
		\leq\sup\limits_{n\in\mathbb{N}}	\bigg\|\sum_{l=1}^{n}(q_{\alpha_l}(x^{(l)}_1,...,x^{(l)}_k)-q_{\alpha_l}(z^{(l)}_1,...,z^{(l)}_k))\mathbf{d}_{l}\bigg\|_{\mathbb{X}}.
	\end{equation}
	Further, by Lemma \ref{lem-con-quant} and by Theorem 6.7 (c) in \cite{Heil} (\ref{Heil-1}) is bounded by
	\begin{equation*}
		\mathcal{K}\sup\limits_{n\in\mathbb{N}}	\bigg\|\sum_{l=1}^{n}\sum_{i=1}^{k}|x^{(l)}_i-z^{(l)}_i|\mathbf{d}_{l}\bigg\|_{\mathbb{X}}\leq \mathcal{K}\sum_{i=1}^{k}\sup\limits_{n\in\mathbb{N}}	\bigg\|\sum_{l=1}^{n}|x^{(l)}_i-z^{(l)}_i|\mathbf{d}_{l}\bigg\|_{\mathbb{X}}
	\end{equation*}
	\begin{equation*}
		\leq \mathcal{K}\sum_{i=1}^{k}\sup\limits_{n\in\mathbb{N},|\lambda_{l}|\leq 1}	\bigg\|\sum_{l=1}^{n}\lambda_{l}(x^{(l)}_i-z^{(l)}_i)\mathbf{d}_{l}\bigg\|_{\mathbb{X}}\leq \mathcal{K}^2\sum_{i=1}^{k}	\bigg\|\mathbf{x}_i-\mathbf{z}_i\bigg\|_{\mathbb{X}},
	\end{equation*}
	where the last equality follows from Theorem 6.4 (f) in \cite{Heil}.
	
	Concerning the Hilbert space case, by Parseval's identity and by Lemma \ref{lem-con-quant}
	\begin{equation*}
		\bigg\|\sum_{l\in I}q_{\alpha_l}(x^{(l)}_1,...,x^{(l)}_k)\mathbf{d}_{l}-\sum_{l\in I}q_{\alpha_l}(z^{(l)}_1,...,z^{(l)}_k)\mathbf{d}_{l}\bigg\|_{\mathbb{X}}
	\end{equation*}
	\begin{equation*}
		=\sum_{l\in I}|\check{x}_{l}-\check{z}_{l}|^{2}\leq \sum_{l\in I}\sum_{i=1,...,k}|x_{i}^{(l)}-z_{i}^{(l)}|^{2}
	\end{equation*}
	\begin{equation*}
		=\sum_{i=1,...,k}\sum_{l\in I}|x_{i}^{(l)}-z_{i}^{(l)}|^{2}=\sum_{i=1,...,k}\|\mathbf{x}_i-\mathbf{z}_i\|_{\mathbb{X}}^{2},
	\end{equation*}
	where $I$ is the (possibly uncountable) index set of the orthonormal basis $(\mathbf{d}_l)_{l\in I}$, although only countably many elements in $\sum_{l\in I}x_{l}\mathbf{d}_l$ are different from zero.
\end{proof}

\begin{proof}[Proof of Theorem \ref{thm-continuity-component-point}]
	In any of the different cases ($\mathbb{X}=L_p[0,1]$, with $1 \leq p \leq \infty$, $\mathbb{X}=C[0,1]$, and $\mathbb{X}=C[0,1]$) uniqueness is ensured by the definition of $q_\alpha$. For $\mathbb{X}=L_p[0,1]$, $p\in[1,\infty)$, by Lemma \ref{lem-con-quant}
	\begin{equation*}
		\bigg(\int_{0}^{1}|\mathbf{q}_{P,\bm{\alpha}}(\mathbf{f}_{1},...,\mathbf{f}_{k})(x)|^{p}dx\bigg)^{1/p}\leq \bigg(\sum_{i=1}^{k}\int_{0}^{1}|\mathbf{f}_{i}(x)|^{p}dx\bigg)^{1/p}\leq \sum_{i=1}^{k}\|\mathbf{f}_i\|_{\mathbb{X}},
	\end{equation*}
	which proves existence. For $L_\infty[0,1]$, by Lemma \ref{lem-con-quant}
	\begin{equation*}
		\|\mathbf{q}_{P,\bm{\alpha}}(\mathbf{f}_{1},...,\mathbf{f}_{k})\|_{\infty}=\inf\{C\geq0:|\mathbf{q}_{P,\bm{\alpha}}(\mathbf{f}_{1},...,\mathbf{f}_{k})(x)|\leq C,\textnormal{ for a.e.~$x\in[0,1]$}\}
	\end{equation*}
	\begin{equation*}
		\leq \inf\{C\geq0:\max_{i=1,...,k}|\mathbf{f}_i(x)|\leq C,\textnormal{ for a.e.~$x\in[0,1]$}\}=\max_{i=1,...,k}\|\mathbf{f}_i\|_{\mathbb{X}}.
	\end{equation*}
	
	Using similar arguments we obtain continuity in $L_p[0,1]$ with $1 \leq p \leq \infty$. For $\mathbb{X}=C[0,1]$, by continuity of $q_{\alpha}$ and of $\mathbf{f}_{1},...,\mathbf{f}_{k}$ we obtain that $\mathbf{q}_{P,\bm{\alpha}}(\mathbf{f}_{1},...,\mathbf{f}_{k})$ is a continuous function. Moreover, by Lemma \ref{lem-con-quant},
	\begin{equation*}
		\sup_{x\in[0,1]}|\mathbf{q}_{P,\bm{\alpha}}(\mathbf{f}_{1},...,\mathbf{f}_{k})(x)-\mathbf{q}_{P,\bm{\alpha}}(\mathbf{g}_{1},...,\mathbf{g}_{k})(x)|\leq \max_{i=1,...,k}\sup_{x\in[0,1]}|\mathbf{f}_i(x)-\mathbf{g}_i(x)|.
	\end{equation*}
	Using similar arguments we obtain continuity in $D[0,1]$.
\end{proof}

\section*{Proofs of Section \ref{SubSec-Continuity-geo}}
\begin{proof}[Proof of Lemma \ref{lem-ineq}]
	Take $\mathbf{w}_{0}\in W$ s.t.~$\|\mathbf{w}_{0}\|_{\mathbb{X}}=1$. Define the linear functional $F$ on $W$ by $F(\lambda \mathbf{w}_0)=\lambda \|\mathbf{w}_{0}\|_{\mathbb{X}}$, where $\lambda\in\mathbb{R}$. This functional is linear, has norm 1 (precisely $\sup\limits_{\mathbf{z}\in W}\frac{|F(\mathbf{z})|}{\|\mathbf{z}\|_{\mathbb{X}}}=1$) and $F(\mathbf{w}_0)=\|\mathbf{w}_{0}\|_{\mathbb{X}}=1$. Thus, by the Hahn-Banach theorem $F$ extends to $\mathbb{X}$, and its extension, which we also denote by $F$, has norm 1. Now, let $P:\mathbb{X}\to\mathbb{X}$ be defined by $P(\mathbf{z})=F(\mathbf{z})\mathbf{w}_0$, where $\mathbf{z}\in\mathbb{X}$. As $P^2=P$, $P$ is a projection operator with range $W$. Boundedness of $F$ implies continuity of $P$ and so $\textnormal{Ker}(P)=\textnormal{Range}(I-P)$, where $I:\mathbb{X}\to\mathbb{X}$ is the identity operator. By letting $Y=\textnormal{Ker}(P)$, we obtain the first statement. 
	
	Since 
	\begin{equation*}
		\|P\|=\sup\limits_{\mathbf{z}\in \mathbb{X}}\frac{\|P(\mathbf{z})\|_{\mathbb{X}}}{\|\mathbf{z}\|_{\mathbb{X}}}=\sup\limits_{\mathbf{z}\in \mathbb{X}}\frac{|F(\mathbf{z})|}{\|\mathbf{z}\|_{\mathbb{X}}}=1,
	\end{equation*}
	we have $\|\mathbf{z}\|_{\mathbb{X}}\geq \|P(\mathbf{z})\|_{\mathbb{X}}$ for every $\mathbf{z}\in\mathbb{X}$. Further, since for every $\mathbf{z}\in\mathbb{X}$ we have $\mathbf{z}=\lambda \mathbf{w}_0+\mathbf{z}_{Y}$, for some $\lambda\in\mathbb{R}$ and $\mathbf{z}_{Y}\in Y$, and $P(\lambda \mathbf{w}_0+\mathbf{z}_{Y})=\lambda \mathbf{w}_0$, we obtain (\ref{ineq}).
	
	If $\mathbb{X}$ is strictly convex then every functional attains its supremum only at most one point $\mathbf{z}$ s.t.~$\|\mathbf{z}\|=1$. Indeed, if there are two points $\mathbf{s}$ and $\mathbf{r}$ in the unit circle for which a functional $g$ attains its supremum then $g(\mathbf{s})=g(\mathbf{r})=\alpha g(\mathbf{s})+(1-\alpha)g(\mathbf{r})$ where $0< \alpha< 1$, but since $\mathbb{X}$ is convex we have $\|\alpha \mathbf{s}+(1-\alpha)\mathbf{r}\|<1$. This leads to the existence of a point $\mathbf{b}=c(\alpha \mathbf{s}+(1-\alpha)\mathbf{r})$ where $c>1$, s.t.~$\|b\|_{\mathbb{X}}=1$ and $g(\mathbf{b})>g(\mathbf{s})=g(\mathbf{r})$, which is a contradiction. For $F$ such a point is $\mathbf{w}_0$. Thus, we obtain that $\|P(\mathbf{z})\|_{\mathbb{X}}=|F(\mathbf{z})|<\|\mathbf{z}\|_{\mathbb{X}}$ for every $\mathbf{z}\notin W$, hence (\ref{ineq-strict}).
	
	To obtain (\ref{ineq-y}) and (\ref{ineq-strict-y}) is sufficient to repeat the same arguments for $\mathbf{y}$ instead of $\mathbf{x}$ noticing that $W$ is a subset of the complementary closed subspace of $\{c\mathbf{y}:c\in\mathbb{R}\}$.
\end{proof}

\begin{proof}[Proof of Theorem \ref{thm-extension-of-Kemperman-}]
	The uniqueness of the geometric quantile under (i) follows directly from Theorem 2.17 in \cite{Kemperman}.
	
	Now, let $\mathbf{x}_1,...,\mathbf{x}_k$ lie on a straight line $W$. If $\mathbf{x}_{1}=...=\mathbf{x}_{k}$ then it is easy to see that the quantile is unique. Otherwise, consider the function
	\begin{equation*}
		f(\mathbf{y})=\sum_{i=1}^{k}\left\|\mathbf{x}_{i}-\mathbf{y}\right\|_{\mathbb{X}}+\langle\mathbf{u} ,\mathbf{x}_{i}-\mathbf{y}\rangle,
	\end{equation*}
	and for every $\mathbf{h}\in\mathbb{X}$ with $\mathbf{h}\neq\mathbf{0}$ let $\Delta_{\mathbf{h}}f(\mathbf{y}):=f(\mathbf{y}+\mathbf{h})-f(\mathbf{y})$. Following the proof of Theorem 2.17 in \cite{Kemperman},
	\begin{equation*}
		\Delta^{2}_{\mathbf{h}}f(\mathbf{y})=\sum_{i=1}^{k}\left\|\mathbf{x}_{i}-\mathbf{y}+2\mathbf{h}\right\|_{\mathbb{X}}-2\left\|\mathbf{x}_{i}-\mathbf{y}+\mathbf{h}\right\|_{\mathbb{X}}+\left\|\mathbf{x}_{i}-\mathbf{y}\right\|_{\mathbb{X}}\geq0.
	\end{equation*}
	So, $f$ is convex. In particular, $\Delta^{2}_{\mathbf{h}}f(\mathbf{y})=0$ when $\mathbf{h}$ is linearly dependent on $\mathbf{x}_{i}-\mathbf{y}$ for every $i=1,...,k$, which can only happen if $\mathbf{h}$ and $\mathbf{y}$ lie on the line passing through $\mathbf{x}_{1},...,\mathbf{x}_{k}$, which we call $W$. So, $f$ is strictly convex for every $\mathbf{y}\notin W$. Assume that the set of geometric quantiles is non-empty. If one of the geometric quantiles, which we denote by $\mathbf{y}^*$, does not lie in $W$ then it is unique because $f$ is strictly increasing in any direction $\mathbf{h}\in\mathbb{X}$. This can also be seen by observing that $g_i(t)=\left\|\mathbf{x}_{i}-\mathbf{y}^*+t\mathbf{h}\right\|_{\mathbb{X}}+\langle\mathbf{u} ,\mathbf{x}_{i}-\mathbf{y}^*+t\mathbf{h}\rangle$ is a strictly convex function on $\mathbb{R}$ if $\mathbf{x}_{i}-\mathbf{y}^*$ and $\mathbf{h}$ are linearly independent, while it is just convex otherwise (see Lemma 2.14 in \cite{Kemperman}), where $i=1,...k$ and $\mathbf{h}\in\mathbb{X}$. Since for every $\mathbf{h}\in\mathbb{X}$ there is always at least one $i\in\{1,...,k\}$ for which $\mathbf{x}_{i}-\mathbf{y}^*$ and $\mathbf{h}$ are linearly independent, we have that $g(t)=\sum_{i=1}^{k}g_i(t)$ is a strictly convex function on $\mathbb{R}$ for every $\mathbf{h}\in\mathbb{X}$. Therefore, if there exists an element of the set of geometric quantiles in $\mathbb{X}\setminus W$, then it is unique.
	
	Consider the case that the set of geometric quantiles is a subset of $W$. By Lemma \ref{lem-ineq} we have
	\begin{equation*}
		1=\|\mathbf{e}\|_{\mathbb{X}^{*}}=\sup\limits_{\mathbf{z}\in\mathbb{X}}\frac{|\langle\mathbf{e},\mathbf{z}\rangle|}{\|\mathbf{z}\|_{\mathbb{X}}}=\sup\limits_{\mathbf{z}\in W}\frac{|\langle\mathbf{e},\mathbf{z}\rangle|}{\|\mathbf{z}\|_{\mathbb{X}}}=\sup\limits_{t\in\mathbb{R}}\frac{|t||\langle\mathbf{e},\mathbf{r}\rangle|}{|t|\|\mathbf{r}\|_{\mathbb{X}}}=\frac{|\langle\mathbf{e},\mathbf{r}\rangle|}{\|\mathbf{r}\|_{\mathbb{X}}}
	\end{equation*}
	for any $\mathbf{r}\in W$, which implies that $|\langle\mathbf{e},\mathbf{r}\rangle|=\|\mathbf{r}\|_{\mathbb{X}}$ for any $\mathbf{r}\in W$. Since $\mathbf{x}_{i}-\mathbf{y}^*\in W$ for any geometric quantile $\mathbf{y}^*$, we have $\langle\mathbf{v} ,\mathbf{x}_{i}-\mathbf{y}^*\rangle=0$. Further, since $\mathbf{x}_{i}-\mathbf{y}^*=\bm{\phi}(x_i-y^*)$ for some $\bm{\phi}\in W$ such that $\langle\mathbf{e} ,\bm{\phi}\rangle=1$ (\textit{i.e.}~$\|\bm{\phi}\|_{\mathbb{X}}=1$) and some $x_1,...,x_k,y^*\in\mathbb{R}$ and for every $i=1,...,k$, we have
	\begin{equation*}
		\sum _{i=1}^{k}\left\|\mathbf{x}_{i}-\mathbf{y}^*\right\|_{\mathbb{X}}+\langle\mathbf{u} ,\mathbf{x}_{i}-\mathbf{y}^*\rangle=\sum _{i=1}^{k}|x_{i}-y^*|+u(x_{i}-y^*).
	\end{equation*}
	By Lemma \ref{lem-uniq-defined}, this implies that in order to have uniqueness it is sufficient that $\alpha\in(0,1)\setminus\{\frac{1}{k},\frac{2}{k},...,\frac{k-1}{k}\}$.
\end{proof}

\begin{proof}[Proof of Theorem \ref{thm-extension-of-Kemperman-median}]
	Assume condition (i). Then, the uniqueness of the geometric median follows directly from Theorem \ref{thm-extension-of-Kemperman-} (see also Theorem 2.17 in \cite{Kemperman}). Now, let $\mathbf{x}_1,...,\mathbf{x}_k$ lie on a straight line $W$ and assume that the geometric median $\mathbf{y}\notin W$. From Lemma \ref{lem-ineq} we have that $\mathbf{y}=\mathbf{y}_{W}+\mathbf{y}_{Y}$, for some $\mathbf{y}_{W}\in W$ and $\mathbf{y}_{Y}\in Y$, and that 
	\begin{equation*}
		\sum _{i=1}^{k}\left\|\mathbf{x}_{i}-\mathbf{y}\right\|_{\mathbb{X}}>\sum _{i=1}^{k}\left\|\mathbf{x}_{i}-\mathbf{y}_{W}\right\|_{\mathbb{X}},
	\end{equation*}
	which implies that $\mathbf{y}$ cannot be the geometric median. Thus, the geometric median must lie in $W$ and in that case it is unique if and only if $k$ is odd.
\end{proof}

\begin{proof}[Proof of Theorem \ref{thm-extension-of-Kemperman}]
	The sufficiency of condition (i) is contained in Theorem \ref{thm-extension-of-Kemperman-}.
	
	Now, let $\mathbf{x}_1,...,\mathbf{x}_k$ lie on a straight line and let $\mathbf{v}\neq\mathbf{0}$. We show uniqueness by showing that the geometric quantile lies outside $W$ (see the proof of Theorem \ref{thm-extension-of-Kemperman-}). Assume that we have more than one geometric quantile lying in $W$ and let $\mathbf{y}^*$ be one of them. Then, we have
	\begin{equation*}
		\sum _{i=1}^{k}\left\|\mathbf{x}_{i}-\mathbf{y}^*\right\|_{\mathbb{X}}+\langle\mathbf{u} ,\mathbf{x}_{i}-\mathbf{y}^*\rangle=\sum _{i=1}^{k}|x_{i}-y^*|+u(x_{i}-y^*)
	\end{equation*}
	for some $x_i,...,x_k,y^*\in\mathbb{R}$, which implies that if there is more than one geometric quantile in $W$, then there is a continuum of them. In particular, the set of geometric quantiles lying in $W$ is a closed connected interval whose extremes are $\mathbf{x}_{i},\mathbf{x}_{j}$ for some $i,j\in\{1,...,k\}$ and which does not contain any $\mathbf{x}_{l}$ for $l=1,...,k$ with $l\neq i,j$. 
	
	Consider $\mathbf{y}^*\neq\mathbf{x}_{i},\mathbf{x}_{j}$. By Assumption \ref{A-1} the norm is G\^{a}teaux differentiable in direction $\mathbf{z}$ at $\mathbf{x}_i\neq\mathbf{0}$ for some $i=1,...,k$, which implies that it is G\^{a}teaux differentiable in direction $\mathbf{z}$ at any point in $W$. Since $\mathbf{y}^*\neq\mathbf{x}_{i},\mathbf{x}_{j}$, we have that $\mathbf{x}_i-\mathbf{y}^*\neq\mathbf{0}$ for every $i=1,...,k$, and so the G\^{a}teaux derivative of the function $f$ defined in the proof of Theorem \ref{thm-extension-of-Kemperman-} evaluated at $\mathbf{y}^*$ in direction $\mathbf{z}$ is well defined:
	\begin{equation*}
		\lim\limits_{t\to 0}\frac{f(\mathbf{y}^*+t\mathbf{z})-f(\mathbf{y}^*)}{t}=\sum_{i=1}^{k}\lim\limits_{t\to 0}\frac{\|\mathbf{x}_i-\mathbf{y}^*-t\mathbf{z}\|_{\mathbb{X}}-\|\mathbf{x}_i-\mathbf{y}^*\|_{\mathbb{X}}}{t}-k\langle\mathbf{u} ,\mathbf{z}\rangle.
	\end{equation*}
	By Lemma \ref{lem-ineq} we have that for every $\mathbf{x}\in W$, $\mathbf{x}\perp\mathbf{z}$, which by Theorem 3.1 and Lemma 3.1 (see also eq.~(1) in \cite{James}) implies that
	\begin{equation*}
		\lim\limits_{t\to 0^{-}}\frac{\|\mathbf{x}+t\mathbf{z}\|_{\mathbb{X}}-\|\mathbf{x}\|_{\mathbb{X}}}{t}\leq0\leq\lim\limits_{t\to 0^{+}}\frac{\|\mathbf{x}+t\mathbf{z}\|_{\mathbb{X}}-\|\mathbf{x}\|_{\mathbb{X}}}{t}.
	\end{equation*}
	However, since by Assumption \ref{A-1} we have
	\begin{equation*}
		\lim\limits_{t\to 0^{-}}\frac{\|\mathbf{x}+t\mathbf{z}\|_{\mathbb{X}}-\|\mathbf{x}\|_{\mathbb{X}}}{t}=\lim\limits_{t\to 0^{+}}\frac{\|\mathbf{x}+t\mathbf{z}\|_{\mathbb{X}}-\|\mathbf{x}\|_{\mathbb{X}}}{t}.
	\end{equation*}
	We conclude that
	\begin{equation*}
		\lim\limits_{t\to 0}\frac{f(\mathbf{y}^*+t\mathbf{z})-f(\mathbf{y}^*)}{t}=-k\langle\mathbf{u} ,\mathbf{z}\rangle=-k\langle\mathbf{v} ,\mathbf{z}\rangle,
	\end{equation*}
	where the last equality follows from $\mathbf{z}\in Y$. The same argument holds for $-\mathbf{z}$ instead of $\mathbf{z}$ because $\|\mathbf{x}+t(-\mathbf{z})\|_{\mathbb{X}}=\|-\mathbf{x}+t\mathbf{z}\|_{\mathbb{X}}$ and $-\mathbf{x}$ is an element of $W$. Thus, we have
	\begin{equation*}
		\lim\limits_{t\to 0}\frac{f(\mathbf{y}^*-t\mathbf{z})-f(\mathbf{y}^*)}{t}=k\langle\mathbf{v} ,\mathbf{z}\rangle.
	\end{equation*}
	Since by assumption $\langle\mathbf{u} ,\mathbf{z}\rangle\neq0$ and so $\langle\mathbf{v} ,\mathbf{z}\rangle\neq0$, we conclude that the G\^{a}teaux derivative of $f$ evaluated at $\mathbf{y}^*$ in direction $\mathbf{z}$ (or $-\mathbf{z}$) is strictly negative. This implies that $\mathbf{y}^*$ cannot be a geometric quantile and since
	\begin{equation*}
		\sum _{i=1}^{k}\left\|\mathbf{x}_{i}-\mathbf{y}^*\right\|_{\mathbb{X}}+\langle\mathbf{u} ,\mathbf{x}_{i}-\mathbf{y}^*\rangle=\sum _{i=1}^{k}\left\|\mathbf{x}_{i}-\mathbf{y}'\right\|_{\mathbb{X}}+\langle\mathbf{u} ,\mathbf{x}_{i}-\mathbf{y}'\rangle
	\end{equation*} 
	for any geometric quantile $\mathbf{y}'$ lying in $W$, this implies that none of the geometric quantiles lying in $W$ can actually be a geometric quantile. Thus, if $\mathbf{v}\neq\mathbf{0}$ geometric quantiles can only lie outside $W$ and, as we showed, there can be at most one geometric quantile outside $W$. Thus, we have uniqueness.
	
	Now, let $\mathbf{x}_1,...,\mathbf{x}_k$ lie on a straight line, $\mathbf{v}=\mathbf{0}$, and $\alpha\in(0,1)\setminus\{\frac{1}{k},\frac{2}{k},...,\frac{k-1}{k}\}$. Then, by Lemma \ref{lem-uniq-defined} we have that there is a unique geometric quantile on $W$.
	
	Let us now investigate the other direction. Let $\mathbf{x}_1,...,\mathbf{x}_k$ lie on a straight line $W$ and $\mathbf{v}=\mathbf{0}$. Consider any $\mathbf{y}\in\mathbb{X}$. Then, by Lemma \ref{lem-ineq} there exist $\mathbf{y}_{Y}\in Y$ and $\mathbf{y}_{W}\in W$, where $\mathbb{X}=Y+W$, such that $\mathbf{y}=\mathbf{y}_{Y}+\mathbf{y}_{W}$. Consider the case of $\mathbf{y}_{Y}\neq\mathbf{0}$. By Lemma \ref{lem-ineq} we have that $\left\|\mathbf{x}_{i}-\mathbf{y}\right\|_{\mathbb{X}}>\left\|\mathbf{x}_{i}-\mathbf{y}_{W}\right\|_{\mathbb{X}}$. Thus, we obtain
	\begin{equation*}
		\sum _{i=1}^{k}\left\|\mathbf{x}_{i}-\mathbf{y}\right\|_{\mathbb{X}}+\langle\mathbf{u} ,\mathbf{x}_{i}-\mathbf{y}\rangle=\sum _{i=1}^{k}\left\|\mathbf{x}_{i}-\mathbf{y}\right\|_{\mathbb{X}}+u\langle\mathbf{e} ,\mathbf{x}_{i}-\mathbf{y}_{W}\rangle
	\end{equation*}
	\begin{equation*}
		>\sum _{i=1}^{k}\left\|\mathbf{x}_{i}-\mathbf{y}_{W}\right\|_{\mathbb{X}}+u\langle\mathbf{e} ,\mathbf{x}_{i}-\mathbf{y}_{W}\rangle,
	\end{equation*}
	which implies that any geometric quantile lies in $W$. Then, by Lemma \ref{lem-uniq-defined} we have that when $\mathbf{x}_1,...,\mathbf{x}_k$ lie on a straight line $W$ and $\mathbf{v}=0$, the geometric quantile is unique if and only if $\alpha\in(0,1)\setminus\{\frac{1}{k},\frac{2}{k},...,\frac{k-1}{k}\}$.
	
	Thus, we have that if the geometric quantile is unique then $\mathbf{x}_1,...,\mathbf{x}_k$ do not lie on a straight line, or if they lie on a straight line $W$ then if $\mathbf{y}^*\notin W$ it means that $\mathbf{v}\neq\mathbf{0}$ while if $\mathbf{y}^*\in W$ then $\mathbf{v}=\mathbf{0}$ and $\alpha\in(0,1)\setminus\{\frac{1}{k},\frac{2}{k},...,\frac{k-1}{k}\}$.
\end{proof}

\begin{proof}[Proof of Corollary \ref{co-unique}]
	Since $\mathbb{X}$ is smooth, any point in $\mathbb{X}$ is G\^{a}teaux differentiable along any direction. Further, if the points $\mathbf{x}_1,...,\mathbf{x}_k$ do not lie on a straight line or if the points lie on a straight line and $\mathbf{v}\neq\mathbf{0}$, then we obtain the result by Theorem \ref{thm-extension-of-Kemperman}. If the points lie on a straight line and $\mathbf{v}=\mathbf{0}$, then $\mathbf{u}=u\mathbf{e}$ and so $\|\mathbf{u}\|_{\mathbb{X}^*}=|u|$. In order to have uniqueness $\alpha\notin \{\frac{1}{k},\frac{2}{k},...,\frac{k-1}{k}\}$, \textit{i.e.}~$|u|\notin\{1-\frac{2j}{k},j=1,...,\lfloor\frac{k}{2}\rfloor\}$ since $u=2\alpha-1$.
\end{proof}

\begin{proof}[Proof of Proposition \ref{pro-extra}]
	We only prove only (iii), since for (i) and (ii) the arguments are similar. For one direction, let $\mathbf{x}_1,...,\mathbf{x}_k$ lie on a straight line $W$ and $\mathbf{v}=0$. Consider any $\mathbf{y}\in\mathbb{X}$. Then, by Lemma \ref{lem-ineq} there are some $\mathbf{y}_{Y}\in Y$ and $\mathbf{y}_{W}\in W$, where $\mathbb{X}=Y+W$, such that $\mathbf{y}=\mathbf{y}_{Y}+\mathbf{y}_{W}$. Consider the case of $\mathbf{y}_{Y}\neq\mathbf{0}$. By Lemma \ref{lem-ineq} we have that $\left\|\mathbf{x}_{i}-\mathbf{y}\right\|_{\mathbb{X}}>\left\|\mathbf{x}_{i}-\mathbf{y}_{W}\right\|_{\mathbb{X}}$. Thus, we obtain
	\begin{equation*}
		\sum _{i=1}^{k}\left\|\mathbf{x}_{i}-\mathbf{y}\right\|_{\mathbb{X}}+\langle\mathbf{u} ,\mathbf{x}_{i}-\mathbf{y}\rangle=\sum _{i=1}^{k}\left\|\mathbf{x}_{i}-\mathbf{y}\right\|_{\mathbb{X}}+u\langle\mathbf{e} ,\mathbf{x}_{i}-\mathbf{y}_{W}\rangle
	\end{equation*}
	\begin{equation*}
		>\sum _{i=1}^{k}\left\|\mathbf{x}_{i}-\mathbf{y}_{W}\right\|_{\mathbb{X}}+u\langle\mathbf{e} ,\mathbf{x}_{i}-\mathbf{y}_{W}\rangle,
	\end{equation*}
	which implies that any geometric quantile lies in $W$.
	
	For the other direction, assume that the set of geometric quantiles is a subset of $W$. This implies that there is no geometric quantile outside $W$, so neither condition (i) nor (ii) can hold (see the proof of Theorem \ref{thm-extension-of-Kemperman}).
\end{proof}

\begin{proof}[Proof of Lemma \ref{lem-continuity-quantiles}]
	The weak* convergence follows by adapting the arguments of Theorem 1 in \cite{Cadre2001} to quantiles. The convergence when $\mathbb{X}$ is finite-dimensional follows from adapting the arguments of Corollary 2.26 in \cite{Kemperman} to quantiles.
\end{proof}

\begin{proof}[Proof of Lemma \ref{lem-uniform-continuity}]
	To simplify the notation we let $\|\cdot\|$ indicate both $\|\cdot\|_{\mathbb{X}}$ and $\|\cdot\|_{\mathbb{X}^*}$. For every $n\in\mathbb{N}$ and $\mathbf{y},\mathbf{u}\in\mathbb{R}^d$ with $\|\mathbf{u}\|<1$, we define
	\begin{equation*}
		f_n(\mathbf{y},\mathbf{u})=\int \| \mathbf{y}-\mathbf{x}\|-\|\mathbf{x}\|\mu_n(d\mathbf{x})-\langle\mathbf{u},\mathbf{y}\rangle
	\end{equation*}
	and
	\begin{equation*}
		f(\mathbf{y},\mathbf{u})=\int \| \mathbf{y}-\mathbf{x}\|-\|\mathbf{x}\|\mu(d\mathbf{x})-\langle\mathbf{u},\mathbf{y}\rangle.
	\end{equation*}
	For every $n\in\mathbb{N}$, $f_n$ and $f$ are continuous functions and, for every $\mathbf{y},\mathbf{u}\in\mathbb{R}^d$, 
	\begin{equation}\label{unif-cont-eq-1}
		\lim\limits_{n\to\infty}	f_n(\mathbf{y},\mathbf{u})=	f(\mathbf{y},\mathbf{u}).
	\end{equation}
	Further, for every $n\in\mathbb{N}$ and $\mathbf{y},\mathbf{z},\mathbf{u},\mathbf{v}\in\mathbb{R}^d$ with $\|\mathbf{u}\|<1$ and $\|\mathbf{v}\|<1$, we have
	\begin{equation}\label{unif-cont-eq-2}
		|f_n(\mathbf{y},\mathbf{u})-f_n(\mathbf{z},\mathbf{v})|\leq 2\|\mathbf{y}-\mathbf{z}\|+\min(\|\mathbf{y}\|,\|\mathbf{z}\|)\|\mathbf{u}-\mathbf{v}\|.
	\end{equation}
	Fix $\varepsilon>0$. For $t,s>0$ define $B_t:=\{\mathbf{y}\in\mathbb{R}^d:\|\mathbf{y}\|_{\mathbb{X}}\leq t\}$ and $B'_s:=\{\mathbf{u}\in\mathbb{R}^d:\|\mathbf{u}\|_{\mathbb{X}^*}\leq s\}$. It is possible to see that the $\mathbf{u}$-geometric quantiles of $\mu_n$ and of $\mu$ for every $n\in\mathbb{N}$ and $\mathbf{u}\in\mathbb{R}^d$ with $\|\mathbf{u}\|<1-\varepsilon$, lie in $B_\mathbf{r}$ for some $r>0$.
	
	Using continuity of $f_n$, pointwise convergence (\ref{unif-cont-eq-1}), and equicontinuity (\ref{unif-cont-eq-2}) on $B_r\times B'_{1-\varepsilon}$,  by the Arzel\'{a}-Ascoli theorem we obtain the uniform convergence of $f_n$ to $f$ on $B_r\times B'_{1-\varepsilon}$.
	
	For every $\mathbf{u},\mathbf{v}\in\mathbb{R}^d$ with $\|\mathbf{u}\|<1$ and $\|\mathbf{v}\|<1$, consider $|\inf\limits_{\mathbf{y}\in B_r}f_n(\mathbf{y},\mathbf{u})-\inf\limits_{\mathbf{z}\in B_r}f_n(\mathbf{z},\mathbf{v})|$. Wlog we assume that $\inf\limits_{\mathbf{y}\in B_r}f_n(\mathbf{y},\mathbf{u})\geq\inf\limits_{\mathbf{z}\in B_r}f_n(\mathbf{z},\mathbf{v})$ and denote by $\mathbf{y}^*_{n,\mathbf{v}}$ the $\mathbf{v}$-geometric quantile of $\mu_n$, \textit{i.e.}~the value such that $f_n(\mathbf{y}^*_{n,\mathbf{v}},\mathbf{v})=\inf\limits_{\mathbf{z}\in B_r}f_n(\mathbf{z},\mathbf{v})$. Then, we have
	\begin{equation}\label{unif-cont-eq-3}
		|\inf\limits_{\mathbf{y}\in B_r}f_n(\mathbf{y},\mathbf{u})-\inf\limits_{\mathbf{z}\in B_r}f_n(\mathbf{z},\mathbf{v})|\leq f_n(\mathbf{y}^*_{n,\mathbf{v}},\mathbf{u})-f_n(\mathbf{y}^*_{n,\mathbf{v}},\mathbf{v})=\langle \mathbf{v}-\mathbf{u},\mathbf{y}^*_{n,\mathbf{v}}\rangle\leq r\|\mathbf{v}-\mathbf{u}\|.
	\end{equation}
	We denote by $\mathbf{y}^*_\mathbf{u}$ the $\mathbf{u}$-geometric quantile of $\mu$. By continuity of the geometric quantiles (see Lemma \ref{lem-continuity-quantiles}), we have 
	\begin{equation}\label{unif-cont-eq-7}
		\lim\limits_{n\to\infty}\mathbf{y}^*_{n,\mathbf{u}}=\mathbf{y}^*_{\mathbf{u}}.
	\end{equation}
	By (\ref{unif-cont-eq-7}) and uniform convergence of $f_n$ to $f$ we obtain
	\begin{equation*}
		\lim\limits_{n\to\infty}\inf\limits_{\mathbf{y}\in B_r}f_n(\mathbf{y},\mathbf{u})=\lim\limits_{n\to\infty}f_n(\mathbf{y}^*_{n,\mathbf{u}},\mathbf{u})=f(\mathbf{y}^*_\mathbf{u},\mathbf{u})=\inf\limits_{\mathbf{y}\in B_r}f(\mathbf{y},\mathbf{u}).
	\end{equation*}
	This shows pointwise convergence of the sequence of functions $\mathbf{u}\mapsto\inf\limits_{\mathbf{y}\in B_r}f_n(\mathbf{y},\mathbf{u})$, $n\in\mathbb{N}$. These functions are continuous by (\ref{unif-cont-eq-3}). Thus, by the Arzel\'{a}-Ascoli theorem we obtain uniform convergence of $\inf\limits_{\mathbf{y}\in B_r}f_n(\mathbf{y},\cdot)$ to $\inf\limits_{\mathbf{y}\in B_r}f(\mathbf{y},\cdot)$ on $B'_{1-\varepsilon}$.
	
	Assume now that \begin{equation*}
		\limsup\limits_{n\to\infty}\sup\limits_{\|\mathbf{u}\|<1-\varepsilon}\| \mathbf{y}^*_{n,\mathbf{u}}-\mathbf{y}^*_{\mathbf{u}}\|\geq2\delta
	\end{equation*}
	for some $\delta>0$. There exists a subsequence $n_k\to\infty$ and points $\mathbf{u}_{n_{k}}$, $\mathbf{y}^*_{n_{k},\mathbf{u}_{n_{k}}}$, and $\mathbf{y}^*_{\mathbf{u}_{n_{k}}}$ such that 
	\begin{equation}\label{unif-cont-eq-6}
		\| \mathbf{y}^*_{n_{k},\mathbf{u}_{n_{k}}}-\mathbf{y}^*_{\mathbf{u}_{n_{k}}}\|\geq\delta,
	\end{equation}
	for all $n_k$ large enough. Since $\mathbf{u}_{n_{k}}$, $\mathbf{y}^*_{n_{k},\mathbf{u}_{n_{k}}}$, and $\mathbf{y}^*_{\mathbf{u}_{n_{k}}}$ are sequences of points in compact sets they have converging subsequences, which wlog we denote again by  $\mathbf{u}_{n_{k}}$, $\mathbf{y}^*_{n_{k},\mathbf{u}_{n_{k}}}$, and $\mathbf{y}^*_{\mathbf{u}_{n_{k}}}$, with limits
	$\tilde{\mathbf{u}}$, $\tilde{\mathbf{y}}^*$, and $\hat{\mathbf{y}}^*$, respectively.
	
	We have
	\begin{equation*}
		|f_{n_k}(\mathbf{y}^*_{n_{k},\mathbf{u}_{n_{k}}},\mathbf{u}_{n_{k}})-f(\tilde{\mathbf{y}}^*,\tilde{\mathbf{u}})|
	\end{equation*}
	\begin{equation}\label{unif-cont-eq-4}
		\leq	|f_{n_k}(\mathbf{y}^*_{n_{k},\mathbf{u}_{n_{k}}},\mathbf{u}_{n_{k}})-f(\mathbf{y}^*_{n_{k},\mathbf{u}_{n_{k}}},\mathbf{u}_{n_{k}})|+	|f(\mathbf{y}^*_{n_{k},\mathbf{u}_{n_{k}}},\mathbf{u}_{n_{k}})-f(\tilde{\mathbf{y}}^*,\tilde{\mathbf{u}})|\to 0,\quad\textnormal{as $k\to\infty$,}
	\end{equation}
	where for the first addendum we used uniform convergence of $f_n$ to $f$ and for the second one continuity of $f$. Moreover, by uniform convergence of $\inf\limits_{\mathbf{y}\in B_r}f_n(\mathbf{y},\cdot)$ to $\inf\limits_{\mathbf{y}\in B_r}f(\mathbf{y},\cdot)$ we have
	\begin{equation*}
		|\inf\limits_{\mathbf{y}\in B_r}f_{n_k}(\mathbf{y},\mathbf{u}_{n_{k}})-\inf\limits_{\mathbf{z}\in B_r}f(\mathbf{z},\mathbf{u}_{n_{k}})|\to0,\quad\textnormal{as $k\to\infty$,}
	\end{equation*}
	and using (\ref{unif-cont-eq-4}) we obtain
	\begin{equation*}
		|f(\tilde{\mathbf{y}}^*,\tilde{\mathbf{u}})-\inf\limits_{\mathbf{z}\in B_r}f(\mathbf{z},\mathbf{u}_{n_{k}})|\to0,\quad\textnormal{as $k\to\infty$,}
	\end{equation*}
	that is
	\begin{equation}\label{unif-cont-eq-5}
		\lim\limits_{k\to\infty}f(\mathbf{y}^*_{\mathbf{u}_{n_{k}}},\mathbf{u}_{n_{k}})=f(\tilde{\mathbf{y}}^*,\tilde{\mathbf{u}}).
	\end{equation}
	Since the limit of $\mathbf{y}^*_{\mathbf{u}_{n_{k}}}$ is $\hat{\mathbf{y}}^*$, by continuity of $f$ and (\ref{unif-cont-eq-5}) we have $f(\hat{\mathbf{y}}^*,\tilde{\mathbf{u}})=f(\tilde{\mathbf{y}}^*,\tilde{\mathbf{u}})$. Moreover, by (\ref{unif-cont-eq-3}) applied to $f$ we have 
	\begin{equation*}
		\lim\limits_{k\to\infty}f(\mathbf{y}^*_{\mathbf{u}_{n_{k}}},\mathbf{u}_{n_{k}})=\inf\limits_{\mathbf{y}\in B_r}f(\mathbf{y},\tilde{\mathbf{u}})=f(\mathbf{y}^*_{\tilde{\mathbf{u}}},\tilde{\mathbf{u}}).
	\end{equation*}
	So $f(\hat{\mathbf{y}}^*,\tilde{\mathbf{u}})=f(\tilde{\mathbf{y}}^*,\tilde{\mathbf{u}})=f(\mathbf{y}^*_{\tilde{\mathbf{u}}},\tilde{\mathbf{u}})$, and since $\mathbf{y}^*_{\tilde{\mathbf{u}}}$ is unique we conclude that $\hat{\mathbf{y}}^*=\tilde{\mathbf{y}}^*=\mathbf{y}^*_{\tilde{\mathbf{u}}}$. This is a contradiction because by (\ref{unif-cont-eq-6}) $\|\hat{\mathbf{y}}^*-\tilde{\mathbf{y}}^*\|\geq\delta$. Therefore, for any $\varepsilon>0$, 
	\begin{equation*}
		\sup_{\|\mathbf{u}\|_{\mathbb{X}^*}<1-\varepsilon}\|\mathbf{y}^*_{n,\mathbf{u}}-\mathbf{y}^*_\mathbf{u}\|\to0\quad\textnormal{as $n\to\infty$}. 
	\end{equation*}
	
	Now, assume that $\mu$ is atomless. Consider any sequence of constants $c_n$ such that $c_n\to0$ as $n\to\infty$, and any $\mathbf{v}\in\mathbb{R}^d$ with $\| \mathbf{v}\|_{\mathbb{X}^*}<1$. For any $n$ large enough $1>c_n\geq 0$ and we have
	\begin{equation*}
		\sup_{\|\mathbf{u}-\mathbf{v}\|_{\mathbb{X}^*}\leq c_n}\|\mathbf{y}^*_{n,\mathbf{u}}-\mathbf{y}^*_\mathbf{v}\|
	\end{equation*}
	\begin{equation*}
		\leq 	\sup_{\|\mathbf{u}-\mathbf{v}\|_{\mathbb{X}^*}\leq c_n}\|\mathbf{y}^*_{n,\mathbf{u}}-\mathbf{y}^*_\mathbf{u}\|+\sup_{\|\mathbf{u}-\mathbf{v}\|_{\mathbb{X}^*}\leq c_n}\|\mathbf{y}^*_{\mathbf{u}}-\mathbf{y}^*_\mathbf{v}\|\to0,\quad\textnormal{as $n\to\infty$,}
	\end{equation*}
	where for the first addendum we used the result just shown and for the second one continuity of the $\mathbf{v}$-geometric quantile with respect to the quantile parameter $\mathbf{v}$ (see Theorem 3.1 in \cite{ChCh2014}).
\end{proof}

\begin{proof}[Proof of Lemma \ref{lem-open-set}]
	It is equivalent to show that the set 
	\begin{equation*}
		\{(\mathbf{x}_1,...,\mathbf{x}_k)\in\mathbb{X}^{k}:\textnormal{$\mathbf{x}_1,...,\mathbf{x}_k$ lie on a straight line}\}
	\end{equation*}
	is closed. Consider first the case of $k=3$ and consider any $\mathbf{x}_1,\mathbf{x}_2,\mathbf{x}_3\in\mathbb{X}$ lying on a straight line. Then, the area of the triangle with vertices $\mathbf{x}_1,$ $\mathbf{x}_2,$ and $\mathbf{x}_3$ is zero, and so by Heron's formula
	\begin{equation}\label{A}
		A((\textbf{x}_1,\mathbf{x}_2,\mathbf{x}_3)):=s(s-a)(s-b)(s-c)=0,
	\end{equation}
	where $a=\|\mathbf{x}_1-\mathbf{x}_2\|_{\mathbb{X}}$, $b=\|\mathbf{x}_2-\mathbf{x}_3\|_{\mathbb{X}}$, $c=\|\mathbf{x}_1-\mathbf{x}_3\|_{\mathbb{X}}$, and $s=(a+b+c)/2$. Thus, 
	\begin{equation*}
		\{(\textbf{x}_1,\mathbf{x}_2,\mathbf{x}_3)\in\mathbb{X}^{3}:\textnormal{$\mathbf{x}_1,...,\mathbf{x}_3$ lie on a straight line}\}
	\end{equation*}
	\begin{equation*}
		=\{(\textbf{x}_1,\mathbf{x}_2,\mathbf{x}_3)\in\mathbb{X}^{3}:A((\textbf{x}_1,\mathbf{x}_2,\mathbf{x}_3))=0\}.
	\end{equation*}
	Since $A$ is continuous, any limit point $(\textbf{z}_1,\mathbf{z}_2,\mathbf{z}_3)$ of $\{(\textbf{x}_1,\mathbf{x}_2,\mathbf{x}_3)\in\mathbb{X}^{3}:A((\textbf{x}_1,\mathbf{x}_2,\mathbf{x}_3))=0\}$ is such that $A((\textbf{z}_1,\mathbf{z}_2,\mathbf{z}_3))=0$ and so $(\textbf{z}_1,\mathbf{z}_2,\mathbf{z}_3)\in\{(\textbf{x}_1,\mathbf{x}_2,\mathbf{x}_k)\in\mathbb{X}^{3}:A((\textbf{x}_1,\mathbf{x}_2,\mathbf{x}_k))=0\}$. Therefore, $\{(\textbf{x}_1,\mathbf{x}_2,\mathbf{x}_k)\in\mathbb{X}^{3}:A((\textbf{x}_1,\mathbf{x}_2,\mathbf{x}_k))=0\}$ is a closed set and so $\mathcal{X}_{3}$ is open. To show that $\mathcal{X}_{k}$ is open, define $B((\textbf{x}_1,...,\mathbf{x}_k))$ as the sum of the area of all triangles whose vertices belong to $\{\mathbf{x}_1,...,\mathbf{x}_k\}$. Then 
	\begin{equation*}
		\{(\mathbf{x}_1,...,\mathbf{x}_k)\in\mathbb{X}^{k}:\textnormal{$\mathbf{x}_1,...,\mathbf{x}_k$ lie on a straight line}\}
	\end{equation*}
	\begin{equation*}
		=\{(\textbf{x}_1,...,\mathbf{x}_k)\in\mathbb{X}^{k}:B((\textbf{x}_1,...,\mathbf{x}_k))=0\},
	\end{equation*}
	and since $B$ is continuous, we obtain that $\mathcal{X}_{k}$ is open.
\end{proof}

\begin{proof}[Proof of Theorem \ref{thm-continuity}]
	We prove the first statement. Consider any sequence 
	\begin{equation*}
		((\mathbf{x}_{1}^{(m)},...,\mathbf{x}_{k}^{(m)}))_{m\in\mathbb{N}}
	\end{equation*}
	of elements in $\mathcal{X}_k$ and any $(\mathbf{x}_{1},...,\mathbf{x}_{k})\in\mathcal{X}_k$ such that $(\mathbf{x}_{1}^{(m)},...,\mathbf{x}_{k}^{(m)})\to(\mathbf{x}_{1},...,\mathbf{x}_{k})$ as $m\to\infty$. Without loss of generality we choose $(\mathbf{x}_{1}^{(m)},...,\mathbf{x}_{k}^{(m)})$ such that 
	\begin{equation*}
		\sup_{m\in\mathbb{N}}\|(\mathbf{x}_{1}^{(m)},...,\mathbf{x}_{k}^{(m)})-(\mathbf{x}_{1},...,\mathbf{x}_{k})\|_{\mathbb{X}^k}<\varepsilon,
	\end{equation*}
	for some $\varepsilon>0$.
	
	Let $\bar{c}:=\max_{i,j=1,...,k}\|\mathbf{x}_i-\mathbf{x}_j\|_{\mathbb{X}}$. Using the fact that $|\langle\mathbf{u},\mathbf{x}-\mathbf{y}\rangle|<\|\mathbf{x}-\mathbf{y}\|_{\mathbb{X}}$ for every $\mathbf{x},\mathbf{y}\in\mathbb{X}$, it is possible to see that the geometric quantile, which coincides with $\mathbf{q}_{\mathbf{u}}(\mathbf{x}_{1},...,\mathbf{x}_{k})$ and which we denote by $\mathbf{y}_*$, lies in the ball of radius $2\bar{c}$ centered at one of the points, say $\mathbf{x}_1$. Hence, we have that the (unique) geometric quantile of $\mathbf{x}_{1}^{(m)},...,\mathbf{x}_{k}^{(m)}$, which coincides with $\mathbf{q}_{\mathbf{u}}(\mathbf{x}^{(m)}_{1},...,\mathbf{x}^{(m)}_{k})$ and which we denote by $\mathbf{y}_*^{(m)}$, lies in the ball of radius $2\bar{c}+\varepsilon$ centered at $\mathbf{x}_1$, for every $m\in\mathbb{N}$. Since $\mathbb{X}$ is reflexive and strictly convex, the geometric quantiles $\mathbf{y}_*$ and $\mathbf{y}_*^{(m)}$ exist and are unique. 
	
	Then, similarly as in the proof of Lemma 2 (i) in \cite{Cadre2001}, we have
	\begin{equation*}
		\bigg|\sum _{i=1}^{k}\left\|\mathbf{x}_{i}^{(m)}-\mathbf{y}_{*}^{(m)}\right\|_{\mathbb{X}}-\langle\mathbf{u} ,\mathbf{y}_{*}^{(m)}\rangle-\sum _{i=1}^{k}\left\|\mathbf{x}_{i}-\mathbf{y}_{*}^{(m)}\right\|_{\mathbb{X}}+\langle\mathbf{u} ,\mathbf{y}_{*}^{(m)}\rangle\bigg|
	\end{equation*}
	\begin{equation}\label{convergence-m}
		\leq 2 \sum _{i=1}^{k}\left\|\mathbf{x}_{i}^{(m)}-\mathbf{x}_{i}\right\|_{\mathbb{X}}\to0,\quad\textnormal{as $m\to\infty$}.
	\end{equation}
	Since for all $m$ we have that
	\begin{equation*}
		\sum _{i=1}^{k}\left\|\mathbf{x}_{i}^{(m)}-\mathbf{y}_{*}^{(m)}\right\|_{\mathbb{X}}-\langle\mathbf{u} ,\mathbf{y}_{*}^{(m)}\rangle\leq \sum _{i=1}^{k}\left\|\mathbf{x}_{i}^{(m)}-\mathbf{y}_{*}\right\|_{\mathbb{X}}-\langle\mathbf{u} ,\mathbf{y}_{*}\rangle,
	\end{equation*}
	then
	\begin{equation*}
		\limsup\limits_{m\to\infty}\sum _{i=1}^{k}\left\|\mathbf{x}_{i}^{(m)}-\mathbf{y}_{*}^{(m)}\right\|_{\mathbb{X}}-\langle\mathbf{u} ,\mathbf{y}_{*}^{(m)}\rangle\leq \sum _{i=1}^{k}\left\|\mathbf{x}_{i}-\mathbf{y}_{*}\right\|_{\mathbb{X}}-\langle\mathbf{u} ,\mathbf{y}_{*}\rangle.
	\end{equation*}
	Hence, by (\ref{convergence-m}) we obtain that
	\begin{equation}\label{convergence-m-2}
		\limsup\limits_{m\to\infty}\sum _{i=1}^{k}\left\|\mathbf{x}_{i}-\mathbf{y}_{*}^{(m)}\right\|_{\mathbb{X}}-\langle\mathbf{u} ,\mathbf{y}_{*}^{(m)}\rangle\leq \sum _{i=1}^{k}\left\|\mathbf{x}_{i}-\mathbf{y}_{*}\right\|_{\mathbb{X}}-\langle\mathbf{u} ,\mathbf{y}_{*}\rangle,
	\end{equation}
	but since $\mathbf{y}_{*}$ is the minimizer of the right hand side of (\ref{convergence-m-2}) (see (\ref{geometric-quantile-Banach-without-x})) we conclude that
	\begin{equation*}
		\lim\limits_{m\to\infty}\sum _{i=1}^{k}\left\|\mathbf{x}_{i}-\mathbf{y}_{*}^{(m)}\right\|_{\mathbb{X}}-\langle\mathbf{u} ,\mathbf{y}_{*}^{(m)}\rangle= \sum _{i=1}^{k}\left\|\mathbf{x}_{i}-\mathbf{y}_{*}\right\|_{\mathbb{X}}-\langle\mathbf{u} ,\mathbf{y}_{*}\rangle.
	\end{equation*}
	Since $\mathbb{X}$ is reflexive, $\mathbb{X}^*$ is also reflexive, which in turn implies that $\mathbb{X}^*$ is an Asplund space (see \cite{Fry2002} and Theorem 8.26 in \cite{Fabian} for equivalent definitions of the Asplund space). Then, we can apply Theorem 3 in \cite{Asplund1968} and obtain that $\mathbf{y}_{*}^{(m)}\to\mathbf{y}^*$ in norm, hence the result.
	
	The same arguments apply to the second statement. The additional assumptions ensure the uniqueness of the quantile.
\end{proof}

\begin{proof}[Proof of Theorem \ref{thm-extension-Gervini}]
	Consider any $\mathbf{x}_{1},...,\mathbf{x}_{k}$, $k\in\mathbb{N}$, in $\mathbb{X}$ and any $\mathbf{u}\in \mathbb{X}$ with $\|\mathbf{u}\|_{\mathbb{X}}<1$. Observe that a geometric quantile $\mathbf{y}^*$ satisfies
	\begin{equation*}
		\sum_{i=1}^{k}\frac{\mathbf{x}_i-\mathbf{y}^*}{\|\mathbf{x}_i-\mathbf{y}^*\|_{\mathbb{X}}}+k\mathbf{u}=0,
	\end{equation*}
	provided that $\mathbf{y}^*\neq \mathbf{x}_i$, $i=1,...,k$. Now, let $\tilde{\mathbf{x}}:=\mathbf{u}+\mathbf{y}^*$, then $\mathbf{u}=\frac{\tilde{\mathbf{x}}-\mathbf{y}^*}{\|\mathbf{u}\|_{\mathbb{X}}^{-1}\|\tilde{\mathbf{x}}-\mathbf{y}^*\|_{\mathbb{X}}}$ and so
	\begin{equation*}
		\sum_{i=1}^{k}\frac{\mathbf{x}_i-\mathbf{y}^*}{\|\mathbf{x}_i-\mathbf{y}^*\|_{\mathbb{X}}}+k\frac{\tilde{\mathbf{x}}-\mathbf{y}^*}{\|\mathbf{u}\|_{\mathbb{X}}^{-1}\|\tilde{\mathbf{x}}-\mathbf{y}^*\|_{\mathbb{X}}}=0
	\end{equation*}
	\begin{equation*}
		\Leftrightarrow \mathbf{y}^*=\sum_{i=1}^{k}w_i\mathbf{x}_i+\tilde{\mathbf{x}}w_{k+1}
	\end{equation*}
	\begin{equation*}
		\Leftrightarrow \mathbf{y}^*=\frac{1}{1-w_{k+1}}\Big(\sum_{i=1}^{k}w_i\mathbf{x}_i+\mathbf{u}w_{k+1}\Big),
	\end{equation*}
	where
	\begin{equation*}
		w_i=\frac{\|\mathbf{x}_i-\mathbf{y}^*\|_{\mathbb{X}}^{-1}}{\sum_{i=1}^{k}\|\mathbf{x}_i-\mathbf{y}^*\|_{\mathbb{X}}^{-1}+k\|\mathbf{u}\|_{\mathbb{X}}\|\tilde{\mathbf{x}}-\mathbf{y}^*\|_{\mathbb{X}}^{-1}}=\frac{\|\mathbf{x}_i-\mathbf{y}^*\|_{\mathbb{X}}^{-1}}{\sum_{i=1}^{k}\|\mathbf{x}_i-\mathbf{y}^*\|_{\mathbb{X}}^{-1}+k}
	\end{equation*}
	for $i=1,...,k$, and
	\begin{equation*}
		w_{k+1}=\frac{k\|\mathbf{u}\|_{\mathbb{X}}\|\tilde{\mathbf{x}}-\mathbf{y}^*\|_{\mathbb{X}}^{-1}}{\sum_{i=1}^{k}\|\mathbf{x}_i-\mathbf{y}^*\|_{\mathbb{X}}^{-1}+k\|\mathbf{u}\|_{\mathbb{X}}\|\tilde{\mathbf{x}}-\mathbf{y}^*\|_{\mathbb{X}}^{-1}}=\frac{k}{\sum_{i=1}^{k}\|\mathbf{x}_i-\mathbf{y}^*\|_{\mathbb{X}}^{-1}+k}
	\end{equation*}
	and so $\sum_{i=1}^{k+1}w_i=1$ and $w_1,...,w_{k+1}\geq 0$. Observe that since $\mathbf{y}^*\neq \mathbf{x}_i$, $i=1,...,k$ and $\|\mathbf{y}^*\|_{\mathbb{X}}<C$, for some $C\geq0$ depending on $\mathbf{x}_1,...,\mathbf{x}_k$, we have that $w_1,...,w_k$ are strictly positive and so $w_{k+1}<1$. Now if $\mathbf{y}^*= \mathbf{x}_i$, for say $i=1,...,j$, then the quantile is still of the form 
	\begin{equation*}
		\frac{1}{1-w_{k+1}}\Big(\sum_{i=1}^{k}w_i\mathbf{x}_i+\mathbf{u}w_{k+1}\Big)
	\end{equation*}
	with $\sum_{i=1}^{k+1}w_i=1$ and $w_1,...,w_{k+1}\geq 0$, in particular $\mathbf{y}^*=\sum_{i=1}^{j}w_i\mathbf{x}_i$. 
	
	Therefore, we have that the quantile is given by 
	\begin{equation*}
		\frac{1}{1-w_{k+1}^*}\Big(\sum_{i=1}^{k}w^*_i\mathbf{x}_i+\mathbf{u}w_{k+1}^*\Big),
	\end{equation*}
	where $(w^*_i,...,w_{k+1}^*)$ is obtained by
	\begin{equation}\label{minimization-on-T_k+1}
		{\displaystyle {\underset {\mathbf{w}\in T_{k+1}}{\operatorname {arg\,min} }}\sum _{i=1}^{k}\left\|\mathbf{x}_{i}-\frac{1}{1-w_{k+1}}\Big(\sum_{l=1}^{k}w_{l}\mathbf{x}_{l}+w_{k+1}\mathbf{u}\Big)\right\|_{\mathbb{X}}}+\langle u,\mathbf{x}_i-\frac{1}{1-w_{k+1}}\Big(\sum_{l=1}^{k}w_{l}\mathbf{x}_{l}+w_{k+1}\mathbf{u}\Big)\rangle.
	\end{equation}
	
	To see that this minimization is well-defined observe that there is an $\varepsilon>0$ depending on $\mathbf{x}_{1},...,\mathbf{x}_{k}$ such that (\ref{minimization-on-T_k+1}) is equal to
	\begin{equation*}
		{\displaystyle {\underset {\mathbf{w}\in T^{(\varepsilon)}_{k+1}}{\operatorname {arg\,min} }}\sum _{i=1}^{k}\left\|\mathbf{x}_{i}-\frac{1}{1-w_{k+1}}\Big(\sum_{l=1}^{k}w_{l}\mathbf{x}_{l}+w_{k+1}\mathbf{u}\Big)\right\|_{\mathbb{X}}}+\langle \mathbf{u},\mathbf{x}_i-\frac{1}{1-w_{k+1}}\Big(\sum_{l=1}^{k}w_{l}\mathbf{x}_{l}+w_{k+1}\mathbf{u}\Big)\rangle
	\end{equation*}
	where $T^{(\varepsilon)}_{k+1}:=\{\mathbf{v}\in[0,1]^{k}\times[0,1-\varepsilon]:v_1+\cdots+v_{k+1}=1\}$. Indeed, consider any sequence $((\mathbf{x}^{(m)}_1,...,\mathbf{x}^{(m)}_k))_{m\in\mathbb{N}}$ of elements in $\mathbb{X}^{k}$ and any $(\mathbf{x}_1,...,\mathbf{x}_k)\in \mathbb{X}^{k}$ such that $(\mathbf{x}^{(m)}_1,...,\mathbf{x}^{(m)}_k)\to(\mathbf{x}_1,...,\mathbf{x}_k)$ as $m\to\infty$. Without loss of generality suppose 
	\begin{equation*}
		\max_{i=1,...,k}\|\mathbf{x}^{(m)}_{i}-\mathbf{x}_i\|_{\mathbb{X}}<\delta
	\end{equation*}
	for some $\delta>0$. Let $\bar{c}:=\max_{i,j=1,...,k}\|\mathbf{x}_i-\mathbf{x}_j\|_{\mathbb{X}}$. Then, any quantile lies in the ball of radius $2\bar{c}$ centred at one of the points, say $\mathbf{x}_1$. Hence, we have that any quantile of $(\mathbf{x}^{(m)}_1,...,\mathbf{x}^{(m)}_k)$, for every $m\in\mathbb{N}$, lies in the ball of radius $2\bar{c}+\delta$ centred at $\mathbf{x}_1$. Thus, $w_{k+1}$ and $w_{k+1}^{(m)}$, $m\in\mathbb{N}$, are bounded by $\frac{k}{k(2\bar{c}+\delta)+k}$ and so 
	\begin{equation*}
		\varepsilon=1-\frac{1}{2\bar{c}+\delta+1}>0.
	\end{equation*}
\end{proof}

\begin{proof}[Proof of Proposition \ref{pro-l1}]
	Since the geometric quantile is given by
	\begin{equation*}
		{\displaystyle {\underset {\mathbf{y}\in \mathbb{X}}{\operatorname {arg\,min} }}\sum _{i=1}^{k}\sum_{l=1}^{\infty}|x^{(l)}_{i}-y^{(l)}|}+ u^{(l)}(x^{(l)}_{i}-y^{(l)})
	\end{equation*}
	and since $u^{(l)}(x^{(l)}_{i}-y^{(l)})<|x^{(l)}_{i}-y^{(l)}|$, we have that the geometric quantile is given by the component-wise quantile, \textit{i.e.}~$\mathbf{y}^*=(y^{*(l)})_{l\in\mathbb{N}}$ with $y^{*(l)}$ defined by 
	\begin{equation*}
		{\displaystyle {\underset {y^{(l)}\in \mathbb{R}}{\operatorname {arg\,min} }}\sum _{i=1}^{k}|x^{(l)}_{i}-y^{(l)}|}+ u^{(l)}(x^{(l)}_{i}-y^{(l)}).
	\end{equation*}
	Thus, we have existence of the geometric quantile. By Lemma \ref{lem-uniq-defined} uniqueness follows if and only if $|u^{(l)}|\notin\{1-\frac{2j}{k},j=1,...,\lfloor\frac{k}{2}\rfloor\}$, $\forall l\in\mathbb{N}$. Given any $\mathbf{x}_1,...,\mathbf{x}_k$ and $\mathbf{z}_1,...,\mathbf{z}_k$ in $\ell_1$, by Lemma \ref{lem-con-quant} we obtain
	\begin{equation*}
		\sum_{l=1}^{\infty}|Q^{(l)}_{\mathbf{u}}(\mathbf{x}_1,...,\mathbf{x}_k)-Q^{(l)}_{\mathbf{u}}(\mathbf{z}_1,...,\mathbf{z}_k)|\leq \sum_{l=1}^{\infty}\sum_{i=1}^{k}|x_{i}^{(l)}-z_{i}^{(l)}|
	\end{equation*}
	\begin{equation*}
		=\sum_{i=1}^{k}\sum_{l=1}^{\infty}|x_{i}^{(l)}-z_{i}^{(l)}|=\sum_{i=1}^{k}\|\mathbf{x}_{i}-\mathbf{z}_{i}\|_{\ell_1}.
	\end{equation*}
\end{proof}

	%
	

	\bibliographystyle{chicago}
	\bibliography{Quantiles-JRSSB}
	
	
\end{document}